\documentclass[11pt]{article}

\usepackage{amsfonts}               
\usepackage{amscd}                  
\usepackage{amsmath}                
\usepackage{amsthm}                 
\usepackage{amssymb}                
\usepackage{graphicx}               
\usepackage{lmodern}                
\usepackage{makeidx}                
\usepackage{hyperref}               
\usepackage{mathrsfs}
\usepackage[nameinlink]{cleveref}   

\usepackage{layout}                 
\usepackage{url}                    
\usepackage{verbatim}               

\usepackage{mathtools}
\usepackage{bm}
\usepackage{bbm}
\usepackage[ruled,vlined]{algorithm2e}
\newlength\mylen

\Crefname{algocf}{Algorithm}{Algorithms}
\usepackage{enumitem}
\usepackage{listings}
\usepackage{booktabs}
\usepackage{subcaption}
\usepackage{multirow}
\usepackage{cases} 
\usepackage[margin=1in]{geometry}

\usepackage{bibentry} 

\hypersetup{hypertexnames=false,colorlinks=true,allcolors=black}

\newtheorem{theorem}{Theorem}                  
\newtheorem{corollary}[theorem]{Corollary}
\newtheorem{lemma}[theorem]{Lemma}
\newtheorem{proposition}[theorem]{Proposition}
\newtheorem{example}[theorem]{Example}
\newtheorem{assumption}[theorem]{Assumption}

\theoremstyle{definition}
\newtheorem{definition}[theorem]{Definition}

\theoremstyle{remark}
\newtheorem{remark}[theorem]{Remark}

\crefformat{equation}{#2(#1)#3}

\newcommand{\bR}{\mathbb{R}}
\newcommand{\bE}{\mathbb{E}}

\newcommand{\bN}{\mathbb{N}}

\newcommand{\cA}{\mathcal{A}}
\newcommand{\cB}{\mathcal{B}}
\newcommand{\cC}{\mathcal{C}}

\newcommand{\cE}{\mathcal{E}}
\newcommand{\cF}{\mathcal{F}}

\newcommand{\cH}{\mathcal{H}}

\newcommand{\cL}{\mathcal{L}}

\newcommand{\cN}{\mathcal{N}}
\newcommand{\cO}{\mathcal{O}}
\newcommand{\cP}{\mathcal{P}}

\newcommand{\cU}{\mathcal{U}}

\newcommand{\cX}{\mathcal{X}}
\newcommand{\cY}{\mathcal{Y}}

\newcommand{\linner}{\left\langle}
\newcommand{\rinner}{\right\rangle}

\newcommand{\tr}{\mathrm{Tr}}

\newcommand{\HS}{\mathrm{HS}}

\newcommand{\dimx}{d_{\mathcal{X}}}
\newcommand{\dimy}{d_{\mathcal{Y}}}

\newcommand{\Cx}{C_{\mathcal{X}}}
\newcommand{\Cy}{C_{\mathcal{Y}}}
\newcommand{\Hx}{H_{\mathcal{X}}}
\newcommand{\Hy}{H_{\mathcal{Y}}}

\newcommand{\rx}{r}
\newcommand{\ry}{r}
\newcommand{\Vrx}{V_{r}}
\newcommand{\Ury}{U_{r}}

\newcommand{\Pry}{P_{r}}
\newcommand{\Qrx}{Q_{r}}

\newcommand{\KLE}{\mathrm{PCA}}
\newcommand{\POD}{\mathrm{PCA}}
\newcommand{\JTJ}{\mathrm{DIS}}
\newcommand{\JJT}{\mathrm{DIS}}

\newcommand{\XC}{\mathcal{E}}

\newcommand{\Pinv}[1]{{(#1)}^{\!\dagger}}
\newcommand{\pinv}[1]{{#1}^{\!\dagger}}

\newcommand{\lambdahat}{\widehat{\lambda}}

\newcommand{\muhat}{\widehat{\mu}}
\newcommand{\uhat}{\widehat{u}}
\newcommand{\vhat}{\widehat{v}}
\newcommand{\Uryhat}{\widehat{U}_{r}}
\newcommand{\Vrxhat}{\widehat{V}_{r}}
\newcommand{\Pryhat}{\widehat{P}_{r}}
\newcommand{\Qrxhat}{\widehat{Q}_{r}}

\newcommand{\Cyhat}{\widehat{C}_{\mathcal{Y}}}
\newcommand{\Cybar}{\bar{C}_{\mathcal{Y}}}
\newcommand{\Hxhat}{\widehat{H}_{\mathcal{X}}}
\newcommand{\Hyhat}{\widehat{H}_{\mathcal{Y}}}

\newcommand{\half}{1/2}

\newcommand{\encoder}{\mathfrak{E}}
\newcommand{\decoder}{\mathfrak{D}}

\newcommand{\derivXC}{D_{\XC}}
\newcommand{\deriv}{D}

\newcommand{\fun}{F}
\newcommand{\gfun}{G}
\newcommand{\fbar}{\bar{F}}
\newcommand{\ftilde}{\widetilde{F}}
\newcommand{\fhat}{\widehat{F}}

\newcommand{\gelu}{\mathrm{GeLU}}
\newcommand{\silu}{\mathrm{SiLU}}

\newcommand{\din}{d_\mathrm{in}}
\newcommand{\dout}{d_\mathrm{out}}

\newcommand{\beq}{\begin{equation}}
\newcommand{\eeq}{\end{equation}}
\newcommand{\bcases}{\begin{subnumcases}}
\newcommand{\ecases}{\end{subnumcases}}

\newcommand{\supp}{\mathrm{supp}\,}

\newcommand{\id}{\mathrm{Id}}
\newcommand{\nelem}{N_{\mathrm{el}}}
\usepackage{parskip}
\usepackage{bookmark}
\usepackage{authblk}

\author[1,*]{Dingcheng Luo}
\author[1]{Thomas O'Leary-Roseberry}

\author[2]{Peng Chen}

\author[1,3]{Omar Ghattas}

\affil[1]{Oden Institute for Computational Engineering and Sciences, The University of Texas at Austin, Austin,~TX~78712,~USA}
\affil[2]{School of Computational Science and Engineering, Georgia Institute of Technology, Atlanta,~GA~30332,~USA}
\affil[3]{Walker Department of Mechanical Engineering, The University of Texas at Austin, Austin,~TX~78712,~USA}

\title{Dimension reduction for derivative-informed operator learning: \\ An analysis of approximation errors}

\begin{document} 
\maketitle

\vspace{1em}
\begin{abstract}
    We study the derivative-informed learning of nonlinear operators between infinite-dimensional separable Hilbert spaces by neural networks. 
Such operators can arise from the solution of partial differential equations (PDEs), and are used in many simulation-based outer-loop tasks in science and engineering, such as PDE-constrained optimization, Bayesian inverse problems, and optimal experimental design. 
In these settings, the neural network approximations can be used as surrogate models to accelerate the solution of the outer-loop tasks. 
However, since outer-loop tasks in infinite dimensions often require knowledge of the underlying geometry, the approximation accuracy of the operator's derivatives can also significantly impact the performance of the surrogate model. 
Motivated by this, we analyze the approximation errors of neural operators in Sobolev norms over infinite-dimensional Gaussian input measures. 
We focus on the reduced basis neural operator (RBNO), which uses linear encoders and decoders defined on dominant input/output subspaces spanned by reduced sets of orthonormal bases. 
To this end, we study two methods for generating the bases; principal component analysis (PCA) and derivative-informed subspaces (DIS), which use dominant eigenvectors of the covariance of the data or the derivatives as the reduced bases, respectively.
We then derive bounds for errors arising from both the dimension reduction and the latent neural network approximation, including the sampling errors associated with the empirical estimation of the PCA/DIS. 
Our analysis is validated on numerical experiments with elliptic PDEs, where our results show that bases informed by the map (i.e., DIS or output PCA) yield accurate reconstructions and generalization errors for both the operator and its derivatives, while input PCA may underperform unless ranks and training sample sizes are sufficiently large.

\end{abstract}


\vspace{3em}

\let\thefootnote\relax\footnotetext{*Corresponding author: dc.luo@utexas.edu}

\pagebreak

\tableofcontents
\section{Introduction}\label{sec:dibasis_introduction}

In this work, we consider the task of learning neural network representations of nonlinear operators between infinite-dimensional spaces in terms of both their outputs and their derivatives.
Such operators arise, for example, as the solution operators of partial differential equations (PDEs),
which model many physical systems and play important roles in scientific discovery and engineering design.
For complex PDEs, their solution operators can only be accessed using numerical methods with significant computational costs, which often renders their use in simulation-intensive outer loop tasks---such as simulation-based design, inference, and uncertainty quantification---extremely computationally expensive, if not intractable.
To this end, neural network approximations of the expensive-to-solve PDEs can be effective surrogate models 
that have the potential to significantly accelerate the aforementioned outer-loop tasks.

Over the past years, a vast selection of neural network architectures for operator learning has been proposed. Well-known examples include 
those using reduced bases representations (PCANet, POD-DL-ROM) \cite{HesthavenUbbiali18, BhattacharyaHosseiniKovachkiEtAl21, FrescaManzoni22},
DeepONets \cite{LuJinPangEtAl21,LuMengCaiEtAl22,JinMengLu22}, 
Fourier neural operators \cite{LiKovachkiAzizzadenesheliEtAl21,BonevKurthHundtEtAl23}
along with many more recent developments (see \cite{KovachkiLiLiuEtAl23,KovachkiLanthalerStuart24} for more comprehensive discussions). 
Typically, these neural network representations are trained in a conventional, supervised manner, 
where the PDE's input parameters and solutions are used as input-output pairs 
to define a loss function based on the misfit between the neural network predictions and the true outputs.


In addition, it has been demonstrated that, when available, 
derivatives of the underlying input-output maps (or approximations thereof) can be utilized to improve the accuracy of the surrogates.
\cite{LiuBatill00,CzarneckiOsinderoJaderbergEtAl17,BouhlelHeMartins20,Tsay21,YuLuMengEtAl22,OLearyRoseberryChenVillaEtAl24}.
In particular, derivatives of the input-output maps can be used as training data (in addition to the input-output pairs used in conventional supervised learning),
and be incorporated into the training loss with a Sobolev-type norm,
thereby encouraging the resulting function approximation to be accurate not only in its predicted output values, but also its derivatives with respect to the input. 
This has been considered for constructing finite-dimensional neural network approximations in works such as \cite{LiuBatill00,CzarneckiOsinderoJaderbergEtAl17,BouhlelHeMartins20,Tsay21,YuLuMengEtAl22}
where the approach has been termed \textit{gradient-enhanced neural networks} or \textit{Sobolev training}. 

In \cite{OLearyRoseberryChenVillaEtAl24}, Sobolev training is considered for the approximation of formally infinite-dimensional mappings that arise from the solution operators of (PDEs), where it is given the name \textit{derivative-informed neural operators} (DINOs). 
Additional works have shown that Sobolev training of neural operators (which we will also refer to as derivative-informed operator learning)
can lead to notable benefits when used to construct surrogates for
outer-loop tasks involving PDEs \cite{BouhlelHeMartins20, LuoOLearyRoseberryChenEtAl23, GoChen23, CaoOLearyRoseberryGhattas24, GoChen24, CaoChenBrennanEtAl24}.
For example, derivative-informed surrogates have been 
applied to PDE-constrained optimization, both deterministic and under uncertainty \cite{BouhlelHeMartins20, LuoOLearyRoseberryChenEtAl23}, 
acceleration of geometric Markov Chain Monte Carlo and measure transport algorithms for Bayesian inverse problems \cite{CaoOLearyRoseberryGhattas24, CaoChenBrennanEtAl24},
and Bayesian optimal experimental design \cite{GoChen23, GoChen24}. 
Due to the important roles played by the derivatives in these outer-loop tasks, 
the improved derivative accuracy in the surrogates provided by Sobolev training can be crucial to obtaining the observed performance gains.

Since the input and output spaces in operator learning can be arbitrarily high dimensional upon discretization, the discretized derivative (i.e., the Jacobian matrix) can be extremely expensive to compute. To address this, \cite{OLearyRoseberryChenVillaEtAl24}, along with subsequent works \cite{LuoOLearyRoseberryChenEtAl23,GoChen23,CaoOLearyRoseberryGhattas24,QiuBridgesChen24, GoChen24, CaoChenBrennanEtAl24}, have leveraged the reduced basis neural operator (RBNO) architecture for the approximation to enable scalable and efficient Sobolev training.
That is, the operator surrogate uses linear encoding and decoding of the input and output spaces given in terms of reduced bases. 
These are constructed before training the neural network, which defines only the mapping between the reduced (latent) representations of the inputs and/or outputs. 
This choice of architecture allows the data generation and training costs associated with the derivatives to scale only with the intrinsic dimensionality of the mapping 
(i.e., rank of the reduced bases), as only the action of the derivative within the reduced subspaces are needed in training.

In this approach, the selection of the reduced bases, or dimension reduction technique, plays an important role.
A common approach is the principal component analysis (PCA) on function space,
which uses the dominant eigenvectors of the covariance operators of the input and outputs as reduced bases \cite{HesthavenUbbiali18,BhattacharyaHosseiniKovachkiEtAl21,FrescaManzoni22}.
The active subspace (AS) method \cite{ConstantineDowWang14,ZahmConstantinePrieurEtAl20} has also been considered for input dimension reduction in \cite{OLearyRoseberryVillaChenEtAl22,OLearyRoseberryDuChaudhuriEtAl22}.
This approach essentially uses the dominant eigenvectors of the covariance of the derivative operator's adjoint to construct 
a subspace of the input space to which the outputs are most sensitive.

In this work, our goal is to analyze the approximation error of reduced basis architectures in the context of derivative-informed operator learning. 
Specifically, we focus on operators between (real-valued) infinite-dimensional separable Hilbert spaces $\cX$ and $\cY$, where the input space is equipped with a Gaussian measure $\gamma$. 
We are interested in studying the not only the approximation error of the operator in the usual $L^2_{\gamma}$ Bochner sense, but also the approximation error of the derivative. That is, we consider the approximations of operators the $H^1_{\gamma}$ Sobolev norm over infinite-dimensional Gaussian measures as defined in \cite{Bogachev98,Bogachev10}, which is closely related to the classical Sobolev spaces for finite-dimensional inputs.

Additionally, since the choice of dimension reduction plays an important role in reduced basis architectures, we investigate how different choices of reduced bases impact the overall approximation error. 
In addition to the PCA, we also consider a generalization of the AS method to both input and output dimension reduction, where the output dimension reduction uses the covariance of the derivative operator instead of its adjoint.
We refer to this approach as the input/output derivative-informed subspaces (DIS).
Our goal is then to obtain combined approximation error bounds in terms of both the dimension reduction and the neural network approximation, where we are additionally interested in the statistical properties of the dimension reduction errors due the use of sample-based covariance estimators in computing the PCA and DIS.

\subsection{Related works}

Many of the existing works for analyzing the approximation errors of operator surrogates focus on the approximation of the operator's output values (e.g.,\cite{LuJinPangEtAl21,KovachkiLanthalerMishra21,BhattacharyaHosseiniKovachkiEtAl21, LanthalerMishraKarniadakis22,Lanthaler23}), 
typically providing error bounds in terms an $L^2_{\gamma}$ Bochner norm or an $L^\infty$ norm over a compact set.
In this context, the works of \cite{BhattacharyaHosseiniKovachkiEtAl21,Lanthaler23} are most closely related to ours. These works establish approximation errors bounds for PCANet---a reduced basis architecture that uses PCA for dimension reduction---in which the error decomposes into dimension reduction errors and neural network approximation errors.
Additionally, we mention that the approximation error of linear encoder-decoder architectures for operator learning has also been studied in \cite{HerrmannSchwabZech24,AdcockDexterMoraga24,AdcockGriebelMaier24,ReinhardtWangZech24,AdcockDexterMoraga25}, which additionally establish approximation rates with respect to the number of neural network parameters and/or training sample size under assumptions such as holomorphy and Lipschitz continuity, utilizing specific parametrizations of the neural network based on polynomial chaos expansions.

On the other hand, we are not aware of approximation results for operator surrogates involving the derivative accuracy. 
The approximation accuracy of neural networks in finite-dimensional settings is considered in works as early as that of \cite{Hornik91}, 
which proves universal approximation results of functions in $C^k$ (the space of $k$ times continuously differentiable functions)
by single-layer neural networks using smooth, sigmoid-like activation functions,
with errors measured in either weighted Sobolev norms or $C^k$ norms over compact sets. 
In our work, we make use of this result of \cite{Hornik91} to develop our universal approximation results, providing extensions to deep neural network architectures and commonly used smooth and non-saturating activation functions such as softplus, sigmoid linear unit (SiLU), and Gaussian error linear unit (GeLU).
We note that more modern results for approximation in classical Sobolev norms are available, e.g.,\cite{GuehringKutyniokPetersen19,OpschoorSchwabZech21,FengZeng22}, 
often with precise estimates on depth and width of the neural network,
though these are typically formulated for ReLU activation functions and over bounded domains.

Our analysis of dimension reduction via the DIS also draws on the works of \cite{CuiMartinMarzoukEtAl14,ZahmConstantinePrieurEtAl20,LamZahmMarzoukEtAl20,CuiTong22,BaptistaMarzoukZahm22,ChenArnaudBaptistaEtAl24}. 
For Gaussian input measures, our input DIS coincides with the vector-valued active subspace (AS) \cite{ZahmConstantinePrieurEtAl20}, which has an established bound on the $L^2_{\gamma}$ error when used for input dimension reduction.
The results of \cite{ZahmConstantinePrieurEtAl20} have subsequently been applied to finite-dimensional neural network approximations in \cite{OLearyRoseberryVillaChenEtAl22, OLearyRoseberryDuChaudhuriEtAl22}.
Moreover, the DIS also coincides with the likelihood-informed subspace (LIS) in the context of input dimension reduction for Bayesian inverse problems \cite{CuiMartinMarzoukEtAl14, ZahmCuiLawEtAl22, CuiTong22} as well as subsequent extensions to joint input and output dimension reduction \cite{BaptistaMarzoukZahm22,ChenArnaudBaptistaEtAl24}.
Again, the aforementioned works typically considered finite-dimensional input/output spaces.
Additionally, the input DIS (LIS) is considered for an infinite-dimensional input space but a finite-dimensional output space in \cite{CaoOLearyRoseberryGhattas24,CaoChenBrennanEtAl24}.


Sampling errors associated with PCA in Hilbert spaces are well studied (e.g., \cite{BlanchardBousquetZwald06, MasRuymgaart14,ReissWahl20}). 
Analogous results are also used in the analyses of PCANet in \cite{BhattacharyaHosseiniKovachkiEtAl21,Lanthaler23}.
Additionally, \cite{LamZahmMarzoukEtAl20,CuiTong22} consider the sampling errors associated with constructing the AS/LIS matrices, though in the finite-dimensional setting.


\subsection{Contributions}

In this work, we analyze the errors associated with approximating both an operator and its derivatives by an RBNO architecture. 
Our main contribution is a universal approximation result in the $H^1_{\gamma}$ 
that decomposes the error into neural network errors, dimension reduction errors, and sampling errors.
These errors contributions are then numerically studied in the context of learning the solution operator of elliptic PDEs with Gaussian random field inputs. 

In developing our main result, we also make the following contributions:
\begin{itemize}
    \item We formulate the input/output DIS (AS/LIS) in the Hilbert space setting 
    and develop analogous results for the reconstruction errors and ridge function errors in the infinite-dimensional setting based on the theory of Sobolev classes over infinite-dimensional Gaussian measures.
    \item We follow the approaches of \cite{ReissWahl20, BhattacharyaHosseiniKovachkiEtAl21, Lanthaler23} to provide bounds for the statistical sampling errors associated with the DIS operators in the Hilbert space setting.
\end{itemize}
This allows us to unify several earlier results 
that have been developed for various choices of input and output bases in either the finite- or infinite-dimensional settings 
\cite{CuiMartinMarzoukEtAl14, ZahmConstantinePrieurEtAl20,BhattacharyaHosseiniKovachkiEtAl21, CuiTong22, BaptistaMarzoukZahm22, OLearyRoseberryVillaChenEtAl22,OLearyRoseberryDuChaudhuriEtAl22, Lanthaler23, CaoOLearyRoseberryGhattas24,ChenArnaudBaptistaEtAl24,CaoChenBrennanEtAl24}.
A comparison of the problem settings in several existing works is presented in \Cref{tab:summary_of_settings}.

\begin{table}
    \centering
    \caption{A summary of the function space settings (finite or infinite dimensions) and dimension reduction strategies (PCA or DIS) considered in a selection of relevant prior work. 
    Our work considers both PCA and DIS in the infinite-dimensional setting and includes results for the approximation error in $H^1_{\gamma}$.
    Note that our use of the term DIS additionally includes the AS and LIS strategies.
    } \label{tab:summary_of_settings}
    \begin{tabular}{c c c c c}
    \toprule
       & Input dim. & Output dim. & Input reduction & Output reduction  \\
    \midrule
    \cite{BhattacharyaHosseiniKovachkiEtAl21, Lanthaler23} & Infinite & Infinite & PCA & PCA \\ 
    \cite{CuiMartinMarzoukEtAl14,ZahmCuiLawEtAl22,CuiTong22} & Finite & Finite & DIS & - \\ 
    \cite{ZahmConstantinePrieurEtAl20,LamZahmMarzoukEtAl20} & Finite & Finite & DIS, PCA & - \\ 
    \cite{OLearyRoseberryChenVillaEtAl24,OLearyRoseberryDuChaudhuriEtAl22} & Finite & Finite & DIS, PCA & PCA \\
    \cite{BaptistaMarzoukZahm22, ChenArnaudBaptistaEtAl24} & Finite & Finite & DIS, PCA & DIS, PCA \\ 
    \cite{CaoOLearyRoseberryGhattas24, CaoChenBrennanEtAl24} & Infinite & Finite & DIS, PCA & - \\ 
    Ours & Infinite & Infinite & DIS, PCA & DIS, PCA \\
    \bottomrule
    \end{tabular}
\end{table}

The remainder of this paper is organized as follows. 
In \Cref{sec:dibasis_setup}, we introduce the mathematical setup of derivative-informed operator learning with reduced basis architectures over an infinite-dimensional Gaussian input measure,
including the definition of Sobolev classes over Gaussian measures and the formulations for PCA and DIS.
In \Cref{sec:dibasis_approximation_error}, we present our approximation error results, providing bounds for the different sources of error---dimension reduction, neural network approximation, and sampling.
For clarity, many supporting lemmas and proofs for these results are provided in the Appendices.
We then present a suite of numerical experiments in \Cref{sec:numerical_experiments}, 
where we validate our theoretical bounds while numerically comparing the effectiveness of PCA and DIS as dimension reduction strategies in derivative learning with RBNOs.
Finally, we conclude in \Cref{sec:dibasis_conclusions} with a summary of our findings and discuss potential future directions.
\section{Learning an operator and its derivatives}\label{sec:dibasis_setup}
\subsection{Mathematical setup and notation}
Let $\cX$ and $\cY$ be real separable Hilbert spaces with inner products $\linner \cdot, \cdot \rinner_{\cX}$ and $\linner \cdot, \cdot \rinner_{\cY}$. 
We will consider the case where $\cX$ is equipped with a Gaussian measure $\gamma = \cN(0, \Cx)$ defined over the Borel sigma algebra $\cB(\cX)$, 
assumed to be centered without loss of generality. 
In this setting, the covariance operator $\Cx = \bE_{x \sim \gamma}[x \otimes x]$ is a trace-class operator on $\cX$, which we assume is non-degenerate.
Additionally, let $\XC = \mathrm{Range}(\Cx^{\half})$ be the Cameron--Martin space associated with $\cX$, 
which itself is a separable Hilbert space with the inner product $\linner w_1, w_2 \rinner_{\XC} = \linner \Cx^{-\half} w_1, \Cx^{-\half} w_2 \rinner_{\cX}$.

For a pair of separable Hilbert spaces $\cH_1$ and $\cH_2$, we will use 
$\cL_{k}(\cH_1, \cH_2)$ to define the space of continuous $k$-linear operators, with $\cL(\cH_1, \cH_2) = \cL_1(\cH_1, \cH_2)$.
For $k > 1$, we will write $A(h_1, \dots, h_k) \in \cH_2$ 
for the action of $A \in \cL_{k}(\cH_1, \cH_2)$ on input directions $h_1, \dots, h_k \in \cH_1$.
Moreover, we write $A(h_1, \dots, h_{k-1}, \cdot) \in \cL(\cH_1, \cH_2)$ as the linear operator defined by leaving one of the input directions open (in this case, the $k$th input direction).
The operator norm for $A \in \cL_{k}(\cH_1, \cH_2)$ is given by 
\[ \|A\|_{\cL_{k}(\cH_1, \cH_2)} = \sup_{h_1, \dots h_k \neq 0} \frac{\|A(h_1, \dots, h_k)\|_{\cH_2}}{\|h_1\|_{\cH_1} \cdots \|h_k\|_{\cH_1}}. \]

Similarly, we will consider the space of $k$-linear Hilbert--Schmidt operators $\HS_{k}(\cH_1, \cH_2)$, 
where the inner product generating the Hilbert--Schmidt norm $\|\cdot \|_{\HS_k(\cH_1, \cH_2)}$ is given by 
\[ \linner A, B \rinner_{\HS_{k}(\cH_1, \cH_2)}
  = \sum_{i_1, \dots, i_k = 1}^{\infty} \linner A(e_{i_1}, \dots, e_{i_k}), B(e_{i_1}, \dots, e_{i_k}) \rinner_{\cH_2},
\]
with any orthonormal basis $(e_i)_{i=1}^{\infty}$ of the input space $\cH_1$.
Specifically, we write $\HS(\cH_1, \cH_2) = \HS_{1}(\cH_1, \cH_2)$ 
with inner product 
\[ 
\linner A, B \rinner_{\HS(\cH_1, \cH_2)} 
  = \sum_{i=1}^{\infty} \linner A e_i, Be_i \rinner_{\cH_2} 
  = \sum_{i=1}^{\infty} \linner e_i, A^*B e_i \rinner_{\cH_2} = \tr(A^* B) 
\]
for $A, B \in \HS(\cH_1, \cH_2)$, where $\tr(\cdot)$ denotes the trace and $A^*$ denotes the adjoint of $A$. 
Note that the Hilbert--Schmidt norm of $A$ is simply $\sqrt{\tr(A^*A)}$.

For any (real) separable Hilbert space $\cH$, let $L^p(\gamma, \cH)$ denote the space of functions from $\cX$ to $\cH$ that is integrable in the Bochner sense with finite $p$th moments, $p \geq 1$, 
such that $\|\fun\|_{L^p(\gamma, \cH)}^p = \bE_{\gamma}[\|\fun\|_{\cH}^p] < \infty$ for $\fun \in L^p(\gamma, \cH)$.
We can then introduce the Sobolev class of functions $W^{m,p}(\gamma, \cH)$, which are functions with Sobolev derivatives up to order $m$ that belong to $L^p$. 
This has a number of equivalent definitions.
Here, we describe the definition using the completion of smooth cylindrical functions, and refer to \cite{Bogachev98, Nualart06} for further details.

Let $\cF \cC^{\infty}$ denote the collection of functions $\fun : \cX \rightarrow \cY$ 
that have the representation 
\[ \fun(x) = \sum_{i=1}^{m} 
\phi_i( \linner v_1, x \rinner_{\cX}, 
\dots, \linner v_n, x \rinner_{\cX} ) u_i ,\]
for some $m, n \in \bN$, where $v_1, \dots, v_n \in \cX$, 
$u_1, \dots, u_m \in \cY$, and $\phi_1, \dots \phi_m \in C^\infty_b(\bR^n, \bR)$, where $C^{\infty}_{b}(\bR^n, \bR)$ refers to the space of smooth and bounded functions from $\bR^n$ to $\bR$.
The class $\cF \cC^{\infty}$ is referred to as the class of smooth cylindrical functions.
For $\fun \in \cF \cC^{\infty}$, the restriction of the derivatives $\deriv^k \fun$ to directions in the Cameron--Martin space, denoted by
$\derivXC^k \fun(x) = \deriv^k \fun(x) |_{\XC}$,
is Hilbert--Schmidt (i.e., $\derivXC^k \fun(x) \in \HS_{k}(\XC, \cY)$), and has bounded moments. 
We define the norm $\| \cdot \|_{W^{m,p}(\gamma, \cH)}$ as 
\begin{equation}
  \| \fun \|_{W^{m,p}(\gamma, \cH)} = \left( \sum_{k=0}^{m} \int \| \derivXC^{k} \fun(x) \|_{\HS_k(\XC, \cY)}^p \; \gamma(dx) \right)^{\!\! \frac{1}{p}}
  = \left( \sum_{k=0}^{m} \| \derivXC^{k} \fun \|_{L^p(\gamma, \cX; \HS_k(\XC, \cY))}^p \right)^{\!\!\frac{1}{p}}.
\end{equation}
The space $W^{m,p}(\gamma, \cH)$ is then defined as the completion of $\cF \cC^{\infty}$ under $\| \cdot \|_{W^{m,p}}$.
For functions $\fun \in W^{m,p}(\gamma, \cH)$, we define their Sobolev derivatives $\derivXC \fun$ 
as the unique limit of sequences $(\derivXC \fun_i)_{i=1}^{\infty}$ for $\fun_i \in \cF \cC^{\infty}$, $i \in \bN$ that converges to $\fun$ in $W^{m,p}(\gamma, \cH)$.
Notably, the Sobolev derivatives 
behave as the stochastic G\^ateaux derivatives (see \cite[Chapter 5]{Bogachev98}).
That is, for $\fun \in W^{1, p}(\gamma, \cH)$, the residual
\begin{equation}\label{eq:stochastic_gateaux_derivative}
  \frac{\fun(x + th) - \fun(x)}{t} - \derivXC \fun(x) h 
\end{equation}
converges to zero in measure $\gamma$ as $t \rightarrow 0$ for any $h \in \XC$.
Analogous results hold for higher-order derivatives. 


In this work, we focus on the case $p = 2$, 
for which we introduce the notation $H^m(\gamma, \cH) := W^{m, 2}(\gamma, \cH)$,
along with the shorthand $L^2_{\gamma} := L^2(\gamma, \cY)$, $H^m_{\gamma} := H^m(\gamma, \cY)$, and $W^{m,p}_{\gamma} := W^{m,p}(\gamma, \cY)$ specifically when the output space is $\cY$. 
Moreover, we will often consider mappings $\fun$ that are additionally twice continuously Fr\'echet differentiable, $\fun \in C^2(\cX, \cY)$,
where the first and second derivatives have bounded second moments, 
i.e., $\deriv \fun \in L^2(\gamma, \cL(\cX, \cY))$ and $\deriv ^2 \fun \in L^2(\gamma, \cL_{2}(\cX, \cY))$.
Note that Fr\'echet differentiability is a stronger notion of differentiability than the Sobolev differentiability that defines $W^{m,p}(\gamma, \cY)$. 
Since the inclusion map $\iota : \XC \rightarrow \cX$ is in $\HS(\XC, \cX)$, the restriction of an operator $A \in \cL(\cX, \cY)$, $A|_{\XC} = A \circ \iota$, is also in $\HS(\XC, \cY)$. 
In particular, this implies $\deriv \fun|_{\XC} = \derivXC \fun$ and its adjoint with respect to $(\XC, \cY)$ is given by $\deriv_{\XC} \fun(x)^*  = \Cx \deriv \fun(x)^*$, where $\deriv \fun(x)^*$ is the adjoint of $\deriv \fun(x)$ defined on $\cL(\cX, \cY)$.
Thus, these assumptions on Fr\'echet differentiability would imply that $\fun \in H^2(\gamma, \cY)$. 

For such $\fun$, we aim to analyze the approximation accuracy of operator surrogates $\widetilde{\fun} \approx \fun$, 
quantified in terms of both the function value and its derivatives. 
We will measure the function value in the $L^2_{\gamma} = L^2(\gamma, \cY)$ norm, 
and the derivative with the seminorm $| \cdot |_{H^1_{\gamma}}$ defined as,
\begin{equation}
  | \fun |_{H^1_{\gamma}}^2 
  := \| \deriv_{\XC} \fun \|_{L^2(\gamma, \HS(\XC, \cY))}^2
  = \bE_{\gamma}[ \| \deriv_{\XC} \fun \|_{\HS(\XC, \cY)}^2].
\end{equation}
That is, we consider overall approximation error estimates in the $H^1_{\gamma}$ norm,
\begin{equation}
  \| \fun \|_{H^1_{\gamma}} = \left( \| \fun \|_{L^2_{\gamma}}^2 + | \fun |_{H^1_{\gamma}}^2 \right)^{1/2}.
\end{equation}

\subsection{Operator surrogates using deep neural networks}
In this work, we consider operator surrogates based on feedforward neural networks.
For an activation function $\psi : \bR \rightarrow \bR$, 
we say $\varphi_{\psi, \theta} : \bR^{\din} \rightarrow \bR^{\dout}$ 
is a (constant width, feedforward) neural network with input size $\din$, output size $\dout$, activation function $\psi$, 
depth $d_L$, and width $d_W$, if it takes the form 
\begin{equation}\label{eq:neural_network}
  \varphi_{\psi, \theta}(x) = \mathbf{W}_{d_L} \psi(\mathbf{W}_{d_L-1} \dots W_{2} \psi(\mathbf{W}_1 x + \mathbf{b}_1)) + \mathbf{b}_{d_L},
\end{equation}
where $\theta = (\mathbf{W}_i, \mathbf{b}_i)_{i=1}^{d_L}$ are the parameters of the neural networks, 
consisting of weights $\mathbf{W}_1 \in \bR^{\din \times {d_W}}$, 
$\mathbf{W}_i \in \bR^{d_W \times d_W}$ for $i = 2, \dots, d_L - 1$, 
$\mathbf{W}_{d_L} \in \bR^{d_W \times \dout}$, 
and biases $\mathbf{b}_i \in \bR^{d_W}$, for $ i = 1, \dots, d_L -1$ 
and $\mathbf{b}_{d_L} \in \bR^{\dout}$.
Our focus will be on activation functions that are smooth and ReLU-like, 
such as the softplus, Gaussian-error linear unit (GeLU), and the Sigmoid-weighted linear unit (SiLU). 
Specifically, we will consider the following definition.
\begin{definition} We say $\psi : \bR \rightarrow \bR$ belongs to the class of smooth, 
  ReLU-like activation functions, which we denote by $\cA^{\infty}_b$, 
  if $\psi \in C^{\infty}(\bR)$ 
  and satisfies the following:
  \begin{enumerate}
    \item (Boundedness for negative inputs) There exists a constant $a_{-} \geq 0$ such that $|\psi(t)| \leq a_{-}$ for $t \leq 0$.
    \item (Non-saturation for positive inputs) The function $\psi$ is strictly increasing for $t > 0$ and non-saturating as $t \rightarrow \infty$ (i.e., $\psi(t)$ tends to $\infty$).
    \item (Boundedness of derivatives) For every $k \in \bN$, there exists a constant $a_{k} > 0 $ such that the $k$th derivative $|\psi^{(k)}(t)| \leq a_k$ for all $t \in \bR$.
  \end{enumerate}
  In particular, examples like softplus, $\gelu$ and $\silu$ belong to $\cA_b^\infty$ (see examples in \Cref{sec:dibasis_hornik}).

  With activation functions $\psi \in \cA^{\infty}_b$, we can show that deep neural networks with $d_L \geq 2$ layers are universal approximators in Sobolev norms. 
  Specifically, we provide an extension to the result of \cite{Hornik91} for functions on $\bR^{d}$ over bounded measures. 
  Given a bounded measure $\nu$ over $\bR^{\din}$, 
  we follow \cite{Hornik91} to define the class of functions
  \[ C^{m,p}_{\nu}(\bR^{\din}, \bR^{\dout}) := \{f \in C^{m}(\bR^{\din}, \bR^{\dout}) : \|f\|_{W^{m,p}(\nu, \bR^{\dout})} < \infty \},
  \]
  where the $W^{m,p}(\nu, \bR^{\dout})$ norm is defined as
  \begin{equation}\label{eq:finite_sobolev_norm}
    \| f \|_{W^{m,p}(\nu, \bR^{\dout})} = \left( \int_{\bR^{d}} \|f(x)\|_2^p d\nu(x) + 
    \sum_{1 \leq |\alpha|_1 \leq m} \int_{\bR^{d}} \|\partial^{\alpha} f(x)\|_2^p d\nu(x) \right)^{\frac{1}{p}}.
  \end{equation}
  Here, $\alpha = (\alpha_1, \alpha_2, \ldots, \alpha_{\din})$, $\alpha_i \in \bN \cup \{0\}$ is a multi-index,
  $\partial^{\alpha} = \partial_{x_1}^{\alpha_1} \dots \partial_{x_{\din}}^{\alpha_{\din}}$ denotes the partial derivatives with order $\alpha_i$ in coordinate $x_i$, 
  $|\alpha| := \sum_{i=1}^{\din} \alpha_i$ denotes the total order of differentiation,
  and $\|\cdot\|_2$ refers to the standard Euclidean norm on $\bR^{\dout}$.
  Note that this definition of the $W^{m,p}(\nu, \bR^{\dout})$ norm is equivalent to our earlier definition of the $W^{m,p}(\gamma, \cY)$ norm when $\nu = \gamma$ is a Gaussian over the finite-dimensional space $\bR^{\din}$ and $\cY = \bR^{\dout}$ is also finite-dimensional.
  The original result of \cite{Hornik91} then states that single-layer neural networks with smooth and bounded activation functions $\psi \in C^{m}_b(\bR)$ 
  are universal approximators of the class of functions $C^{m,p}_{\nu}(\bR^{d}, \bR)$ in the $W^{m,p}(\nu, \bR)$ norm (we restate this in \Cref{theorem:universal_approx_hornik}).
  
  In practice, deep neural networks with non-saturating activation functions are preferred. We have the following extension for deep neural networks with activation functions $\psi \in \cA^{\infty}_b$.
  The proof of this result is presented in \Cref{sec:proof_nn_deep_extension}.
  \begin{theorem}[Deep universal approximation with smooth, ReLU like activation functions]\label{theorem:universal_approx_hornik_extended_deep}
    Suppose $g \in C_{\nu}^{m, p}(\bR^{\din}, \bR^{\dout})$ for some $m \geq 0$
    and $\nu$ is a finite measure on $\bR^{\din}$.
    Let $p \in [1, \infty)$, $\psi \in \cA^{\infty}_{b}$, and $d_L \geq 2$. 
    Then, for any $\epsilon > 0$, 
    there exists a $d_L$-layered neural network $\varphi_{\psi, \theta}$ 
    with activation function $\psi$ such that 
    \begin{equation}
      \| g - \varphi_{\psi, \theta} \|_{W^{m,p}(\nu, \bR^{\dout})} \leq \epsilon.
    \end{equation}
  \end{theorem}
\end{definition}

\subsection{Reduced basis neural operators}
Due to the high-dimensionality of the spaces $\cX$ and ${\cY}$, a naive feedforward neural network may consist of a large number of weights and biases, and can be prone to overfitting. 
Instead, more appropriate architectures have been considered for the approximation of PDE solution maps, such as the Fourier neural operator 
\cite{LiKovachkiAzizzadenesheliEtAl20}, 
DeepONet \cite{LuJinPangEtAl21}, 
and reduced basis architectures such as PCANet \cite{BhattacharyaHosseiniKovachkiEtAl21} and DIPNet \cite{OLearyRoseberryVillaChenEtAl22}. 
We will specifically focus on reduced basis architectures, 
which is a type of encoder-decoder architecture that uses a neural network to approximate the mapping $F$ within suitably-chosen $r_{\cX}$- and $r_{\cY}$-dimensional subspaces of $\cX$ and $\cY$, respectively.
This allows the neural network to be constructed as a mapping between the low-dimensional latent spaces, 
reducing the required dimensionality of its parameters (and thereby the risk of overfitting) as well as the training cost.

To this end, let us first introduce linear encoder and decoder operators 
$\encoder_{\cX} \in \cL(\cX, \bR^{\rx})$ and $\decoder_{\cX} \in \cL(\bR^{\rx}, \cX)$ on the input space 
and $\encoder_{\cY} \in \cL(\cY, \bR^{\ry})$ and $\decoder_{\cY} \in \cL(\bR^{\ry}, \cY)$ on the output space.
Note that for clarity of presentation, we will assume $r_{\cX} = r_{\cY} = r$ in our analysis. Cases with $r_{\cX} \neq r_{\cY}$ follow as corollaries.
We will then consider neural operator surrogates of the form
\[ \ftilde_{\theta}(x) = \decoder_{\cY} \circ \varphi_{\psi,\theta} \circ \encoder_{\cX} + b \] 
where $b \in \cY$ is a bias term, and 
$\varphi_{\psi,\theta} : \bR^{\rx} \rightarrow \bR^{\ry}$ is a feedforward neural network with activation function $\psi$ that maps between the latent spaces.


In the reduced basis architecture, 
the encoder and decoder operators are defined using orthonormal bases.
For the output space, we consider $\ry$ orthonormal basis vectors in $\cY$, $\Ury = (u_i)_{i=1}^{\ry}$. 
The corresponding decoder for this reduced basis is given by 
$\decoder_{\cY}(\xi) = \Ury \xi = \sum_{i=1}^{\ry} \xi_i u_i$ for $\xi = (\xi_1, \dots, \xi_{\ry}) \in \bR^{\ry}$,
where we have overloaded the notation $\Ury$ to additionally denote the decoding operator $\Ury = \decoder_{\cY} \in \cL(\bR^{d}, \cY)$.
Similarly, for the input space, we consider $\rx$ orthonormal basis vectors in the Cameron--Martin space $\XC \subset \cX$, $\Vrx = (v_i)_{i=1}^{\rx}$, and analogously write the decoder as 
$
  \decoder_{\cX}(\xi) = \Vrx \xi = \sum_{i=1}^{\rx} \xi_i v_i.
$
The corresponding encoders $\encoder_{\cX} = \pinv{\Vrx}$ and $\encoder_{\cY} = \pinv{\Ury}$ can then be defined as left inverses of the decoders.
Note that we will also refer to the size of the reduced basis $r$ as its rank.
Evidently, the approximation properties of $\ftilde_{\theta}$ depend on the choices of the bases $\Ury$ and $\Vrx$.

\begin{remark}[Comment on notation]
  In the remainder of the paper, unless otherwise specified, we will use $\lambda$, $u$, $U$, $P$ to denote eigenvalues, eigenvectors, bases/decoders, and projections on the output space $\cY$, and $\mu$, $v$, $V$, and $Q$ for that on the input space $\cX$. 
\end{remark}

\subsection{Sobolev training of reduced basis neural operators}
\label{sec:sobolev_training}
To construct operator surrogates that are accurate in the $H^1_{\gamma}$ norm, 
one can pose the supervised learning problem directly in the $H^1_{\gamma}$ norm.
That is, for a neural operator architecture parametrized by $\theta$, one solves the following minimization problem,
\begin{equation}\label{eq:full_sobolev_training}
  \min_{\theta} \| \fun  - \ftilde_{\theta} \|_{H^1_{\gamma}}^2 = \bE_{\gamma}[\| \fun - \ftilde_{\theta}\|_{\cY}^2] + \bE_{\gamma}[\| \derivXC \fun - \derivXC \ftilde_{\theta} \|_{\HS(\XC, \cY)}^2].
\end{equation}
The incorporation of derivatives in the training loss is termed Sobolev training,
and its application to operator learning leads to the so-called derivative-informed neural operators.
When using a reduced basis architecture of the form $\ftilde_{\theta} = \Ury \circ \varphi_{\psi, \theta} \circ \pinv{\Vrx} + b$, where $\Ury$ and $\Vrx$ are fixed orthonormal bases of $\cY$ and $\XC$, the $H^1_{\gamma}$ training problem can be equivalently written in the latent/reduced spaces as 
\begin{equation}\label{eq:reduced_sobolev_training}
  \min_{\theta} \bE_{x \sim \gamma} [\| \pinv{\Ury} (\fun(x) - b) - \varphi_{\psi,\theta}(\pinv{\Vrx}(x)) \|_{2}^2]
  + \bE_{x \sim \gamma}[\| \pinv{\Ury} \derivXC \fun(x) \Vrx - \varphi_{\psi,\theta}(\pinv{\Vrx}x)\|_{F}^2], 
\end{equation}
where $\|\cdot\|_2$ refers to the $\ell^2$ norm on $\bR^{\ry}$ and $\|\cdot\|_F$ refers to the matrix Frobenius norm on $\bR^{\ry \times \rx}$ (see the error decomposition in \Cref{prop:overall_error})

In practice, the training problem is solved using a finite number of training samples,
which are tuples of the form $(x_k, \fun(x_k), \pinv{\Ury} \derivXC \fun(x_k) \Vrx)$, $k = 1, \dots, N$, where $x_k$ are sampled from the input measure $\gamma$.
Notably, the derivative-informed training problem requires evaluations of both the operator's output values and its derivatives at the sampled points. 
Since the reduced basis architecture only maps between the subspaces spanned by $\Ury$ and $\Vrx$, 
the training problem involves only projection of the derivative $\pinv{\Ury} \derivXC \fun(x) \Vrx$.
For PDE-based mappings, this can be efficiently computed using forward/adjoint sensitivity methods (see for example \cite{OLearyRoseberryChenVillaEtAl24}).

Moreover, the training loss is often normalized to balance the contributions of the output and derivative terms. 
This leads to the following practical form of the empirical risk minimization problem, 
\begin{equation}
  \min_{\theta} \frac{1}{N} \left( \sum_{k=1}^{N} \frac{1}{a_{0,k}} \|\pinv{\Ury} (\fun(x_k) - b) - \varphi_{\psi,\theta}(\pinv{\Vrx}x_k)\|_{2}^2
  + \frac{1}{a_{1,k}}\|\pinv{\Ury} \derivXC \fun(x_k) \Vrx - \deriv \varphi_{\psi,\theta}(\pinv{\Vrx}x_k)\|_{F}^2 \right),
\end{equation}
where $(a_{0,k})_{k=1}^{N}$ and $(a_{1,k})_{k=1}^{N}$ are positive weights scaling the output and derivative terms, respectively.
Common choices of normalization include uniformly normalizing by the second moments of the training data,
$a_{0,k} = \sum_{k=1}^{N} \|\pinv{\Ury} (\fun(x_k)-b)\|_2^2$ and $a_{1,k} = \sum_{k=1}^{N} \|\pinv{\Ury} \derivXC \fun(x_k) \Vrx\|_F^2$,
or normalizing each data point by its norm, 
$a_{0,k} = \|\pinv{\Ury} (\fun(x_k) - b)\|_{2}^2 + \tau$ 
and $a_{1,k} = \|\pinv{\Ury} \derivXC \fun(x_k) \Vrx\|_{F}^2 + \tau$, 
where, if required, $\tau$ can be chosen to be a small constant to avoid division by zero.

\subsection{Choices of dimension reduction}
\subsubsection{Principal component analysis}
We now consider choices of reduced bases $\Vrx$ and $\Ury$ for the input and output spaces, respectively. 
One approach is to use subspaces spanned by the PCA bases. 
For the input space, the PCA coincides with the Karhunen Lo\`eve expansion (KLE) of the Gaussian measure $\gamma$. 
Since we have assume $\gamma$ is centered, we simply consider the eigenvectors of the covariance operator, 
\begin{equation}\label{eq:kle_basis_general}
  \Cx v_i^{\KLE} = \mu_i^{\KLE} v_i^{\KLE},
\end{equation}
where the eigenvalues are written in decreasing order $\mu_1^{\KLE} \geq \cdots \geq \mu_r^{\KLE} \geq \dots > 0$. 
Since $\Cx$ is self-adjoint and of trace class, its full collection of eigenvectors, $(v_i)_{i=1}^{\infty}$, yields an orthonormal basis in $\cX$. 
Alternatively, they can be normalized in the Cameron--Martin norm such that they are also an orthonormal basis in $\XC$. 
For consistency, we will take this latter approach. 

To reduce to $r$ dimensions, we take the first $\rx$ eigenvectors as a reduced basis, 
$\Vrx^{\KLE}: = (v_i^{\KLE})_{i=1}^{\rx}$.
As previously described, we overload $\Vrx^{\KLE}$ to denote the decoder $\Vrx^{\KLE} \in \cL(\bR^{\rx}, \cX)$ 
along with the corresponding encoder (its left inverse) 
$\Pinv{\Vrx^{\KLE}} \in \cL(\cX, \bR^{\rx})$
such that for $\xi = (\xi_1, \dots, \xi_{\rx}) \in \bR^{\rx}$ and $x \in \cX$, we have 
\begin{equation}
  \Vrx^{\KLE} \xi = \sum_{i=1}^{\rx} \xi_i v_i^{\KLE} ,
  \qquad
  \Pinv{\Vrx^{\KLE}} x = \left( \linner \Cx^{-1} v_1^{\KLE}, x \rinner_{\cX}, 
  \dots, \linner \Cx^{-1} v_{\rx}^{\KLE}, x \rinner_{\cX} \right).
\end{equation}
In particular, when restricted to elements in the Cameron--Martin space $h \in \XC$, we have
\[
  \Pinv{\Vrx^{\KLE}} h = \left( \linner  v_1^{\KLE}, h \rinner_{\XC}, 
  \dots, \linner v_{\rx}^{\KLE}, h \rinner_{\XC} \right).
\]
Thus, 
\begin{equation}
  \Qrx^{\KLE} := \Vrx^{\KLE} \Pinv{\Vrx^{\KLE}}  =\sum_{i=1}^{\rx} v_i^{\KLE} \otimes \Cx^{-1} v_i^{\KLE} = \sum_{i=1}^{\rx} v_i^\KLE \otimes v_i^{\KLE} / \mu_i^{\KLE}
\end{equation} 
defines an orthogonal projection on $\cX$, 
and its restriction to $\XC$, $\Qrx^{\KLE}|_{\XC} = \sum_{i=1}^{\rx} v_i^{\KLE} \otimes_{\XC} v_i^{\KLE}$, 
is also an orthogonal projection on $\XC$.
Here, we have used the notation $v \otimes v = v \linner v, \cdot \rinner_{\cX}$ and $v \otimes_{\XC} v = v \linner v, \cdot \rinner_{\XC}$. 

In this work, we will assume that the covariance operator $\Cx$ is known and can be accessed. 
This is an appropriate assumption in the setting of operator learning for surrogate modeling, 
where the training distribution $\gamma = \cN(0, \Cx)$ is often known a priori,
and the covariance operator $\Cx$ is used to draw samples $x \sim \gamma$ to generate the training data.

Analogously, one can consider a PCA of the output space ${\cY}$.
In this work, we will focus on the PCA obtained with the centered covariance 
\begin{equation}
  \Cy = \bE_{\gamma}[(\fun - \fbar) \otimes (\fun - \fbar)],
\end{equation}
which leads to favorable properties for surrogate modeling compared to the uncentered covariance. 
We discuss this fact in \Cref{sec:constraint_property}.
The output PCA basis is then obtained by solving the eigenvalue problem 
\begin{equation}\label{eq:pod_basis_general}
  \Cy u_i^{\POD} = \lambda_i^{\POD} u_i^{\POD}, 
\end{equation}
with $\lambda_1^{\POD} \geq \cdots \geq \lambda_{\ry}^{\POD} \geq 0$.
We can then take the dominant eigenvectors, $\Ury^{\POD} := (u_i^{\POD})_{i = 1}^{\ry}$, as the output basis, such that the decoding and encoding operators are 
\begin{equation}
\Ury^{\POD} \xi = \sum_{i=1}^{\ry} \xi_i u_i^{\POD},
\qquad 
\Pinv{\Ury^{\POD}} y = \left( \linner u_1^{\POD}, y \rinner_{\cY}, \dots, \linner u_{\ry}^{\POD}, y \rinner_{\cY} \right).
\end{equation}
Unlike the covariance of the inputs (which is assumed to be known), this covariance is generally not known, 
and in practice, is typically approximated empirically from samples. That is, 
\begin{equation}
  \Cyhat = \frac{1}{N}\sum_{k=1}^{N} (\fun(x_k) - \fhat) \otimes (\fun(x_i) - {\fhat}) = \frac{1}{N}\sum_{k = 1}^{N} \fun(x_k) \otimes \fun(x_k) - \fhat \otimes \fhat,
\end{equation}
where $\fhat = \frac{1}{N} \sum_{k=1}^{N} \fun(x_k)$ is an estimator of the mean using $N$ i.i.d.~samples $(x_k)_{k=1}^{N}$, $x_k \sim \gamma$. 
We note that this covariance estimator is biased, where the bias decreases with sample size.
We will write the corresponding eigenvalue problem using the covariance estimator as 
\begin{equation}\label{eq:pod_basis_estimator}
  \Cyhat \uhat^{\POD}_i = \lambdahat^{\POD}_i \uhat^{\POD}_i, 
\end{equation}
again with $\lambdahat_1^{\POD} \geq \cdots \geq \lambdahat_{\ry}^{\POD} \geq 0$.
Thus, we have $\Uryhat^{\POD} := (\uhat_i^{\POD})_{i=1}^{\ry}$ and 
\begin{equation}
  \Uryhat^{\POD} \xi = \sum_{i=1}^{\ry} \xi_i \uhat_i^{\POD},
  \qquad 
  \Pinv{\Uryhat^{\POD}} y = \left( \linner \uhat_1^{\POD}, y \rinner_{\cY}, \dots, \linner \uhat_{\ry}^{\POD}, y \rinner_{\cY} \right),
\end{equation}
as the decoder and encoder corresponding to the empirical output PCA basis, defined in the same manner as before. 
The projection operators
\begin{equation}
  \Pry^{\POD} := \Ury^{\POD} \Pinv{\Ury^{\POD}} 
  \quad \text{and}  \quad
  \Pryhat^{\POD} := \Uryhat^{\POD} \Pinv{\Uryhat^{\POD}}
\end{equation} 
are both orthogonal projections on $\cY$. 

The input and output PCA bases minimize the mean squared reconstruction errors of the inputs $x$ and outputs $\fun(x)$ measured in 
$\| \cdot \|_{\cX}$ and $\| \cdot \|_{\cY}$, respectively. We will state this result in \Cref{sec:dibasis_approximation_error}.
However, neither the input nor the output PCA bases explicitly capture the sensitivity of the map $\fun$.
In particular, the input PCA basis is independent of $\fun$ and hence does not account for the dependence of the outputs with respect to the inputs. 
On the other hand, the output PCA basis implicitly captures the dependence of the output with respect to the input, as it is constructed using the dominant features of $\fun(x)$. 
As we will see, this has implications for the approximation errors of the derivatives.

\subsubsection{Derivative-informed subspaces}
We also consider an alternative set of reduced bases that are constructed directly using the derivatives of the map $\fun$. 
To this end, we define operators $\Hx : \XC \rightarrow \XC$ and $\Hy : \cY \rightarrow \cY$, where
\begin{align}
  \Hx &:= \bE_{x \sim \nu}[ \derivXC \fun(x)^* \derivXC \fun(x)] = \Cx \bE_{x \sim \nu}[ \deriv \fun(x)^* \deriv \fun(x)], \label{eq:input_dis} \\
  \Hy &:= \bE_{x \sim \nu}[ \derivXC \fun(x) \derivXC \fun(x)^*] = \bE_{x \sim \nu}[ \deriv \fun(x) \Cx \deriv \fun(x)^*], \label{eq:output_dis}
\end{align}
where the latter equalities in \eqref{eq:input_dis} and \eqref{eq:output_dis} hold
if we assume $\fun \in C^1(\cX, \cY)$ with bounded second moments on the derivatives, i.e., $\deriv \fun \in L^2(\gamma, \cL(\cX,\cY))$, such that 
$\deriv \fun^* \deriv \fun \in L^1(\gamma, \cL(\cX, \cX))$.
More generally, when $\fun \in H^1_{\gamma}$, the sensitivity operators $\Hx : \XC \rightarrow \XC$ and $\Hy : \cY \rightarrow \cY$ are both self-adjoint, positive, trace-class operators.
Thus, their eigenvectors form orthonormal bases of $\XC$ and $\cY$. 
The reduced bases are then given by the dominant eigenvectors of $\Hx$ and $\Hy$. 
In this work, we will refer to $\Hx$ as the input sensitivity operator and $\Hy$ as the output sensitivity operator,
and refer to the subspaces generated by this approach as the input/output derivative-informed subspaces.

For the input, we consider the eigenvalue problem
\begin{equation}\label{eq:input_sensitivity_basis}
  \Hx v_i^{\JTJ} = \mu_{i}^{\JTJ} v_i^{\JTJ}, 
\end{equation}
with $\mu_{1}^{\JTJ} \geq \mu_2^{\JTJ} \geq \cdots \geq 0$, 
and take the first $\rx$ eigenvectors $\Vrx^{\JTJ} := (v_i^{\JTJ})_{i=1}^{\rx}$ as the input basis/decoder, i.e.,
We recognize that $v_i^{\JTJ}$ are orthonormal in $\XC$, and so for any $x \in \cX$, 
we write the decoder and encoder as
\begin{equation}
\Vrx^{\JTJ} \xi = \sum_{i=1}^{r} \xi_i v_i^{\JTJ} , \qquad
\Pinv{\Vrx^{\JTJ}} x := \left(
  \linner \Cx^{-1} v_1^{\JTJ}, x \rinner_{\cX}, \dots,
  \linner \Cx^{-1} v_{\rx}^{\JTJ}, x \rinner_{\cX}
  \right),
\end{equation}
where the encoder is well-defined in $\cL(\cX, \bR^{r})$ whenever $\Cx^{-1} v_i^{\JTJ}$ is well-defined in $\cX$ (i.e., $\Cx^{-1} v_i^{\JTJ}$ are elements of $\cX$). 
We also note that the projection 
\begin{equation}
  \Qrx^{\JTJ} := \Vrx^{\JTJ} \Pinv {\Vrx^{\JTJ}}
\end{equation}
is not orthogonal in $\cX$. However, when restricted to $\XC$, we do have 
\[\Pinv{\Vrx^{\JTJ}}h = \left( \linner v_1^{\JTJ}, h\rinner_{\XC}, \dots, \linner v_{\rx}^{\JTJ}, h \rinner_{\XC} \right),\] such that 
$ \Qrx^{\JTJ}|_{\XC} = \sum_{i=1}^{\rx} v_i^{\JTJ} \otimes_{\XC} v_i^{\JTJ}$
is orthogonal in $\XC$.

\begin{remark}
    For convenience, in the remainder of this work, 
    we will assume that the input DIS bases (and their empirical counterparts) do indeed satisfy $\Cx^{-1} v_i^{\JTJ} \in \cX$. 
    This allows us to define the encoder $\Pinv{\Vrx^{\JTJ}}$ and projections $\Qrx^{\JTJ}$ as continuous linear operators over the entire input space $\cX$. 
    As mentioned above, this is true if $\fun \in C^1(\cX, \cY)$ and $\deriv \fun \in L^2(\gamma, \cL(\cX,\cY))$.
\end{remark}

The output derivative-informed subspace is found by the eigenvalue problem
\begin{equation}\label{eq:output_sensitivity_basis}
  \Hy u_i^{\JJT} = \lambda_i^{\JJT} u_i^{\JJT}, 
\end{equation}
with $\lambda_{1}^{\JJT} \geq \lambda_2^{\JJT} \geq \cdots \geq 0$, 
and take the first $\ry$ eigenvectors $\Ury^{\JJT} := (u_i^{\JJT})_{i=1}^{\ry}$ as the output basis
such that the decoder and encoder are given by
\begin{equation}
  \Ury^{\JJT} \xi = \sum_{i=1}^{\ry} \xi_i u_i^{\JJT} , 
  \qquad 
  \Pinv{\Ury^{\JJT}} y = \left( \linner u_1^{\JJT}, y \rinner_{\cY}, \dots, \linner u_{\ry}^{\JJT}, y \rinner_{\cY} \right).
  \end{equation}
This again yields an orthogonal projection on $\cY$, 
\begin{equation}
  \Pry^{\JJT} := \Ury^{\JJT} \Pinv{\Ury^{\JJT}}.
\end{equation}
As we will show, by nature of their construction, the subspaces spanned by $\Vrx^{\JTJ}$ and $\Ury^{\JJT}$ have optimal reconstruction properties for the derivative operator.

Similar to the output PCA, the operators $\Hx$ and $\Hy$ need to be approximated using i.i.d.~samples $(x_k)_{k=1}^{N}$, where $x_k \sim \gamma$. 
This leads to the estimators
\begin{align}
  \label{eq:input_sensitivity_estimator}
  \Hxhat := \frac{1}{N} \sum_{k=1}^{N} \derivXC \fun(x_k)^* \derivXC \fun(x_k) , \\
  \label{eq:output_sensitivity_estimator}
  \Hyhat := \frac{1}{N} \sum_{k=1}^{N} \derivXC \fun(x_k) \derivXC \fun(x_k)^*.
\end{align}
Their corresponding eigenvalue problems are then given by 
\begin{align}
  \label{eq:input_sensitivity_basis_estimator}
  \Hxhat \vhat_i^{\JTJ} &= \muhat_{i}^{\JTJ} \vhat_i^{\JTJ},  \\
  \label{eq:output_sensitivity_basis_estimator}
  \Hyhat \uhat_i^{\JJT} &= \lambdahat_i^{\JJT} \uhat_i^{\JJT}, 
\end{align}
with again with $\muhat_{1}^{\JTJ} \geq \muhat_{2}^{\JTJ} \geq \cdots \geq 0$
and $\lambdahat_{1}^{\JJT} \geq \lambdahat_{2}^{\JJT} \geq \cdots \geq 0$.
The resulting input and output bases are given by 
$\Vrxhat^{\JTJ} := (\vhat_{i}^{\JTJ})_{i=1}^{\rx}$ 
and 
$\Uryhat^{\JJT} := (\uhat_{i}^{\JJT})_{i=1}^{\ry}$, 
such that corresponding decoders and encoders are given by 
\begin{align}
  \Vrxhat^{\JTJ} \xi &:= \sum_{i=1}^{\rx} \xi_i \vhat_i^{\JTJ} , 
  && \Pinv{\Vrxhat^{\JTJ}} x := \left(
    \linner \Cx^{-1} \vhat_1^{\JTJ}, x \rinner_{\cX}, \dots
    \linner \Cx^{-1} \vhat_{\rx}^{\JTJ}, x \rinner_{\cX}, 
    \right),
  \\
  \Uryhat^{\JJT} \xi &:= \sum_{i=1}^{\ry} \xi_i \uhat_i^{\JJT} ,
  &&
  \Pinv{\Uryhat^{\JJT}} y := \left(
    \linner \uhat_1^{\JJT}, y \rinner_{\cY}, \dots
    \linner \uhat_{\ry}^{\JJT}, y \rinner_{\cY} 
    \right),
\end{align}
which yield $\XC$ and $\cY$ orthogonal projections 
\begin{equation}
  \Qrxhat^{\JTJ} := \Vrxhat^{\JTJ} \Pinv{\Vrxhat^{\JTJ}}, \qquad 
  \Pryhat^{\JJT} := \Uryhat^{\JJT} \Pinv{\Uryhat^{\JJT}}.
\end{equation}

The operator $\Hx$ coincides with that of the AS method for finite-dimensional vector-valued functions studied in \cite{ZahmConstantinePrieurEtAl20}.
In the context of Bayesian inverse problems with Gaussian prior and noise,
$\Hx$ and $\Hy$ are analogous to the diagnostic matrices used in the dimension reduction strategies of
\cite{CuiMartinMarzoukEtAl14,ZahmCuiLawEtAl22,BaptistaMarzoukZahm22,ChenArnaudBaptistaEtAl24}.
In this work, we will extend the previous analysis by formulating these subspaces for Gaussian measures on infinite-dimensional separable Hilbert spaces.

\subsubsection{Satisfaction of non-parametric linear constraints}\label{sec:constraint_property}
Before moving on to analyzing the approximation errors, 
we point out an additional property of interest for the output bases.
In mappings that arise from PDE solution operators, 
it is common for the output $\fun(x)$, i.e., the solution of the PDE, to satisfy linear constraints, 
$B \fun(x)  = h$, for some continuous linear operator $B \in \cL(\cY, \cH)$ mapping to another separable Hilbert space $\cH$, independent of the input $x \in \cX$.
For example, these can arise from Dirichlet boundary conditions or divergence-free conditions on the solution. 
When using the centered PCA and DIS bases for the output space,
we can ensure that the surrogates of the form $\ftilde_{\theta} = \Ury \circ \varphi_{\psi,\theta} \circ \pinv{\Vrx} + \fbar$ also satisfy the constraint, i.e., $B \ftilde_{\theta}(x) = h$.
\begin{proposition}
  Suppose $\fun \in H^1_{\gamma}$. 
  Let $\cH$ be another separable Hilbert space and 
  $B \in \cL(\cY, \cH)$ be a continuous linear operator such that 
  \begin{equation}
    B \fun(x)  = h \qquad \forall x \in \cX,
  \end{equation}
  for some fixed $h \in \cH$.
  Then, the affine space 
  $\fbar + \mathrm{span}(u_1, \dots, u_r)$ corresponding to the output PCA and DIS bases with nonzero eigenvalues
  $\lambda_1 \geq \cdots \geq \lambda_r > 0$ satisfies 
  \[B y = h \qquad \forall y \in \fbar + \mathrm{span}((u_i)_{i=1}^{r}).\]
  The same holds when $\fbar$ is replaced by the estimator $\fhat$ and 
  when the PCA or DIS bases are replaced by their estimators $(\uhat_i)_{i=1}^{r}$ with nonzero eigenvalues $\lambdahat_1 \geq \dots \geq \lambdahat_r > 0$. 
\end{proposition}
\begin{proof}
  Since $B \fun(x) = h$ for all $x$, we have $B \fbar = B \bE_{\gamma}[\fun] = \bE_{\gamma}[B \fun] = h$.
  Thus, it remains to verify that $B u_i = 0$ for the reduced bases $u_i = u_i^{\JJT}$ or $u_i = u_i^{\POD}$.
  For the centered PCA, we have $u_i^{\POD} = \Cy u_i^{\POD}/\lambda_i^{\POD}$ when $\lambda_i^{\POD} > 0$. 
  This implies
  \[ B u_i^{\POD}
    = \frac{1}{\lambda_i^{\POD}}B \bE_{\gamma}[(\fun - \fbar) \otimes (\fun - \fbar)] u_i^{\POD}
    = \frac{1}{\lambda_i^{\POD}}\bE_{\gamma}[B(\fun - \fbar) \otimes (\fun - \fbar)] u_i^{\POD}
    = 0.
  \]
  For the DIS, we begin by noting that $B\fun(x) = h$ implies 
  $ \derivXC \left(B \fun(x) \right) = B \derivXC \fun(x) = 0. $
  We then have, 
  \[
    B u_i^{\JJT} = \frac{1}{\lambda_i^{\JJT}} B \Hy u_i^{\JJT} = \frac{1}{\lambda_i^{\JJT}} \bE_{\gamma}[ B \derivXC \fun \derivXC \fun^*] u_i^{\JJT} = 0,
  \]
  which again holds for all $i \leq \ry$ with $\lambda_i^{\JJT} > 0$. 
  Results for the empirical mean $\fhat$ and bases $\uhat_i^{\POD}$ and $\uhat_i^{\JJT}$ can be shown by replacing $\bE_{\gamma}$ with the sum over the samples, where the result holds almost surely with respect to the sampling.
\end{proof}

\section{Error analysis}\label{sec:dibasis_approximation_error}
\subsection{Main results}
In this section, we analyze the approximation errors of reduced basis neural operators (RBNOs) in the Sobolev norm $H^1_{\gamma}$, where $\gamma$ is a Gaussian measure.
We will focus on the case where the neural operator is constructed as 
\[ \ftilde_{\theta} = \Ury \circ \varphi_{\psi, \theta} \circ \pinv{\Vrx} + \fhat, \]
with reduced bases $\Ury$ and $\Vrx$ that are orthonormal in $\cY$ and $\XC$, 
a latent space neural network $\varphi$, 
and a sample-based mean estimate $\fhat = \frac{1}{N}\sum_{k=1}^{N} f(x_k)$.
This represents a practical setup, where the neural operator is constructed to fit the (empirical) mean shifted data 
$\fun - \fhat$ obtained from the training data. 

We begin by presenting our two main results, \Cref{theorem:naive_universal_approximation} and \Cref{theorem:main_theorem}.
Their proofs are presented in the \Cref{sec:proof_naive_universap_approximation} and \Cref{sec:proof_main_theorem} using results developed throughout the rest of this section.
Our first main result is that the RBNO architecture is indeed universal in the Sobolev norm $H^m_{\gamma}$ given orthonormal reduced bases $\Ury$ and $\Vrx$ of $\cY$ and $\XC$, respectively.
\begin{theorem}[Generic universal approximation] \label{theorem:naive_universal_approximation}
  Suppose $\fun \in H^{m}_{\gamma}$ for any $m \geq 0$. 
  Let $(u_i)_{i=1}^{\infty}$ be an orthonormal basis of $\cY$ and $(v_i)_{i=1}^{\infty}$ be an orthonormal basis of $\XC$ such that $(\Cx^{-1}v_i)_{i=1}^{\infty}$ are well-defined in $\cX$.
  Additionally, let $\psi \in \cA_b^{\infty}$ and $d_L \geq 2$.
  Then, for any $\epsilon$, 
  there exists a rank $r \in \bN$ and a $d_L$-layered neural network 
  $\varphi_{\psi, \theta}$ with activation function $\psi$ 
  such that the neural operator approximation 
  $\ftilde_{\theta}
    = \Ury \circ \varphi_{\psi, \theta} \circ \pinv{\Vrx}$ 
  with the reduced bases $\Ury = (u_i)_{i=1}^{r}$ 
  and $\Vrx = (v_i)_{i=1}^{r}$ 
  achieves the error bound 
  \begin{equation}
    \|\fun - \ftilde_{\theta}\|_{H^m_{\gamma}} \leq \epsilon.
  \end{equation}
\end{theorem}

\Cref{theorem:naive_universal_approximation} trivially reduces to the universal approximation theorem of neural networks in finite dimensions. 
We note that it does not quantify the effectiveness of the chosen dimension reduction strategies and how they impact the approximation errors. 
To more explicitly determine the contribution of errors, 
we proceed by performing a more careful accounting of the errors arising from the dimension reduction. 

Before presenting this result, we introduce some additional assumptions on the function $\fun$
that are used to bound approximation errors for specific choices of dimension reduction. 
First, when using the input PCA, we will require the additional assumption that $\fun$ is twice continuously Fr\'echet differentiable, and derivatives of $\fun$ also have bounded second moments.
\begin{assumption}[Bounded second moments of Fr\'echet derivatives]\label{assumption:derivative_second_moments}
  Suppose $\fun \in C^2(\cX, \cY)$, $\deriv \fun \in L^2(\gamma, \cL(\cX, \cY))$, and 
  $\deriv^2\fun \in L^2(\gamma, \cL_2(\cX, \cY))$.
\end{assumption}

We will also make use of \Cref{assumption:derivative_inverse_inequality} and \Cref{assumption:hessian_inverse_inequality} 
in the analysis of output PCA and input DIS, respectively. 
Intuitively, these state that lower-order derivatives of $\fun$ have control over its higher-order derivatives.
This ensures that subspaces computed from lower-order derivatives have controlled error when approximating higher derivatives of $\fun$.
We note that measurable polynomials of degree $n$ satisfy these assumptions with $K_D = n$ and $K_H = n - 1$. 
This is made precise in \Cref{prop:hermite_derivative_and_hessian} presented in the \Cref{sec:example_polynomial_forms}.

\begin{assumption}[Derivative inverse inequality] \label{assumption:derivative_inverse_inequality}
  Suppose that for $\fun \in H^1_{\gamma}$, there exists a constant $K_D = K_D(\fun) \geq 0$ such that 
  \begin{equation}\label{eq:inverse_pod_constant}
      \linner u, \Hy u \rinner_{\cY}  
      \leq K_D \linner u, \Cy u \rinner_{\cY} 
      \quad \forall u \in \cY.
  \end{equation}
\end{assumption}
\begin{assumption}[Hessian inverse inequality]\label{assumption:hessian_inverse_inequality}
  Suppose for $\fun \in H^2_{\gamma}$, there exists a constant $K_H = K_H(\fun) \geq 0$ such that 
  \begin{equation}\label{eq:hessian_inverse_inequality}
    \bE_{x \sim \gamma}[\|\derivXC^2 \fun(x)(\cdot, v) \|_{\HS(\XC, \cY)}^2] 
      \leq K_H \bE_{\gamma}\left[\|\derivXC \fun v\|_{\cY}^2 \right] = K_H \linner v, \Hx v \rinner_{\XC}, \quad \forall v \in \XC.
  \end{equation}
  Here, $\derivXC^2 \fun(x)(\cdot, v)$ refers to the linear operator $\XC \ni h \mapsto \derivXC^2 \fun(x)(h, v) \in \cY$, such that 
  \[ \|\derivXC^2 \fun(x)(\cdot, v)\|_{\HS(\XC, \cY)}^2 = \sum_{j=1}^{\infty} \|\derivXC^2 \fun(x)(w_j, v)\|_{\cY}^2, \]
  for any orthonormal basis $(w_j)_{j=1}^{\infty}$ of $\XC$.
\end{assumption}

Finally, we also note that to ensure the sampling errors associated with the empirical PCA/DIS are bounded, we additionally require that $\fun$ has bounded fourth moments in its output and its Sobolev derivatives, i.e., $\fun \in W^{1,4}_{\gamma}$.
This is typical in the analysis of sampling errors in PCA (e.g., \cite{ReissWahl20,BhattacharyaHosseiniKovachkiEtAl21}).

We then have our main result, with the proof presented in the \Cref{sec:proof_main_theorem}.
Here, the approximation error can essentially be decomposed into (1) reconstruction errors arising from representing the outputs in reduced bases, (2) ridge function errors arising from truncating the input by the input reduced basis, and (3) neural network approximation errors. 
Moreover, when using empirical estimators to construct the PCA/DIS, the reconstruction and ridge function errors additionally have a statistical error component. 
Similar to the analyses provided by \cite{ZahmConstantinePrieurEtAl20,BhattacharyaHosseiniKovachkiEtAl21,Lanthaler23}
the reconstruction and ridge function errors can be bounded in terms of summations over trailing eigenvalues ($k > r$) of the covariance or sensitivity operators that are used to define the dimension reduction, while statistical errors have a Monte Carlo type bound related to the sample size $N$. Finally, the neural network approximation error can be bounded using the universal approximation result for weighted Sobolev norms (\Cref{theorem:universal_approx_hornik_extended_deep}).

\begin{theorem}[Detailed universal approximation]\label{theorem:main_theorem}
  Suppose $\fun \in H^2_{\gamma} \cap W^{1,4}_{\gamma}$ satisfies 
  Assumptions \ref{assumption:derivative_second_moments}--\ref{assumption:hessian_inverse_inequality}.
  Additionally, let $\psi \in \cA^{\infty}_b$ and $d_L \geq 2$.
  Then, for any $\epsilon > 0$, 
  there exist $d_L$-layered neural networks $\varphi_{\psi, \theta}$ with activation $\psi$, 
  dependent on the realization of the samples,
  such that for $\Ury = \Uryhat^{\POD}$ or $\Uryhat^{\JJT}$ and $\Vrx = \Vrx^{\KLE}$ or $\Vrxhat^{\JTJ}$,
  the neural operator
  \[ \ftilde_{\theta} = \Ury \circ \varphi_{\psi, \theta} \circ \pinv{\Vrx} + \fhat \]
  simultaneously satisfies the bounds 
  \begin{align}
     \bE_{N}[ \|\fun - \ftilde_{\theta}\|_{L_{\gamma}^2}^2 ] &\leq \epsilon_\theta 
      + K_{\cY}^{(1)} \left( \sum_{i=r+1}^{\infty} {\lambda_i} 
      + \mathscr{E}_{\cY}(N, r) \right)
      + K_{\cX}^{(1)} \left( \sum_{i=r+1}^{\infty} \mu_i 
      + \mathscr{E}_{\cX}(N, r)
      \right)
      + \mathscr{E}_{\fun}(N),
      \\ 
    \bE_{N}[ |\fun - \ftilde_{\theta}|_{H_{\gamma}^1}^2 ] &\leq \epsilon_\theta 
      + K_{\cY}^{(2)} \left( \sum_{i=r+1}^{\infty} {\lambda_i} 
      + \mathscr{E}_{\cY}(N, r) 
      \right)
      + K_{\cX}^{(2)} \left( \sum_{i=r+1}^{\infty} \mu_i 
      + \mathscr{E}_{\cX}(N, r) \right),
  \end{align}
  with sampling errors bounds of the form
  \begin{align}
    \mathscr{E}_{\fun}(N) &= \frac{\|\fun - \fbar\|_{L^2_{\gamma}}^2}{N}, \\
    \mathscr{E}_{\cX}(N, r) & = \min\left\{\sqrt{\frac{2r M_\cX}{N}}, \frac{M_{\cX}}{(\lambda_{r}-\lambda_{r+1})N} \right\}, \\
    \mathscr{E}_{\cY}(N, r) & = \min\left\{\sqrt{\frac{2r M_\cY}{N}}, \frac{M_{\cY}}{(\mu_{r}-\mu_{r+1})N} \right\},
  \end{align}
  and
  \begin{equation}
    \lambda_i = \begin{cases}
      \lambda_i^{\POD} & \text{if } \Ury = \Uryhat^{\POD} \\
      \lambda_i^{\JJT} & \text{if } \Ury = \Uryhat^{\JJT}
    \end{cases}, \quad
    \mu_i = \begin{cases}
      \mu_i^{\KLE} & \text{if } \Vrx = \Vrx^{\KLE} \\
      \mu_i^{\JTJ} & \text{if } \Vrx = \Vrxhat^{\JTJ}
    \end{cases},
  \end{equation}
  where the expectation is taken with respect to the samples $(x_k)_{k=1}^{N} \sim \gamma$ used to compute 
  $\fhat$, $\Uryhat^{\POD}$, $\Uryhat^{\JJT}$, and $\Vrxhat^{\JTJ}$.
  Here, $K_{\cX}^{(i)}$, $K_{\cY}^{(i)}$,$i= 1,2$, $M_{\cX}$, $M_{\cY}$ are non-negative constants depending on the choice of dimension reduction. 
  \end{theorem}
  \begin{remark}\label{remark:on_the_assumptions}
  Not all of the assumptions are simultaneously needed. 
  We can break down the required assumptions and the corresponding forms of the constants based on the choice of dimension reduction as follows. 
  \begin{enumerate}
    \item Empirical output PCA, $\Uryhat = \Uryhat^{\POD}$: Without additional assumptions, we have $K_{\cY}^{(1)} = 1$ and 
    $M_{\cY} = 2 \bE_{\gamma}[\|(\fun - \fbar) \otimes (\fun - \fbar) - \Cy \|_{\HS(\cY, \cY)}^2] + 2 \|\fun - \fbar\|_{L^4_{\gamma}}^4 + 6\|\fun - \fbar\|_{L^2_{\gamma}}^4.$
    \Cref{assumption:derivative_inverse_inequality} is required for $K_{\cY}^{(2)}$ to be finite, in which case $K_{\cY}^{(2)} = K_D$.
    \item Empirical output DIS, $\Uryhat = \Uryhat^{\JJT}$: This does not require additional assumptions. 
    We have $K_{\cY}^{(1)} = K_{\cY}^{(2)} = 1$, and  $M_{\cY} = \bE_{\gamma}[\|\derivXC \fun \derivXC \fun^* - \Hy \|_{\HS(\cY,\cY)}^2].$
    \item Exact input PCA, $\Vrxhat = \Vrx^{\KLE}$: 
    Since the input PCA is exact, we trivially have 
    $M_{\cX} = 0$.
    \Cref{assumption:derivative_second_moments} is required for $K_{\cX}^{(1)}$ and $K_{\cX}^{(2)}$ to be finite, in which case
    \[
      K_{\cX}^{(1)}  = 2 \bE_{\gamma}[\| \deriv \fun \|_{\cL(\cX,\cY)}^2],
      \quad
      K_{\cX}^{(2)}  = 2 \bE_{\gamma}[\|\deriv^2 \fun\|_{\cL_2(\cX,\cY)}^2]
      \bE_{\gamma}[\|x\|_{\cX}^2] + 
      \bE_{\gamma}[\| \deriv \fun \|_{\cL(\cX,\cY)}^2].
    \]
    \item Empirical input DIS, $\Vrxhat = \Vrxhat^{\JTJ}$: 
    Note that we use \Cref{assumption:derivative_second_moments} as a simple way of ensuring that $(\Cx^{-1}\vhat_i^{\JTJ})_{i=1}^{\infty}$ that are well-defined in $\cX$. 
    This can be relaxed by instead directly assuming $(\Cx^{-1}\vhat_i^{\JTJ})_{i=1}^{\infty}$ are indeed well-defined in $\cX$. 
    Without additional assumptions, we have $K_{\cX}^{(1)} = 2$ and 
    $M_{\cX} = \bE_{\gamma}[ \|\derivXC \fun^* \derivXC \fun - \Hx\|_{\HS(\XC,\XC)}^2]$.
    \Cref{assumption:derivative_inverse_inequality} is then required for $K_{\cX}^{(2)}$ to be finite, in which case $K_{\cX}^{(2)} = 2 K_H + 1$. 
  \end{enumerate}
  \end{remark}
  \begin{remark}
    We also note that factors of two show up in $K_{\cX}^{(1)}, K_{\cX}^{(2)}$ due to the application of the triangle inequality for squared quantities $\|\fun - \ftilde_{\theta}\|^2$, i.e., $\|\fun - \ftilde_{\theta}\|^2 \leq (\|\fun - \gfun\| + \|\gfun - \ftilde_{\theta}\|)^2 \leq 2\|\fun - \gfun\|^2 + 2\|\gfun - \ftilde_{\theta}\|^2$ (see the proof of \Cref{theorem:error_decomposition_with_conditional_expectation}).
    These can be eliminated if we consider directly $\|\fun - \ftilde_{\theta}\|$ instead of $\|\fun - \ftilde_{\theta}\|^2$, in which case applying the triangle inequality does not incur the additional factor.
  \end{remark}
  \begin{remark}
    The sampling error estimates, $\mathscr{E}_{\cX}(N, r)$ and $\mathscr{E}_{\cY}(N, r)$, are the minima of two different bounds,
    as is typical in the analysis of PCA sampling error (e.g., \cite{ReissWahl20}). 
    These are often termed the \textit{global} bound, which has a slow $\cO(N^{-1/2})$ convergence rate 
    and a \textit{local} bound, which has a faster $\cO(N^{-1})$ convergence rate but only holds when the eigenvalues are sufficiently well separated (i.e., $\lambda_{r} - \lambda_{r+1} > 0$ and $\mu_{r} - \mu_{r+1} > 0$).
    The constant in the local bound can become very large as $r$ increases. 
    We note that both forms of the bounds are pessimistic, and more refined bounds can be derived by using more sophisticated techniques akin to those presented in \cite{ReissWahl20} for the analysis of PCA, though these often require additional assumptions such as boundedness or sub-Gaussianity of $\fun$.
  \end{remark}

The rest of this section proceeds as follows. We begin by introducing the conditional expectation, a key tool that we use to analyze the RBNO architecture. 
We proceed to consider the reconstruction, ridge function, and statistical sampling errors in detail.

\subsection{Conditional expectation and universal approximation}
The reduced basis architectures for given bases $\Ury$ and $\Vrx$ are projections (with $\Pry$) of ridge functions---$L^2_{\gamma}$ functions
that are measurable in the restricted sigma algebra $\sigma(\pinv{\Vrx})$ (i.e., constant in the orthogonal complement directions of $\Vrx$).
In $L^2_{\gamma}$, the conditional expectation $\bE_{\gamma}[\fun | \sigma(\pinv{\Vrx})]$ of a function $\fun \in L^2_{\gamma}$
is a projection, and hence $\bE_{\gamma}[\fun | \sigma(\pinv{\Vrx})]$ 
is an optimal representation of $\fun$ in $L^2_{\gamma}$ when restricted to the $\sigma(\pinv{\Vrx})$ measurable functions \cite{Bogachev10}.
When $\gamma$ is Gaussian and $\Vrx$ are orthonormal in the Cameron--Martin space $\XC$, 
the conditional expectation has the representation (Corollary 3.5.2,~\cite{Bogachev98})
\begin{equation}\label{eq:conditional_expectation}
  \bE_{\gamma}[\fun | \sigma(\pinv{\Vrx})](x) = \bE_{z \sim \gamma}[\fun(\Qrx x + (I - \Qrx) z)].
\end{equation}
This is used to develop the theory of active subspaces in the finite-dimensional setting in \cite{ZahmConstantinePrieurEtAl20}, 
but also holds in the infinite-dimensional setting.

For $\fun \in W^{m,p}(\gamma, \cY)$, when $(v_i)_{i=1}^{\infty}$ is an orthonormal basis of $\XC$, the conditional expectations 
$\bE_{\gamma}[\fun | \sigma(\pinv{\Vrx})]$ converges to $\fun$ in $W^{m,p}(\gamma, \cY)$ as $r \rightarrow \infty$ (Theorem 5.4.5,~\cite{Bogachev98}). 
Moreover, we can show that the Sobolev derivative can be interchanged with the conditional expectation here for $\fun \in H^1_{\gamma}$. 
\begin{lemma}[Interchanging Sobolev differentiation and conditional expectation]\label{lemma:interchange_differentiation_conditional_expectation}
  Suppose $\fun \in H^1_{\gamma}$,
  and let $(v_i)_{i=1}^{\infty}$ be an orthonormal basis of $\XC$ for which $(\Cx^{-1}v_i)_{i=1}^{\infty}$ are well-defined in $\cX$.
  Then, for the reduced basis $\Vrx = (v_i)_{i=1}^{r}$ and the corresponding projection $\Qrx = \Vrx \pinv{\Vrx} = \sum_{i=1}^{r} v_i \otimes \Cx^{-1} v_i$,
  the Sobolev derivative of the conditional expectation is given by
  \[\derivXC \bE_{\gamma}[\fun | \sigma(\pinv{\Vrx})] = \bE_{\gamma}[\derivXC \fun \Qrx | \sigma(\pinv{\Vrx})].\]
\end{lemma}
The proof of this lemma is provided in \Cref{sec:interchange_differentiation_integration}.
This result implies that the conditional expectation can be simultaneously used to optimally approximate $\fun$ and its projected Sobolev derivative $\derivXC \fun \Qrx$.
Thus, in developing our approximation results, we can first approximate $\fun \in H^m_{\gamma}$ 
by its conditional expectation with outputs projected onto an orthonormal basis, 
and then approximate the projected conditional expectation by a neural network. 
To this end, the following lemma is useful, with the proof provided in \Cref{sec:proof_universal_approx_cond_exp}.

\begin{lemma}[Universal approximation of the projected conditional expectation]\label{theorem:universal_approx_cond_exp}
  Suppose $\fun \in H^m_{\gamma}$ for some $m \geq 0$. 
  Let $(u_i)_{i=1}^{\infty}$ and $(v_i)_{i=1}^{\infty}$ be orthonormal bases in $\cY$ and $\XC$, respectively, 
  for which $(\Cx^{-1}v_i)_{i=1}^{\infty}$ are well-defined in $\cX$.
  For the reduced bases $\Ury = (u_i)_{i=1}^{r}$ and $\Vrx = (v_i)_{i=1}^{r}$, 
  consider the projected conditional expectation $\Pry \bE_{\gamma}[\fun | \sigma(\pinv{\Vrx})]$ 
  where $\Pry = \Ury \pinv{\Ury}$.
  Then, given $\psi \in \cA^{\infty}_b$ and $d_L \geq 2$, for any $\epsilon > 0$, 
  there exists a $d_L$-layered neural network $\varphi_{\psi, \theta}$ with activation function $\psi$ 
  such that the neural operator $\ftilde_{\theta} = \Ury \circ \varphi_{\psi,\theta} \circ \pinv{\Vrx}$ satisfies 
  \begin{equation}
    \| \Pry \bE_{\gamma}[\fun | \sigma(\pinv{\Vrx})] - \ftilde_{\theta} \|_{H^m_{\gamma}} \leq \epsilon.
  \end{equation}
\end{lemma}

\Cref{theorem:universal_approx_cond_exp} immediately allows us to prove the generic universal approximation result of \Cref{theorem:naive_universal_approximation}.
Moreover, it allows us to focus on the approximation error committed by the optimal ridge function (the projected conditional expectation), knowing that the subsequent neural network approximation error of the projected conditional expectation can be made arbitrarily small.



\subsection{Decomposition of the approximation errors}
Recall that we are considering RBNOs of the form
$ \ftilde_{\theta} = \Ury \circ \varphi_{\psi, \theta} \circ \pinv{\Vrx} + \fhat, $
where $\Ury$ and $\Vrx$ define orthogonal projections $\Pry = \Ury \pinv{\Ury}$ and 
$\Qrx = \Vrx \pinv{\Vrx}$
on $\XC$ and $\cY$, respectively.
This form of the RBNO allows us to decompose the approximation errors in the following manner, the proof of which is presented in \Cref{sec:proof_overall_error}.
\begin{proposition}\label{prop:overall_error}
  Suppose $\fun \in H^1_{\gamma}$.
  Let $(u_i)_{i=1}^{\infty}$ and $(v_i)_{i=1}^{\infty}$ be orthonormal bases in $\cY$ and $\XC$, respectively, 
  for which $(\Cx^{-1}v_i)_{i=1}^{\infty}$ are well-defined in $\cX$.
  Additionally, let $\varphi \in C^1(\bR^{\rx}, \bR^{\ry})$.
  Then, with the reduced bases $\Ury = (u_i)_{i=1}^{r}$ and $\Vrx = (v_i)_{i=1}^{r}$, 
  the ridge function $\ftilde = \Ury \circ \varphi \circ \pinv{\Vrx} + \fhat$ 
  satisfies
  \begin{equation}\label{eq:overall_error_l2}
    \| \fun - \ftilde \|_{L^2_{\gamma}}^2 \leq 
    \bE_{\gamma} [\| \Pry(\fun - \ftilde)\|_{\cY}^2]
      + \bE_{\gamma}[\| (I - \Pry)(\fun - \fbar) \|_{\cY}^2]
      + \| (I - \Pry) (\fbar - \fhat)\|_{\cY}^2
  \end{equation}
  and 
  \begin{align}\label{eq:overall_error_h1}
    | \fun - \ftilde |_{H^1_{\gamma}}^2 &\leq 
      \bE_{\gamma}[
      \| \Pry(\derivXC \fun - \derivXC \ftilde) \Qrx \|_{\HS(\XC, \cY)}^2] \nonumber \\
      & \quad + \bE_{\gamma}[\| \derivXC \fun (I - \Qrx)\|_{\HS(\XC, \cY)}^2]
      + \bE_{\gamma}[\|(I-\Pry)\derivXC f\|_{\HS(\XC, \cY)}^2]. 
  \end{align}
  Moreover, the terms involving $\ftilde$ can alternatively be written as 
  \begin{equation}
    \bE_{\gamma}[\|\Pry(\fun - \ftilde) \|_{\cY}^2] = 
     \bE_{\gamma}[\| \pinv{\Ury} (\fun - \fhat) - \varphi \circ \pinv{\Vrx} \|_{2}^2] 
  \end{equation}
  and
  \begin{equation}
      \bE_{\gamma}[\| \Pry(\derivXC \fun - \derivXC \ftilde) \Qrx \|_{\HS(\XC, \cY)}^2] = 
      \bE_{\gamma}[\|\pinv{\Ury} \derivXC \fun \Vrx - \deriv \varphi \circ \pinv{\Vrx} \|_{F}^2]
  \end{equation}
  where $\left(\varphi \circ \pinv{\Vrx} \right)(x) = \varphi \left( \pinv{\Vrx} x \right)$, $\left( \deriv \varphi \circ \pinv{\Vrx} \right)(x) = \deriv \varphi \left( \pinv{\Vrx} x \right)$, 
  and $\|\cdot\|_{2}$ and $\|\cdot\|_{F}$ denote the standard Euclidean and Frobenius norms in finite dimensions, respectively.
\end{proposition}

Here, the dimension reduction introduces errors of two types; reconstruction errors and ridge function errors.
The reconstruction errors, $\|\Pry(\fun - \fbar)\|_{L^2_{\gamma}}$, $\|(I - \Pry) \derivXC \fun \|_{L^2(\gamma, \HS(\XC, \cY))}$, 
and $\|\derivXC \fun (I - \Qrx)\|_{L^2(\gamma, \HS(\XC, \cY))}$
arise from truncating the representations of $\fun(x) - \fbar \in \cY$ in the subspace spanned by $\Ury$, 
and $\derivXC \fun(x) \in \HS(\XC, \cY)$ 
in the subspaces spanned by $\Ury$ on the output side 
and by $\Vrx$ on the input side. 
These are independent of the choice of $\varphi$.
On the other hand, the ridge function errors $\|\Pry(\fun - \ftilde)\|_{L^2_{\gamma}}$ and 
$\bE_{\gamma}[\|\Pry (\derivXC \fun - \derivXC \ftilde) \Qrx\|_{L^2(\gamma, \HS(\XC, \cY))}^2]$
are due to the projection of the input to the function itself, and can be analyzed using the conditional expectation.
Additionally, sampling errors are present in the mean estimate $\fhat$, and also in constructing the reduced bases. 

Recall that the conditional expectation is the optimal ridge function in $L^2_{\gamma}$ whose derivative is also optimal in $L^2(\gamma, \HS(\XC,\cY))$.
As previously discussed, we may further decompose the error of the ridge function term into the optimal ridge function error committed by the conditional expectation, 
and an additional error contribution due to the approximation of the conditional expectation by a neural network.
This leads to the following error decomposition that forms the basis of the proof for the main result in \Cref{theorem:main_theorem}.

\begin{corollary} \label{theorem:error_decomposition_with_conditional_expectation}
  Suppose $\fun \in H^1_{\gamma}$. 
  Let $(u_i)_{i=1}^{\infty}$ and $(v_i)_{i=1}^{\infty}$ be orthonormal bases in $\cY$ and $\XC$, respectively, 
  for which $(\Cx^{-1}v_i)_{i=1}^{\infty}$ are well-defined in $\cX$.
  Additionally, let $\psi \in \cA^{\infty}_b$ and $d_L \geq 2$.
  Then, for any $\epsilon > 0$, 
  there exists a $d_L$-layered neural network $\varphi_{\psi, \theta}$ with activation function $\psi$
  such that simultaneously, 
  the $L^2_{\gamma}$ approximation error of the neural operator $\ftilde_{\theta} = \Ury \circ \varphi_{\psi, \theta} \circ \pinv{\Vrx} + \fhat$ satisfies
  \begin{align} \label{eq:error_decomposition_with_conditional_expectation_l2}
    \|\fun - \ftilde_{\theta} \|_{L^2_{\gamma}}^2
    \leq \epsilon_{\theta} + 2 \|\fun - \ftilde_r\|_{L^2_{\gamma}}^2
        + \|(I - \Pry)(\fun - \fbar) \|_{L^2_{\gamma}}^2 
         + \|\fbar - \fhat\|_{\cY}^2 ,
  \end{align}
  and the $H^1_{\gamma}$ seminorm approximation error satisfies
  \begin{align} \label{eq:error_decomposition_with_conditional_expectation_h1}
     |\fun - \ftilde_{\theta} |_{H^1_{\gamma}}^2
    \leq & \epsilon_{\theta} 
       + 2 \|\Pry (\derivXC \fun - \derivXC \ftilde_r )\Qrx \|_{L^2(\gamma, \HS(\XC,\cY))}^2
        \nonumber \\
      & + \|\derivXC \fun ( I - \Qrx )\|_{L^2(\gamma,\HS(\XC, \cY))}^2
        + \|(I - \Pry)\derivXC \fun \|_{L^2(\gamma,\HS(\XC, \cY))}^2.
  \end{align}
  where we have used 
  $\ftilde_{r} := \bE_{\gamma}[\fun | \sigma(\pinv{\Vrx})]$ 
  to denote the conditional expectation with respect to the sigma algebra generated by $\pinv{\Vrx}$.

\end{corollary}
The proof of this result is provided in \Cref{sec:proof_error_decomposition_with_conditional_expectation}.
We now proceed by studying the error contributions individually, and how they depend on the choice of dimension reduction. 

\subsection{Reconstruction errors}\label{sec:reconstruction_errors}
\subsubsection{PCA subspaces}

We begin by considering the reconstruction errors associated with the PCA bases. 
It is well known that the PCA bases yield optimal linear reduced representations. 
We start with the input case, where no sampling errors are involved.
\begin{proposition}\label{prop:kle_input_reconstruction}
  The projection given by the exact input PCA
  minimizes the reconstruction error amongst all rank $\rx$ orthogonal projections on $\cX$, $\cP_r(\cX)$, in the sense that,
  \begin{equation}
    \bE_{x \sim \gamma}[\|(I - \Qrx^{\KLE}) x \|_{\cX}^2] 
    = \min_{{Q}_{\rx} \in \cP_{\rx}(\cX)} \bE_{x \sim \gamma}[\|(I - {Q}_{\rx}) x\|_{\cX}^{2} ] 
    = \sum_{i = \rx + 1}^{\infty} \mu_i^{\KLE}.
  \end{equation}
\end{proposition}
\begin{proof} We reproduce the proof here for completeness. We note that for any $\Qrx \in \cP_{\rx}(\cX)$, we have
  \begin{equation}
    \bE_{x \sim \gamma}[\|(I - \Qrx) x \|_{\cX}^2] 
    = \tr((I - \Qrx) \Cx (I - \Qrx)),
  \end{equation}
  which, by Fan's theorem (\Cref{lemma:fans_theorem}) is minimized by the projection onto the leading eigenvectors of $\Cx$, i.e., $\Qrx = \Qrx^{\KLE}$.
\end{proof}

As evident in the proof, the result gives the value of the Hilbert--Schmidt norm of the 
exact PCA projection $I - \Qrx^{\KLE}$ when viewed as a mapping from $\XC \rightarrow \cX$,
\begin{equation}\label{eq:kle_hs_error}
  \|I - \Qrx^{\KLE} \|_{\HS(\XC, \cX)}^2 = \tr((I - \Qrx^{\KLE}) \Cx (I - \Qrx^{\KLE})) = \sum_{i = \rx + 1}^{\infty} \mu_i^{\KLE}.
\end{equation}

Although a rank $r$ PCA optimally reconstructs the input $x$, 
intuitively, we can see that it may not provide a good reconstruction $\derivXC \fun (I - \Qrx^{\KLE})$ of $\derivXC \fun$ 
since the PCA on the input does not incorporate any information about the mapping $\fun$. 
As an example, when $\fun$ is a map depending only on the $(r + 1)$th basis, 
a truncated input PCA of rank $r$ will always yield a projected derivative $\derivXC \fun \Qrx$ of zero.
However, when measuring the reconstruction error of $\derivXC \fun (I - \Qrx^{\KLE})$ in the Hilbert--Schmidt norm $\HS(\XC, \cY)$,
smaller weights are placed on the input directions with small variance.
This allows one to place a bound on the reconstruction error of the derivative.
\begin{proposition}\label{prop:kle_derivative_reconstruction}
  Suppose $\fun \in H^1_{\gamma} \cap C^1(\cX, \cY)$ 
  and $\deriv \fun \in L^2(\gamma, \cL(\cX, \cY))$.
  Then the derivative reconstruction error for the exact input PCA 
  is bounded by 
  \begin{equation}\label{eq:kle_derivative_bound_first}
    \bE_{\gamma}[\|\derivXC \fun (I - \Qrx^{\KLE})\|^2_{\HS(\XC, \cY)}]
    \leq \bE_{\gamma}[\|\deriv \fun\|_{\cL(\cX, \cY)}^2]\sum_{i= \rx + 1}^{\infty} \mu_{i}^{\KLE}.
  \end{equation}
  Moreover, when $\fun : \cX \rightarrow \cY$ is (globally) Lipschitz continuous, 
  i.e., 
  \begin{equation}\label{eq:kle_derivative_bound_Lipschitz}
    \|\fun(x_1) - \fun(x_2) \|_{\cY} \leq L \|x_1 - x_2\|_{\cX}, \qquad \forall x_1, x_2 \in \cX
  \end{equation}
  with Lipschitz constant $L \geq 0$, we have
  \begin{equation}
    \bE_{\gamma}[\|\derivXC \fun (I - \Qrx^{\KLE})\|^2_{\HS(\XC, \cY)}]
    \leq L^2 \sum_{i = \rx + 1}^{\infty} \mu_{i}^{\KLE}.
  \end{equation}
\end{proposition}
\begin{proof}
  The approach taken here is analogous to that for the proof of Proposition 3.1 of \cite{ZahmConstantinePrieurEtAl20}.
  We will use the fact that $I - \Qrx^{\KLE} \in \HS(\XC, \cX)$ to write 
  \[ \|\derivXC \fun (I - \Qrx^{\KLE}) \|_{\HS(\XC, \cY)} \leq \|\deriv \fun\|_{\cL(\cX, \cY)} \|I - \Qrx^{\KLE}\|_{\HS(\XC, \cX)}.\]
  Therefore,
  \[ \bE_{\gamma}[\|\derivXC \fun (I - \Qrx^{\KLE}) \|_{\HS(\XC, \cY)}^2] \leq \bE_{\gamma}[\|\deriv \fun\|_{\cL(\cX, \cY)}^2] \|I - \Qrx^{\KLE}\|_{\HS(\XC, \cX)}^2,\]
  along with \eqref{eq:kle_hs_error} gives the bound \eqref{eq:kle_derivative_bound_first}.
  Moreover, since $\fun$ is continuously differentiable, if $\fun$ is Lipschitz with constant $L$, we have $\|\deriv \fun\|_{\cL(\cX, \cY)} \leq L$,
  from which the bound \eqref{eq:kle_derivative_bound_Lipschitz} follows.
\end{proof}

Since we consider a centered (mean-shifted) PCA for the output, 
the output PCA basis optimally reconstructs the mean-shifted output $\fun - \fbar$. 
On the other hand, when sampling errors are present, 
the empirical PCA basis becomes suboptimal, where the gap in the error bound is termed the excess risk. 
As discussed \cite{ReissWahl20, BhattacharyaHosseiniKovachkiEtAl21, Lanthaler23}, 
the excess risk can be bounded using the Hilbert--Schmidt distance between $\Cy$ and its estimator $\Cyhat$. 
This is summarized in the following proposition.

\begin{proposition}\label{prop:pod_reconstruction_bound}
  Suppose $\fun \in L^2_{\gamma}$. 
  Then, the exact output PCA basis minimizes the reconstruction error amongst all rank $\ry$ orthogonal projections $\cP_{\ry}(\cY)$ in the sense that,
  \begin{equation}
    \bE_{\gamma}[\|(I - \Pry^{\POD}) (\fun - \fbar) \|_{\cY}^2] 
    = \min_{\Pry \in \cP_{\ry}(\cY)} \bE_{\gamma}[\|(I - \Pry) (\fun - \fbar)\|_{\cY}^{2} ] \nonumber
    = \sum_{i = r + 1}^{\infty} \lambda_i^{\POD},
  \end{equation}
  while the reconstruction error for the empirical output PCA is bounded by
  \begin{equation}
    \bE_{\gamma}[\|(I - \Pryhat^{\POD}) (\fun - \fbar) \|_{\cY}^2] 
    \leq \sum_{i = \ry + 1}^{\infty} \lambda_i^{\POD}  + \min \left\{ \sqrt{2 r} \|\Cy - \Cyhat \|_{\HS(\cY, \cY)}, \frac{2 \| \Cy - \Cyhat\|_{\HS(\cY,\cY)}^2}{\lambda_{r}^{\POD} - \lambda_{r+1}^{\POD}} \right\}.
  \end{equation}
\end{proposition}
\begin{proof} 
  The reconstruction error for the exact output PCA is analogous to the input case shown in \Cref{prop:kle_input_reconstruction}.
  The excess risk corresponding to the empirical PCA with sample-based estimator $\Cyhat$ is a consequence of \Cref{lemma:projection_errors}.
\end{proof}

The output PCA basis is constructed on the outputs of the map $\fun$, and hence is aware of the sensitivity of the map $\fun$ in the output space. 
That is, directions in which $\fun$ changes significantly 
are expected to be captured by the output PCA basis. 
However, it is still difficult to derive similar bounds for the derivative reconstruction error using output PCA in terms of its eigenvalues.
For example, we can see that output PCA error is such that
\begin{align*}
  & \bE_{\gamma}[\|(I - \Pry^{\POD})\derivXC \fun\|_{\HS(\XC, \cY)}^2 ]  \\
  &\qquad = \tr((I - \Pry^{\POD}) \Hy (I - \Pry^{\POD})) \\
  &\qquad = \tr((I - \Pry^{\POD}) \Cy (I - \Pry^{\POD}))
    + \tr((I - \Pry^{\POD})(\Hy - \Cy)(I - \Pry^{\POD}))  \\
  &\qquad = \sum_{i = \ry + 1}^{\infty} \lambda_i^{\POD}
    + \bE_{\gamma}[\|(I - \Pry^{\POD})\derivXC \fun\|_{\HS(\XC, \cY)}^2]
    - \bE_{\gamma}[\|(I - \Pry^{\POD})(\fun - \fbar)\|_{\HS(\XC, \cY)}^2], 
\end{align*}
where the additional error term 
\[ \bE_{\gamma}[\|(I - \Pry^{\POD})\derivXC \fun\|_{\HS(\XC, \cY)}^2]
    - \bE_{\gamma}[\|(I - \Pry^{\POD})(\fun - \fbar)\|_{\HS(\XC, \cY)}^2] \geq 0\]
due to the Poincar\'e inequality (\Cref{theorem:poincare}).

Instead, we can look towards \Cref{assumption:derivative_inverse_inequality} to provide a bound on the derivative reconstruction error in terms of the eigenvalues of the output PCA basis.
In particular, the constant $K_D$ in \Cref{assumption:derivative_inverse_inequality}
quantifies the extent to which the output PCA basis captures the variation in derivatives. 
Intuitively, for operators that are largely monotonic, the deviations $\fun - \fbar$ do capture the derivatives well, 
while this may not be the case for operators with highly oscillatory behavior.
Note that the Poincar\'e inequality again implies $K_D \geq 1$. 
As previous mentioned, for polynomials, the constant $K_D$ can be bounded by the degree of the polynomial, as shown in \Cref{prop:hermite_derivative_and_hessian} in \Cref{sec:example_polynomial_forms}.
We then have the following result.
\begin{proposition}\label{prop:pod_derivative_reconstruction_multiplicative}
  Suppose $\fun \in H^1_{\gamma}$ satisfies \Cref{assumption:derivative_inverse_inequality}. 
  Then, the derivative reconstruction error for the empirical output PCA is bounded by
  \begin{equation}\label{eq:pod_derivative_error_multiplicative}
    \bE_{\gamma}[\| (I - \Pryhat^{\POD}) \derivXC \fun \|_{\HS(\XC, \cY)}^2] 
    \leq K_{D} \left(
    \sum_{i = \ry + 1}^{\infty} \lambda_i^{\POD} + \min \left\{ \sqrt{2 r} \|\Cy - \Cyhat \|_{\HS(\cY, \cY)}, \frac{2 \| \Cy - \Cyhat\|_{\HS(\cY,\cY)}^2}{\lambda_{r}^{\POD} - \lambda_{r + 1}^{\POD}} \right\} \right).
  \end{equation}
\end{proposition}
\begin{proof}
  We recognize that 
  $\|(I - \Pryhat^{\POD}) \derivXC \fun\|_{\HS(\XC, \cY)}^2 = \sum_{i=1}^{\infty} \| \derivXC \fun^* (I - \Pryhat^{\POD}) u_i\|_{\XC}^2 $ 
  for any $(u_i)_{i=1}^{\infty}$ that is an orthonormal basis of $\cY$. 
  By extending the empirical output PCA basis, we obtain a basis $(\uhat_i^{\POD})_{i=1}^{\infty}$ of $\cY$, which we can use to evaluate the Hilbert--Schmidt norm,
  \[ 
    \|(I - \Pryhat^{\POD}) \derivXC \fun\|_{\HS(\XC, \cY)}^2 = \sum_{i=r+1}^{\infty} \| \derivXC \fun^*  \uhat_i\|_{\XC}^2 
    = \sum_{i=\ry + 1}^{\infty} \linner \derivXC \fun \derivXC \fun^* \uhat_i^{\POD}, \uhat_i^{\POD} \rinner_{\cY}.
  \]
  Taking an expectation and applying \Cref{assumption:derivative_inverse_inequality} yields
  \[
    \sum_{i=\ry + 1}^{\infty} 
    \bE_{\gamma} \left[ \linner \derivXC \fun \derivXC \fun^* \uhat_i^{\POD}, \uhat_i^{\POD} \rinner_{\cY} \right] = \sum_{i=\ry + 1}^{\infty} \linner \uhat_i^{\POD}, \Hy \uhat_i^{\POD} \rinner_{\cY} \leq K_D \sum_{i=\ry + 1}^{\infty} \linner \uhat_i^{\POD}, \Cy \uhat_i^{\POD} \rinner_{\cY}.
  \]
  Finally, the desired bound follows from \Cref{prop:pod_reconstruction_bound}.

\end{proof}

\subsubsection{Derivative-informed subspaces}
We now consider the reconstruction properties for projections defined using derivative-informed subspaces. 
Since they effectively amount to a PCA on the derivative operators, 
we have similar optimal reconstruction properties of the derivative.
The following proposition extends the results of \cite{ZahmConstantinePrieurEtAl20} and \cite{BaptistaMarzoukZahm22} 
to the infinite-dimensional setting. It also includes the effect of sample errors, 
which is considered for the input DIS (active subspace) in the finite-dimensional setting by \cite{LamZahmMarzoukEtAl20, CuiTong22}.
Again, the sampling errors are analyzed in \Cref{sec:sampling_error}. 
\begin{proposition}\label{prop:dis_derivative_reconstruction}
  Suppose $\fun \in H^1_{\gamma}$ 
  and assume that for the exact and empirical input DIS bases, $(\Cx^{-1} v_i^{\JTJ})_{i=1}^{\infty}$ and $(\Cx^{-1} \vhat_i^{\JTJ})_{i=1}^{\infty}$ are well-defined in $\cX$.
  Then, the projections onto the exact derivative-informed subspaces
  minimize the derivative reconstruction errors.
  That is,
  \begin{align}
    \bE_{\gamma}[\|\derivXC \fun(I - \Qrx^{\JTJ})\|^2_{\HS(\XC, \cY)}] &= 
    \min_{\Qrx \in \cP_{\rx}(\XC)} \bE_{\gamma}[\|\derivXC \fun(I - \Qrx)\|^2_{\HS(\XC, \cY)}] = \sum_{i = \rx + 1}^{\infty} \mu_i^{\JTJ}, \\
    \bE_{\gamma}[\|(I - \Pry^{\JJT}) \derivXC \fun\|^2_{\HS(\XC, \cY)}] &=
    \min_{\Pry \in \cP_{\ry}(\cY)} \bE_{\gamma}[\|(I - \Pry) \derivXC \fun\|^2_{\HS(\XC, \cY)}] = \sum_{i = \ry + 1}^{\infty} \lambda_i^{\JJT}.
  \end{align}
  Moreover, in the presence of sampling errors, the derivative reconstruction errors for the empirical DIS
  are bounded by 
  \begin{align}
    \bE_{\gamma}[\|\derivXC \fun(I - \Qrxhat^{\JTJ})\|^2_{\HS(\XC, \cY)}] &\leq \sum_{i = \rx + 1}^{\infty} \mu_i^{\JTJ} 
      + \min \left\{ \sqrt{2 r} \|\Hx - \Hxhat \|_{\HS(\XC, \XC)}, \frac{2 \| \Hx - \Hxhat\|_{\HS(\XC,\XC)}^2}{\mu_{r}^{\JTJ} - \mu_{r+1}^{\JTJ}} \right\}, \label{eq:jacobian_error_estimator_input} \\
    \bE_{\gamma}[\|(I - \Pryhat^{\JJT}) \derivXC \fun\|^2_{\HS(\XC, \cY)}] &\leq \sum_{i = \ry + 1}^{\infty} \lambda_i^{\JJT} 
    + \min \left\{ \sqrt{2 r} \|\Hy - \Hyhat \|_{\HS(\cY, \cY)}, \frac{2 \| \Hy - \Hyhat\|_{\HS(\cY,\cY)}^2}{\lambda_{r}^{\JJT} - \lambda_{r+1}^{\JJT}} \right\}.   \label{eq:jacobian_error_estimator_output}
  \end{align}
\end{proposition}
\begin{proof}
  We start with the input side. 
  For any projection $\Qrx \in \cP_{r}(\XC)$, 
  \[
    \bE_{\gamma}[ \| \derivXC \fun ( I - \Qrx) \|_{\HS(\XC, \cY)}^2] 
    = \bE_{\gamma}[ \tr((I -  \Qrx) \derivXC^* \fun \derivXC \fun (I -  \Qrx))]
    = \tr((I -  \Qrx) \Hx (I -  \Qrx)),
  \]
  where the interchange of trace with expectation is valid,
  since each term in the infinite sum for the trace is non-negative, and $\derivXC^* \fun \derivXC \fun$ is assumed to be integrable.
  By \Cref{lemma:trace_minimization}, the trace is minimized by $\Qrx^{\JTJ}$ with value
  \[
    \bE_{\gamma}[ \| \derivXC \fun ( I - \Qrx^{\JTJ}) \|_{\HS(\XC, \cY)}^2]  = \sum_{i = \rx + 1}^{\infty} \lambda_i^{\JTJ}. 
  \]
  The case for the output projection is analogous, since for $\Pry \in \cP_{\ry}(\cY)$, 
  \[
    \bE_{\gamma}[ \|(I - \Pry)\derivXC \fun \|_{\HS(\XC, \cY)}^2]
    = \bE_{\gamma}[\tr((I - \Pry) \derivXC \fun \derivXC \fun^* (I - \Pry))]
    = \tr((I - \Pry) \Hy (I - \Pry))
  \]
  In the presence of sampling errors, \Cref{lemma:projection_errors} yields the desired bounds for both the input and output subspaces.
\end{proof}

We can use the Poincar\'e inequality (\Cref{theorem:poincare}) to obtain a reconstruction error bound of the output itself using the output DIS, analogous to the output PCA.
We note that this bound is, in general, not optimal due to the use of the Poincar\'e inequality.
In fact, the trailing eigenvalue sum for DIS is at most as small as the trailing output PCA eigenvalues.

\begin{proposition}\label{prop:output_dis_function_error}
  Suppose $\fun \in H^1_{\gamma}$
  Then, the reconstruction error for the empirical output DIS is bounded by
  \begin{equation}
      \bE_{\gamma}[\|(I - \Pryhat^{\JJT})(\fun - \fbar)\|_{\cY}^2] 
        \leq \sum_{i = \ry + 1}^{\infty} \lambda_i^{\JJT} 
        + \min \left\{ \sqrt{2 r} \|\Hy - \Hyhat \|_{\HS(\cY, \cY)}, \frac{2 \| \Hy - \Hyhat\|_{\HS(\cY,\cY)}^2}{\lambda_{r}^{\JJT} - \lambda_{r+1}^{\JJT}} \right\}.
  \end{equation}
\end{proposition}
\begin{proof}
  Since for any $\Pry \in \cP_r(\cY)$, the function $(I - \Pry)\fun$ has derivative $(I - \Pry) \derivXC \fun$,
  and mean $(I - \Pry) \fbar$, 
  we can apply the Poincar\'e equality (\Cref{theorem:poincare}) to obtain 
  \[ 
    \bE_{\gamma}[\|(I - \Pry)(\fun - \fbar)\|_{\cY}^2] \leq \bE_{\gamma}[ \|( I - \Pry) \derivXC \fun\|_{\HS(\XC, \cY)}^2].
  \]
  The bound \eqref{eq:jacobian_error_estimator_input} yields the desired result.
  Note that if we instead use the exact DIS, $\Pry = \Pry^{\JJT}$, the bound is simply 
  \[ 
    \bE_{\gamma}[\|(I - \Pry^{\JJT})(\fun - \fbar)\|_{\cY}^2] \leq \sum_{i = \ry + 1}^{\infty} \lambda_i^{\JJT}.
  \]
\end{proof}

\subsection{Ridge function errors}
\subsubsection{Ridge function errors of the output values}
We now consider the ridge function errors that arise due to the projection of the input. 
This is studied by \cite{ZahmConstantinePrieurEtAl20} for the approximation error in $L^2_{\gamma}$ for finite-dimensional input and output spaces,
and subsequently used to derive estimates on DIPNet approximations \cite{OLearyRoseberryVillaChenEtAl22,OLearyRoseberryDuChaudhuriEtAl22}. 
In this section, we provide extensions to the case of infinite-dimensional inputs and outputs, 
and also provide bounds on the derivative error.

As previously discussed, given $\Vrx = (v_i)_{i=1}^{\rx}$, a set of orthonormal vectors in $\XC$ and its encoding operator
$\pinv{\Vrx} = (\linner v_1, \cdot \rinner_{\XC}, \dots, \linner v_{\rx}, \cdot \rinner_{\XC})$,
the conditional expectation $\bE_{\gamma}[\fun | \sigma(\pinv{\Vrx})]$ 
is an optimal representation of $\fun$ when restricted to the sigma algebra $\sigma(\pinv{\Vrx})$.
Our strategy is to study how the conditional expectation approximates the function, and then use a RBNO to approximate the projection of the conditional expectation onto the output subspace spanned by $\Ury$, i.e., $\Pry \bE_{\gamma}[\fun | \sigma(\pinv{\Vrx})]$.

In particular, the conditional expectation satisfies a subspace Poincar\'e inequality, which is shown by \cite{ZahmConstantinePrieurEtAl20} in the finite-dimensional context. 
We will verify that an analogous version holds in the context of separable Hilbert spaces. 
The proof follows from the Poincar\'e inequality itself, and is given in \Cref{sec:proof_subspace_poincare}.
\begin{theorem}[Subspace Poincar\'e inequality] \label{theorem:subspace_poincare}
  Suppose $\fun \in H^1_{\gamma}$.
  Let $(v_i)_{i=1}^{\infty}$ be an orthonormal basis of the Cameron--Martin space $\XC$ for which $(\Cx^{-1} v_i)_{i=1}^{\infty}$ are well-defined in $\cX$. 
  Then, with the reduced basis $\Vrx = (v_i)_{i=1}^{\rx}$ and its corresponding projection $\Qrx = \Vrx \pinv{\Vrx} = \sum_{i=1}^{\rx} v_i \otimes \Cx^{-1} v_i$, 
  the operator $\fun$ satisfies the subspace Poincar\'e inequality,
  \begin{equation}\label{eq:subspace_poincare}
    \bE_{\gamma}\left[ \|\fun - \bE_{\gamma}[\fun | \sigma(\pinv{\Vrx})] \|_{\cY}^2 \right]
    \leq 
    \bE_{\gamma} \left[ \|\derivXC \fun (I - \Qrx)\|_{\HS(\XC, \cY)}^2 \right].
  \end{equation}
\end{theorem}

Substituting the bounds from \Cref{prop:kle_derivative_reconstruction} and \Cref{prop:dis_derivative_reconstruction} on
the derivative reconstruction error, we can explicitly write out the bounds for the PCA and DIS bases. 
This result is analogous to Proposition 3.1 of \cite{ZahmConstantinePrieurEtAl20}. 
\begin{corollary}[Ridge function error in $L^2_{\gamma}$] \label{cor:ridge_function_error}
  Suppose $\fun \in H^1_{\gamma}$.
  Additionally, assume $(\Cx^{-1} v_i^{\JTJ})_{i=1}^{\infty}$ and $(\Cx^{-1} \vhat_i^{\JTJ})_{i=1}^{\infty}$ are well-defined in $\cX$ for the exact and empirical input DIS bases.
  Then,
  \begin{equation} 
    \bE_{\gamma} \left[ \| \fun - \bE_{\gamma}[\fun | \sigma(\Pinv{\Vrx^{\JTJ}})] \|_{\cY}^2 \right] \leq \sum_{i = \rx + 1}^{\infty} \mu_i^{\JTJ}
  \end{equation}
  when using the exact input DIS and
  \begin{equation} 
    \bE_{\gamma} \left[ \| \fun - \bE_{\gamma}[\fun | \sigma(\Pinv{\Vrxhat^{\JTJ}})]\|_{\cY}^2  \right]\leq 
    \sum_{i = \rx + 1}^{\infty} \mu_i^{\JTJ} 
    + \min \left\{ \sqrt{2 r} \|\Hx - \Hxhat \|_{\HS(\XC, \XC)}, \frac{2 \| \Hx - \Hxhat\|_{\HS(\XC,\XC)}^2}{\mu_{r}^{\JTJ} - \mu_{r+1}^{\JTJ}} \right\} 
  \end{equation}
  when using the empirical input DIS.

  Alternatively, if $\fun \in H^1_{\gamma} \cap C^1(\cX, \cY)$ and $D \fun \in L^2(\gamma, \cL(\cX, \cY))$, 
  then 
  \begin{equation} 
    \bE_{\gamma}\left[ \| \fun - \bE_{\gamma}[\fun | \sigma(\Pinv{\Vrx^{\KLE}})] \|_{\cY}^2 \right] \leq 
    \bE_{\gamma}\left[ \|\deriv \fun \|_{\cL(\cX, \cY)}^2 \right]\sum_{i = \rx + 1}^{\infty} \mu_i^{\KLE} 
  \end{equation}
  when using the exact input PCA.

\end{corollary}

\subsubsection{Ridge function errors of the derivatives}
We next consider the ridge function error of the derivative term $\bE_{\gamma}[\|\Pry (\derivXC \fun - \derivXC \ftilde) \Qrx\|_{\HS(\XC, \cY)}^2]$. 
Recall that $\fun \in H^2_{\gamma}$, the Sobolev derivative $\derivXC \fun$ is also a mapping in $H^1(\gamma, \HS(\XC, \cY))$.
Thus, the projected derivative $\derivXC \fun(x) \Qrx$ has an optimal ridge function representation 
\[ \bE_{\gamma}[ \derivXC \fun \Qrx | \sigma(\pinv{\Vrx}) ](x) = \bE_{z \sim \gamma}[\derivXC \fun(\Qrx x + (I - \Qrx)z) \Qrx] = \bE_{\gamma}[\derivXC \fun | \sigma(\pinv{\Vrx})](x) \Qrx, \]
which also coincides with the derivative of the conditional expectation, $\derivXC \bE_{\gamma}[\fun | \sigma(\pinv{\Vrx})]$.
Using this fact, we can apply the Subspace Poincar\'e inequality on the projected Sobolev derivative,
which bounds the derivative ridge function error by the Hessian acting on the orthogonal complement of the projection $\Qrx$.
\begin{proposition}[Subspace Poincar\'e inequality for the derivative]
  Suppose $\fun \in H^2_{\gamma}$.
  Let $(v_i)_{i=1}^{\infty}$ be an orthonormal basis of the Cameron--Martin space $\XC$ for which $(\Cx^{-1} v_i)_{i=1}^{\infty}$ are well-defined in $\cX$. 
  Then, with the reduced basis $\Vrx = (v_i)_{i=1}^{\rx}$ and its corresponding projection $\Qrx = \Vrx \pinv{\Vrx} = \sum_{i=1}^{\rx} v_i \otimes \Cx^{-1} v_i$, 
  the operator $\fun$ satisfies the inequality 
  \begin{equation}\label{eq:hessian_poincare}
    \bE_{\gamma} \left[\|\derivXC \fun - \derivXC \bE_{\gamma}[\fun | \sigma(\pinv{\Vrx})]\|_{\HS(\XC, \cY)}^2 \right] 
    \leq \sum_{i=1}^{\infty} \bE_{\gamma}[ \| \derivXC^{2} \fun(w_i, \cdot) (I - \Qrx) \|_{\HS(\XC, \cY)}^2],
  \end{equation}
  for any orthonormal basis $(w_i)_{i=1}^{\infty}$ in $\XC$, where $\derivXC^{2}\fun(w_i, \cdot)$ denotes the linear operator $h \mapsto \derivXC^{2} \fun(w_i, h)$
\end{proposition}
\begin{proof}
  \Cref{lemma:interchange_differentiation_conditional_expectation} implies that
  $ \derivXC \bE_{\gamma}[\fun | \sigma(\pinv{\Vrx})] = \bE_{\gamma}[\derivXC \fun |\sigma(\pinv{\Vrx})]\Qrx.$ 
  Therefore,  
  \begin{align*}
    \bE_{\gamma} \left[ \| \derivXC \fun \Qrx - \derivXC \bE_{\gamma}[\fun | \sigma(\pinv{\Vrx})] \|_{\HS(\XC,\cY)}^2 \right]
      &= \bE_{\gamma}\left[ \| \derivXC \fun \Qrx - \bE_{\gamma}[\derivXC \fun | \sigma(\pinv{\Vrx})] \Qrx \|_{\HS(\XC, \cY)}^2 \right] \\ 
      &\leq \bE_{\gamma}\left[ \| \derivXC \fun - \bE_{\gamma}[\derivXC \fun | \sigma(\pinv{\Vrx})] \|_{\HS(\XC, \cY)}^2 \right]. 
  \end{align*}
  where we have made use of $\|\Qrx\|_{\cL(\XC, \XC)} = 1$. 
  By our assumptions, $\derivXC \fun \in H^1(\gamma, \HS(\XC, \cY))$, so we can apply the subspace Poincar\'e inequality 
  \[ 
    \bE_{\gamma} \left[\|\derivXC \fun - \derivXC \bE_{\gamma}[\fun | \sigma(\pinv{\Vrx})]\|_{\HS(\XC, \cY)}^2 \right] 
    \leq \bE_{\gamma} [\derivXC^2 \fun(\cdot, (I - \Qrx)(\cdot)) \|_{\HS_2(\XC, \cY)}^2], 
  \]
  where $\derivXC^2 \fun(x)(\cdot, (I - \Qrx)(\cdot)) : (h_1, h_2) \mapsto \derivXC^2 \fun(x) (h_1, (I - \Qrx)h_2)$ takes values in $\HS_{2}(\XC, \cY)$.
  Moreover, we can write the $\HS_2(\XC, \cY)$ norm as 
  \[
    \|\derivXC^2 \fun(\cdot, (I - \Qrx)(\cdot)) \|_{\HS_2(\XC, \cY)}^2 = \sum_{i=1}^{\infty} \| \derivXC^2 \fun(x)(w_i, \cdot)(I - \Qrx) \|_{\HS(\XC, \cY)}^2,
  \]
  where $\derivXC^2 \fun(x)(w_i, \cdot)(I - \Qrx) : h \mapsto \derivXC^2(x)(w_i, (I - \Qrx)h)$. 
  Since the sum consists of non-negative terms, interchanging expectation with summation yields \eqref{eq:hessian_poincare}.

\end{proof}

In the case of $\Qrx$ coming from the input PCA, we can further bound the Hessian action by the trailing input PCA eigenvalues at the cost of a (potentially large) constant, 
analogous to that of the derivative reconstruction error.
\begin{proposition}[Derivative ridge function error bound with input PCA]\label{prop:ridge_derivative_bound_kle}
  Suppose $\fun \in H^2_{\gamma}$ satisfies \Cref{assumption:derivative_second_moments}.
  Then, for any orthogonal projection $\Pry \in \cP_{\ry}(\cY)$, the exact input PCA satisfies
  \begin{align}\label{eq:kle_basis_derivative_ridge_error}
    \nonumber
    & \bE_{\gamma} \left[\|\Pry (\derivXC \fun  - \derivXC \bE_{\gamma}[\fun | \sigma(\Pinv{\Vrx^{\KLE}})])\Qrx^{\KLE} \|_{\HS(\XC, \cY)}^2 \right] \\
    & \qquad \qquad \qquad\leq 
      \bE_{\gamma} \left[\|\deriv^{2}\fun\|_{\cL_2(\cX, \cY)}^2 \right] 
      \bE_{\gamma}[\|x\|_{\cX}^2] 
      \sum_{i = \rx + 1}^{\infty} \mu_i^{\KLE}.
  \end{align}
\end{proposition}
\begin{proof}
Since $\|\Pry\|_{\cL(\cY, \cY)} = 1$, we have
\[ 
  \|\Pry (\derivXC \fun  - \derivXC \bE_{\gamma}[\fun | \sigma(\Pinv{\Vrx^{\KLE}})])\Qrx^{\KLE} \|_{\HS(\XC, \cY)}^2 \leq 
  \|(\derivXC \fun  - \derivXC \bE_{\gamma}[\fun | \sigma(\Pinv{\Vrx^{\KLE}})])\Qrx^{\KLE} \|_{\HS(\XC, \cY)}^2,
\]
such that we can ignore the output projection. 
For any orthonormal basis $(w_i)_{i=1}^{\infty}$ of $\XC$, we have 
\[ \| \derivXC^2 \fun(x)(w_i, \cdot)(I - \Qrx^{\KLE})\|_{\HS(\XC,\cY)}^2 \leq \| \deriv^2 \fun(x)(w_i, \cdot)\|_{\cL(\cX, \cY)}^2 \|I - \Qrx^{\KLE} \|_{\HS(\XC, \cY)}^2, \]
where $ \|I - \Qrx^{\KLE}\|_{\HS(\XC, \cY)}^2 = \sum_{i=\rx +1}^{\infty} \mu_i^{\KLE}$ due to \eqref{eq:kle_hs_error}.
Moreover, we know 
\begin{align*} 
  \|\deriv^2 \fun(x)(w_i, \cdot)\|_{\cL(\cX,\cY)} = \sup_{\|\hat{x}\| \neq 0} \frac{\| \deriv^2\fun(x)(w_i, \hat{x}) \|_{\cY}}{\|\hat{x}\|_{\cX}} \leq \|\deriv^2 \fun(x) \|_{\cL_2(\cX, \cY)} \|w_i \|_{\cX}.
\end{align*}
Let $w_i = v_i^{\KLE}$. Due to our definition of $v_i^{\KLE}$ being orthonormal in $\XC$, we have 
\[  
\|w_i\|_{\cX}^2 = \| v_i^{\KLE} \|_{\cX}^2 = \| \Cx^{-1/2} \Cx^{1/2} v_i^{\KLE} \|_{\cX}^2 = \mu_i^{\KLE} \|v_i^{\KLE}\|_{\XC}^2 = \mu_i^{\KLE}.
\]
Furthermore, since $\sum_{i=1}^{\infty} \mu_i^{\KLE} =  \bE_{x \sim \gamma}[\|x\|^2_{\cX}]$, we can write 
\begin{align*} 
  \sum_{i=1}^{\infty} \bE_{x \sim \gamma}[\|\deriv^2\fun(x)(w_i, \cdot)&(I - \Qrx)\|_{\HS(\XC, \cY)}^2] \leq \\
  & \bE_{x \sim \gamma}[ \|x\|_{\cX}^2] \bE_{x \sim \gamma}[ \| \deriv^2 \fun(x) \|_{\cL_2(\cX, \cY)}^2] \sum_{i = \rx + 1}^{\infty} \mu_i^{\KLE}, 
\end{align*}
which concludes the proof.
\end{proof}

The case for the derivative-informed subspace is less straightforward. 
Unlike the PCA basis, the DIS basis does not guarantee an optimal reconstruction of the input $x$ (i.e., not necessarily small $\| I - \Qrx \|_{\HS(\XC, \cY)}^2$), and we cannot make use of an analogous argument.
However, similar to the case with output PCA, we recognize that the derivative-informed subspace can also result in a bounded Hessian projection error when the derivatives $\derivXC \fun$ 
are able to represent the Hessian $\derivXC ^2 \fun$ on average. 
This is the subject of \Cref{assumption:hessian_inverse_inequality}, which directly translates to a bound on the derivative ridge function error in terms of the trailing eigenvalues of the DIS problem. 
\begin{proposition}[Derivative ridge function error bound with DIS]\label{prop:ridge_derivative_bound_dis}
  Suppose $\fun \in H^2_{\gamma}$ satisfies \Cref{assumption:hessian_inverse_inequality}
  and $(\Cx^{-1}\vhat_i^{\JTJ})_{i=1}^{\infty}$ are well-defined in $\cX$.
  Then, for any orthogonal projection $\Pry \in \cP_{\ry}(\cY)$, the empirical input DIS satisfies
  \begin{align}
    & \bE_{\gamma} \left[\|\Pry (\derivXC \fun  - \derivXC \bE_{\gamma}[\fun | \sigma(\Pinv{\Vrxhat^{\JTJ}})])\Qrxhat^{\JTJ} \|_{\HS(\XC, \cY)}^2 \right] \nonumber \\
    & \qquad \qquad \qquad \leq 
      K_H \left( \sum_{i = \rx + 1}^{\infty} \mu_i^{\JTJ} +\min \left\{ \sqrt{2 r} \|\Hx - \Hxhat \|_{\HS(\XC, \XC)}, \frac{2 \| \Hx - \Hxhat\|_{\HS(\XC,\XC)}^2}{\mu_{r}^{\JTJ} - \mu_{r+1}^{\JTJ}} \right\}\right).
      \label{eq:input_senstivity_basis_derivative_ridge_error}
  \end{align}
\end{proposition}
\begin{proof}
  As in the computations for the input PCA basis, we have 
  \[ 
    \|\Pry (\derivXC \fun  - \derivXC \bE_{\gamma}[\fun | \sigma(\Pinv{\Vrxhat^{\JTJ}})])\Qrxhat^{\JTJ} \|_{\HS(\XC, \cY)}^2 \leq 
    \|(\derivXC \fun  - \derivXC \bE_{\gamma}[\fun | \sigma(\Pinv{\Vrxhat^{\JTJ}})])\Qrxhat^{\JTJ} \|_{\HS(\XC, \cY)}^2.
  \]
  We then use the Hessian subspace Poincar\'e inequality \eqref{eq:hessian_poincare}, which can be alternatively written as 
  \[ 
    \bE_{\gamma} \left[\|\derivXC \fun - \derivXC \bE_{\gamma}[\fun | \sigma(\Pinv{\Vrxhat^{\JTJ}})]\|_{\HS(\XC, \cY)}^2 \right] 
    \leq \sum_{i= \rx + 1}^{\infty} \bE_{\gamma}[ \| \derivXC^{2} \fun(\cdot, \vhat_i^{\JTJ}) \|_{\HS(\XC, \cY)}^2].
  \]
  \Cref{assumption:hessian_inverse_inequality} is sufficient to give 
  \begin{align*}
    \sum_{i= \rx + 1}^{\infty} \bE_{\gamma}[ \| \derivXC^{2} \fun(\cdot, \vhat_i^{\JTJ}) \|_{\HS(\XC, \cY)}^2 ]
    \leq K_H \sum_{i= \rx + 1}^{\infty} \bE_{\gamma}[\| \derivXC \fun \vhat_i^{\JTJ}\|_{\cY}^2] 
    = K_H \bE_{\gamma}[\|\derivXC \fun (I - \Qrxhat^{\JTJ}) \|_{\HS(\XC, \cY)}^2].
  \end{align*}
  \Cref{prop:dis_derivative_reconstruction} then yields the desired bound.

\end{proof}

We note that, like the derivative inverse inequality in \Cref{assumption:derivative_inverse_inequality}, this is again a strong assumption.
We can, like before, provide examples of mappings for which it is satisfied. 
In particular, polynomial mappings of degree $n$ also satisfy the inequality \eqref{eq:hessian_inverse_inequality} with $K_H = n-1$. This is also summarized in \Cref{prop:hermite_derivative_and_hessian} of \Cref{sec:example_polynomial_forms}.

\subsection{Sampling errors}\label{sec:sampling_error}
We now discuss approaches for dealing with sampling errors associated with the estimators of the mean $\fhat$, covariance $\Cyhat$, and sensitivity operators $\Hxhat$ and $\Hyhat$.
Under mild moment assumptions, one can prove a bound on the expected sample errors.
To this end, we largely follow the approaches taken in \cite{ReissWahl20,BhattacharyaHosseiniKovachkiEtAl21}, 
where the authors analyzed the sampling errors for the construction of the PCA. 
The following useful result is presented in the proof of \cite[Lemma B.2]{BhattacharyaHosseiniKovachkiEtAl21},
which we summarize as a standalone lemma and present its proof in \Cref{sec:proof_monte_carlo_error_bounds}.

\begin{lemma}[Monte Carlo error for vector-valued random variables] \label{lemma:monte_carlo_error}
  Let $z$ be a random variable taking values in a separable Hilbert Space $\cH$ with distribution $\nu$ and mean $\bar{z} = \bE_{z \sim \nu}[z]$. 
  Suppose that $z$ has a bounded second moment, $\bE_{z \sim \nu}[\|z\|_{\cH}^2] < \infty$. 
  Then, given $N$ i.i.d.~samples $(z_k)_{i=1}^{N}$, $z_k \sim \nu$, 
  the sample estimator $\widehat{z} = \frac{1}{N}\sum_{k=1}^{N} z_k$ has an expected error of 
  \begin{equation}
    \bE_N [\|\hat{z} - \bar{z}\|_{\cH}^2] = \frac{\bE_{z \sim \nu}[\|z - \bar{z}\|_{\cH}^2]}{N},
  \end{equation}
  where the expectation $\bE_{N}$ is taken with respect to the random samples $(z_k)_{k=1}^{N}$.
\end{lemma}

This allows us to bound all the sample errors in expectation. In particular, the usual Monte Carlo error of the mean estimator $\fhat$ follows as an immediate corollary. 
\begin{corollary}\label{corollary:mean_estimator_error}
  Suppose $\fun \in L^2_{\gamma}$. 
  Then, given $N$ i.i.d.~samples $(x_k)_{k=1}^{N}$, $x_k \sim \nu$, the sample estimator $\fhat = \frac{1}{N} \sum_{k=1}^{N} f(x_k)$ has an expected error of 
  \begin{equation}
    \bE_N [\|\fhat - \fbar\|_{\cY}^2] = \frac{\|\fun - \fbar\|_{L^2_{\gamma}}^2}{N}
  \end{equation}
  where the expectation $\bE_N$ is taken with respect to the random samples $(x_k)_{k=1}^{N}$.
\end{corollary}

Additionally, under the assumptions that the random variable $z$ has bounded fourth moments, we can obtain a bound on error of the sample estimator $\widehat{z}$ in terms of its fourth moment. The proof is again provided in \Cref{sec:proof_monte_carlo_error_bounds}.
\begin{lemma}\label{lemma:fourth_moment_error}
  Let $z$ be a random variable taking values in a separable Hilbert Space $\cH$ with distribution $\nu$ and mean $\bar{z} = \bE_{z \sim \nu}[z]$. 
  Suppose that $z$ has a bounded fourth moment, $\bE_{z \sim \nu}[\|z\|_{\cH}^4] < \infty$. 
  Then, given $N$ i.i.d.~samples $(z_k)_{i=1}^{N}$, $z_k \sim \nu$, 
  the sample estimator $\widehat{z} = \frac{1}{N}\sum_{k=1}^{N} z_k$ has an expected error of 
  \begin{equation}
    \bE_N [\|\hat{z} - \bar{z}\|_{\cH}^4] \leq \frac{\bE_{z \sim \nu}[\|z - \bar{z}\|_{\cH}^4] + 3 \bE_{z \sim \nu}[\|z - \bar{z}\|^2_{\cH}]^2}{N^2},
  \end{equation}
  where the expectation $\bE_{N}$ is taken with respect to the random samples $(z_k)_{k=1}^{N}$.
\end{lemma}

We use the above results to obtain bounds on the sampling errors for the empirical PCA and DIS.
In particular, the empirical output PCA errors are considered in \cite{BhattacharyaHosseiniKovachkiEtAl21} in the uncentered case. 
In the centered (mean-shifted) case, we additionally need to account for the error in the mean estimate, $\fbar - \fhat$.
The result is a similar estimate as in \cite[Lemma B.2]{BhattacharyaHosseiniKovachkiEtAl21}.
\begin{proposition}\label{prop:pod_sample_error_expectation}
  Suppose $\fun \in L^4_{\gamma}$.
  Then, there exists a constant $M_{\Cy}$ such that 
  the sampling error for the empirical output PCA based on $N$ samples of 
  $\fun(x_k)$, $k = 1, \dots, N$, $x_k \sim \gamma$ i.i.d., satisfies the bound 
  \begin{equation}
    \bE_N [\| \Cy - \Cyhat \|_{\HS(\cY, \cY)}^2] \leq \frac{M_{\Cy}}{N},
  \end{equation}
  where the expectation $\bE_N$ is taken with respect to the random samples $(x_k)_{k=1}^{N}$.
\end{proposition}

We can perform a similar calculation for the sample errors of $\Hxhat$ and $\Hyhat$, this time under the assumption that the Sobolev derivatives $\|\derivXC \fun\|_{\HS(\XC, \cY)}$ have bounded fourth moments.
\begin{proposition}\label{prop:dis_sampling_error_mean}
  Suppose $\fun \in W^{1,4}_{\gamma}$.
  Then, there exists constants $M_{\Hx}$ and $M_{\Hy}$ 
  such that the sampling error for the 
  empirical sensitivity operators 
  \[\Hxhat = \sum_{k=1}^{N} \derivXC \fun(x_k)^* \derivXC \fun(x_k) \quad \text{and} \quad \Hyhat = \sum_{k=1}^{N} \derivXC \fun(x_k) \derivXC \fun(x_k)^*\]
  based on $N$ i.i.d.~samples $(x_k)_{k=1}^{N}$, $x_k \sim \gamma$, satisfy the bounds 
  \begin{equation}
    \bE_N [\| \Hx - \Hxhat \|_{\HS(\XC, \XC)}^2] \leq \frac{M_{\Hx}}{N}
  \end{equation}
  for the input, and 
  \begin{equation}
    \bE_N [\| \Hy - \Hyhat \|_{\HS(\cY, \cY)}^2] \leq \frac{M_{\Hy}}{N}
  \end{equation}
  for the output. The expectation $\bE_{N}$ is taken with respect to the random samples $(x_k)_{k=1}^{N}$. 
\end{proposition}

The proofs for both sampling error results, \Cref{prop:pod_sample_error_expectation} and \Cref{prop:dis_sampling_error_mean}, are given in \Cref{sec:proof_pod_dis_sampling_errors}.

\section{Numerical experiments}\label{sec:numerical_experiments}

In this section, we present several numerical experiments to validate the performance of RBNOs in learning an operator and its derivative. 
Here, we focus on operators arising from the solution of PDEs given inputs distributed as Gaussian random fields. 
We compare the generalization errors in terms of the $L^2_{\gamma}$ norm and $H^1_{\gamma}$ (semi)norm across different dimension reduction strategies, ranks, architectural choices, and training sample sizes. 
The generalization errors are also decomposed into their reconstruction error and ridge function error (neural network error) components. 
These numerical results seek to explore both agreement with our theory
and the practical settings that are not covered by the theoretical analysis.
In particular, error sources such as the statistical errors in the finite training set for the neural network and optimization errors associated with the nonconvex training problem are not analyzed, but will contribute to the overall generalization error. 

\subsection{PDE problems}
For the numerical experiments, we consider PDE problems posed on 1D and 2D domains $\Omega$, where discretization errors can be controlled with modest discretization dimensions. This ensures that the discretized derivative operators (i.e., Jacobian matrices) can be computed in full for the purpose of generating test data.
Nevertheless, the set of PDE problems we consider feature nonlinear PDEs and non-trivial input-output dependence.

Here, we will let $x$ denote the input function, $y$ denote the solution,
and instead will use $s \in \bR^{d}, d = 1,2$ to denote the spatial variable 
Additionally, we will use $\nabla$ to denote the gradient and $\Delta$ to denote the Laplacian, both with respect to the spatial variable $s$.
For the inputs, we will use Gaussian random fields $\gamma = \cN(0, \Cx)$ over $\cX = L^2(\Omega)$, 
with $\Cx$ taking the form of the inverse of an elliptic PDE operator, 
$\Cx = (-a_{\Delta} \Delta + a_{I})^{-\alpha}$,
where $a_{\Delta} > 0, a_{I} > 0$, and $\alpha > 0$ control the pointwise variance, correlation structure, and smoothness of the random field.

\paragraph{Semilinear elliptic PDE} Our first 1D example is the semilinear elliptic PDE posed on the unit interval $\Omega = (0, 1)$:
\begin{align}
    - \nabla \cdot (c_1(s) \nabla y(s)) + c_2 y(s)^3 &= x(s), \qquad s \in \Omega, \\
    y(0) &= 0, \\ 
    \frac{\partial y}{\partial s}(1) &= 0, 
\end{align}
where $c_1(s) = 0.0001 + 0.01 \chi_{(0.5, 1)}(s)$ and $c_2 = 0.1$ are the diffusion and reaction coefficients, 
and $x \sim \cN(0, (-2\Delta + 10)^{-1})$.

\paragraph{Steady Burgers equation}
For a second 1D example, we consider the steady Burgers equation posed on the unit interval $\Omega = (0, 1)$:
\begin{align}
    - \nabla \cdot (c_1 \nabla y(s)) + y(s) \nabla y(s) &= c_2(s) x(s), \qquad s \in \Omega, \\
    y(0) &= 0, \\ 
    y(1) &= 0,
\end{align}
where $c_1 = 0.01$ is the viscosity and the source term is a profile $c_2(s) = (0.025 \sqrt{2 \pi})^{-1} \exp(-(s - 0.4)^2/0.05^2)$ centered at $s = 0.4$
and 
scaled by the random field $x \sim \cN(0, (-10\Delta + 20)^{-1})$.

\paragraph{Linear elasticity}
We also consider a 2D problem arising from linear elasticity with a random Young's modulus. 
To avoid overloading notation, 
we use $c_{\lambda}$ and $c_{\mu}$ to denote the Lam\'e parameters (as opposed to the commonly used $\lambda$ and $\mu$). 
These are given as 
\[
    c_\lambda = \frac{c_E c_\nu}{(1 + c_\nu)(1 - 2 c_\nu)},
    \quad
    c_{\mu} = \frac{c_E}{2(1 + c_\nu)}, 
\]
where $c_E$ is the Young's modulus and $c_\nu$ is the Poisson's ratio.
We then introduce the stress tensor $\mathbf{T}$ as a function of the displacement $y = (y_1, y_2)$ such that 
\[
    \mathbf{T} = c_\lambda \nabla \cdot y \mathbf{I} + c_\mu \left( \nabla y + \nabla y^T \right),
\]
where $\mathbf{I}$ is the identity tensor on $\bR^2$.
We pose the PDE on the domain $\Omega = (0,1)^2 \setminus \Omega_R$, 
where $\Omega_R$ is a disk of radius $R = 0.2$ centered at $(0.5, 0.5)$.
The governing equations are then given by
\begin{align*}
    -\nabla \cdot \mathbf{T} &= 0 \quad \text{in } \Omega, \\
    y &= 0 \quad \text{on } \Gamma_L, \\
    \mathbf{T} n &= \tau \quad \text{on } \Gamma_R, \\
    \mathbf{T} n &= 0 \quad \text{on } \Gamma_T \cup \Gamma_B \cup \partial \Omega_R,
\end{align*}
where $\Gamma_L$ is the left boundary, $\Gamma_R$ is the right boundary,
$\Gamma_T$ is the top boundary, $\Gamma_B$ is the bottom boundary, and $\partial \Omega_R$ is the boundary of the disk $\Omega_R$.
Additionally, $\tau$ is the traction vector on the right boundary, $n$ is the outward normal vector.
The input variable $x(s)$ is a Gaussian random field defining the Young's modulus
$c_{E}(s) = 1 + 40 \exp(x(s))$ where $x \sim \cN(0, (-0.2 \Delta + 4)^{-2})$.
Moreover, we fix $c_{\nu} = 0.3$, and take 
\[
    \tau(s) = \left(
        \frac{1}{0.1 \sqrt{2\pi}} \exp \left(-\frac{(s_2 - 0.5)^2}{2 \cdot 0.1^2} \right), 0
    \right).
\]
This corresponds to a tensile load with a Gaussian profile applied to the right boundary.



\subsection{Data generation and basis computation}
We use the finite element method to discretize the PDEs and generate the training and test data.
This is implemented using FEniCS \cite{LoggMardalWellsEtAl12}
for the finite element solvers and hIPPYlib/hiPPYflow \cite{VillaPetraGhattas21,hippyflow}
for training data generation.
In particular, for each input sample $x_k \sim \gamma$ 
the PDE is solved to obtain the output $y_k = \fun(x_k)$.
Moreover, given discretizations of the input and output spaces with dimensions $\dimx$ and $\dimy$, 
the derivative of the solution operator, $DF(x_k)$, 
corresponds to a Jacobian matrix of size $\dimy \times \dimx$.
This can be computed using direct and adjoint sensitivity methods. 
That is, we can obtain a Jacobian-vector product $DF(x_k) v$ or a vector-Jacobian product $DF(x_k)^T v$ for any vector $v \in \bR^{\dimx}$ or $\bR^{\dimy}$ by solving the linearized state PDE or its adjoint, respectively.
We refer to \cite{OLearyRoseberryChenVillaEtAl24} for more details on the computation of the derivatives.

For the 1D examples, we use a uniform mesh with $\nelem = 256$ linear elements.
In this case, the discretized dimensions of the input and output spaces are $\dimx = \dimy = 257$.
For the 2D examples, we use an unstructured mesh with linear triangular elements,
where the discretized dimensions of the input and output spaces are $\dimx = 407$ and $\dimy = 814$.
In either case, we have chosen the discretization dimension to be sufficiently large such that numerical errors in solving the PDE are unimportant, yet sufficiently small so that we can compute the Jacobian matrix, $\deriv \fun (x) \in \bR^{\dimy \times \dimx}$, in full by direct sensitivity methods.
That is, we generate training and testing data tuples of the form $(x_k, y_k, \deriv \fun (x_k))$ for $k = 1, \ldots, N$, where $x_k \sim \gamma$.
For each training run, the training data is used to estimate the covariance operator $\Cyhat$ or the sensitivity operators $\Hxhat$ and $\Hyhat$, which are then used to solve the eigenvalue problem to obtain the basis functions.

We note that in practice, for high discretization dimensions, randomized methods such as those described in \cite{SaibabaLeeKitanidis15} can be used to efficiently solve the PCA and DIS eigenvalue problems.
This requires only the Jacobian-vector products and vector-Jacobian products on a small number of random vectors, rather than the full Jacobian matrix itself.
Similarly, for Sobolev training, one only needs to compute the reduced Jacobian matrix, $\pinv{\Ury}\deriv \fun(x_k) \Vrx$,
corresponding to the action of the derivative operator on the reduced bases $\Ury$ and $\Vrx$ obtained from the PCA/DIS procedures.
We refer to \cite{OLearyRoseberryVillaChenEtAl22,OLearyRoseberryChenVillaEtAl24} for more detailed description on practical procedures for computing DIS and derivative training data.



\subsection{Neural network architecture and training}
In the numerical experiments, we use the Sobolev training formulation \eqref{eq:reduced_sobolev_training} to train the neural networks using both the output value and derivative data.
The $L^2_{\gamma}$ and $H^1_{\gamma}$ components of the loss function are normalized by scaling each data point by its norm, 
i.e., $a_{0,k} = \|\pinv{\Ury} (\fun(x_k) - b)\|_{2}^2$ 
and $a_{1,k} = \|\pinv{\Ury} \derivXC \fun(x_k) \Vrx\|_{F}^2$, 
as described in \Cref{sec:sobolev_training}.
We have found that this approach works well to prevent the neural network from overfitting to outliers in the derivative data.


We train the neural networks using the Adam optimizer with a constant batch size of 25 for 3,000 epochs, using a piecewise constant learning rate schedule that is successively halved five times starting at the 2,250th epoch.
The hyperparameters, including $d_W$, $d_L$, and initial learning rates are selected independently for each PDE problem.
These are summarized in the \Cref{tab:hyperparameters}.

\begin{table}[t]
\centering
\caption{Hyperparameters used for the latent neural networks in the numerical experiments for the three PDE problems considered.
Note the width is expressed as a multiple of the rank $r$.}
\label{tab:hyperparameters}
\begin{tabular}{ c c c c c }
    \toprule
    Problem & Depth $d_L$ & Width $d_W$ & Activation &  Initial learning rate  \\ 
    \midrule
    Semilinear elliptic & 6 & $2r$ & softplus & 0.001  \\
    Steady Burgers & 4 & $2r$ & softplus  & 0.00025  \\
    Linear elasticity & 4 & $2r$ & softplus  & 0.001  \\
    \bottomrule
\end{tabular}
\end{table}

\subsection{Numerical results}
\subsubsection{Comparison of reconstruction errors}
We begin by presenting a comparison of the different choices of dimension reduction strategies in reconstructing the output and the derivatives.
To this end, we compute the PCA/DIS bases using a large sample size of $N = 20,\!000$, 
which we take as a proxy for the exact PCA/DIS. 
The resulting eigenvalues are shown in \Cref{fig:spectrum} for the PCA/DIS bases corresponding to the three PDE problems: semilinear elliptic, steady Burgers, and linear elasticity.
Additionally, we present samples of the inputs and outputs, along with a selection of the first few basis vectors in each case. 
These are given in \Cref{fig:samples_and_bases_1d} for the 1D examples (semilinear elliptic and steady Burgers PDEs) and in \Cref{fig:samples_and_bases_2d} for the 2D linear elasticity problem.

In the spectra of the covariance/sensitivity operators in \Cref{fig:spectrum},
we observe that the eigenvalues decay rapidly at an algebraic rate (i.e., $\lambda_i \sim i^{-\beta}$ for some $\beta > 0$) in all cases. 
On the input side, the PCA eigenvalues simply correspond to the eigenvalues of the covariance operator (second order/fourth order elliptic PDE operators). 
These decay rates are independent of the input-output map.
On the other hand, the DIS eigenvalues tend to decay at a faster rate than the PCA eigenvalues, reflecting the sensitivity information of the underlying map.
This can be seen in the input basis vectors shown in \Cref{fig:samples_and_bases_1d}. 
Whereas PCA basis vectors are global sinusoidal functions, the DIS basis vectors are localized to the regions of low diffusion coefficient in the semilinear elliptic example and 
the near the Gaussian peak in the steady Burgers equation example.
On the output side, the spectra and corresponding basis vectors are similar between the PCA and DIS in the examples considered.
This reflects the fact that both the output PCA and DIS take into account the underlying mapping. 
Moreover, we also observe that the DIS eigenvalues are larger than the PCA eigenvalues.

\begin{figure}
    \centering
    \includegraphics[width=0.99\textwidth]{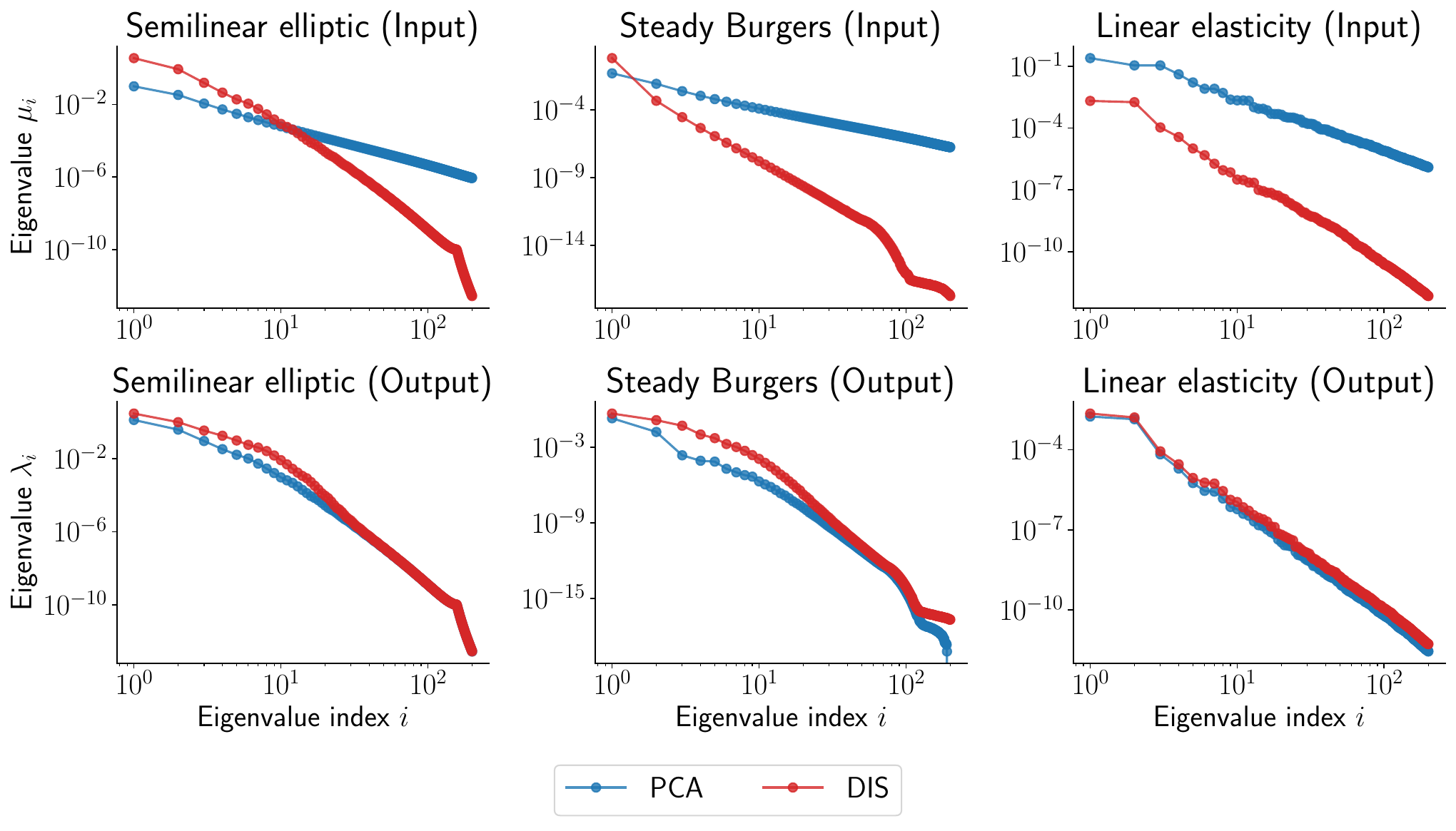}
    \caption{First 200 eigenvalues corresponding to the PCA and DIS bases for the three PDE problems considered. These are computed using a large sample size of $N$ = 20,000, which we take as a proxy for the exact PCA/DIS.}
    \label{fig:spectrum}
\end{figure}

\begin{figure}
    \begin{subfigure}{0.99\textwidth}
        \centering
        \includegraphics[width=0.99\textwidth]{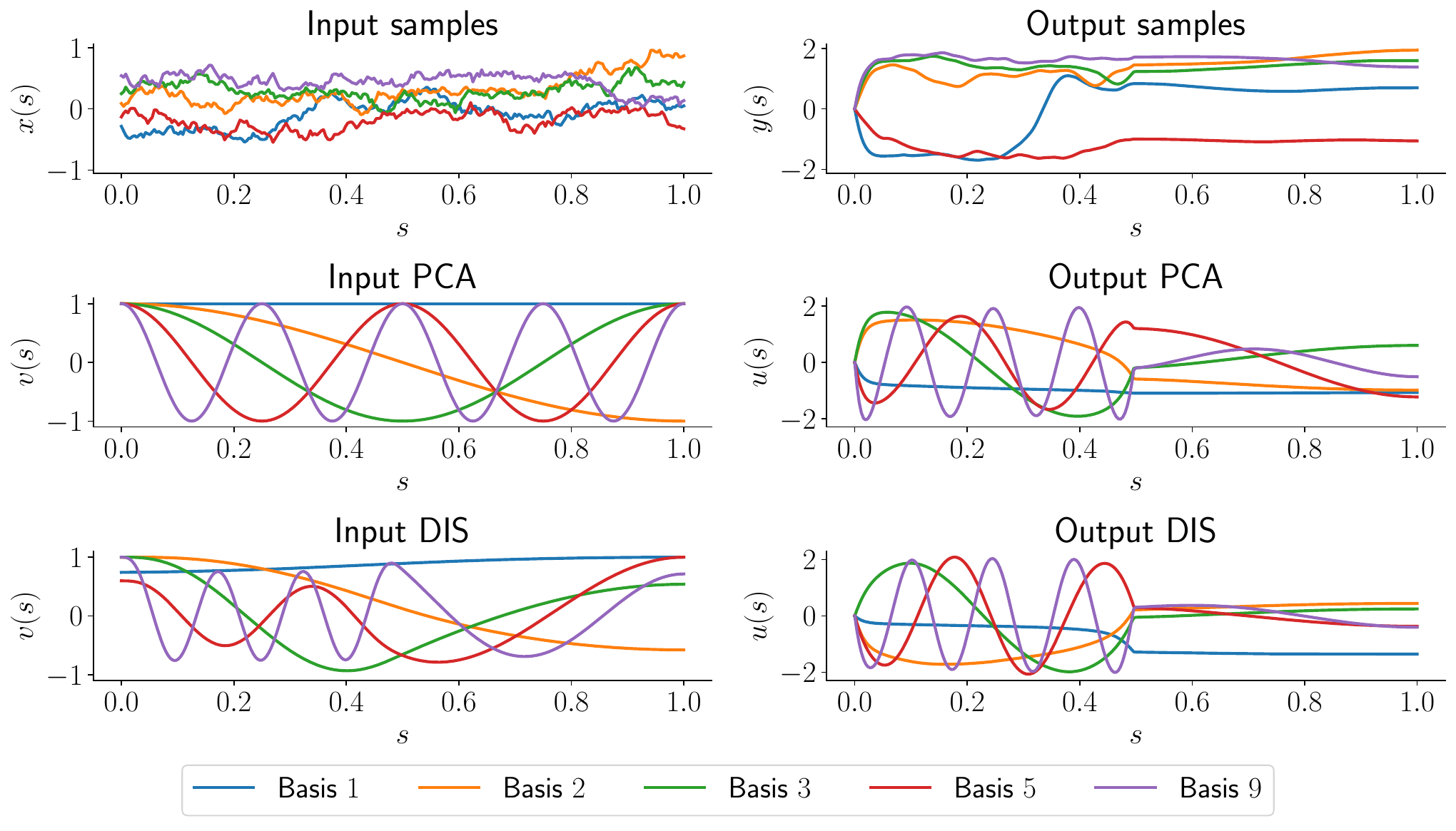}
        \caption{Semilinear elliptic PDE}
    \end{subfigure}
    \begin{subfigure}{0.99\textwidth}
        \centering
        \includegraphics[width=0.99\textwidth]{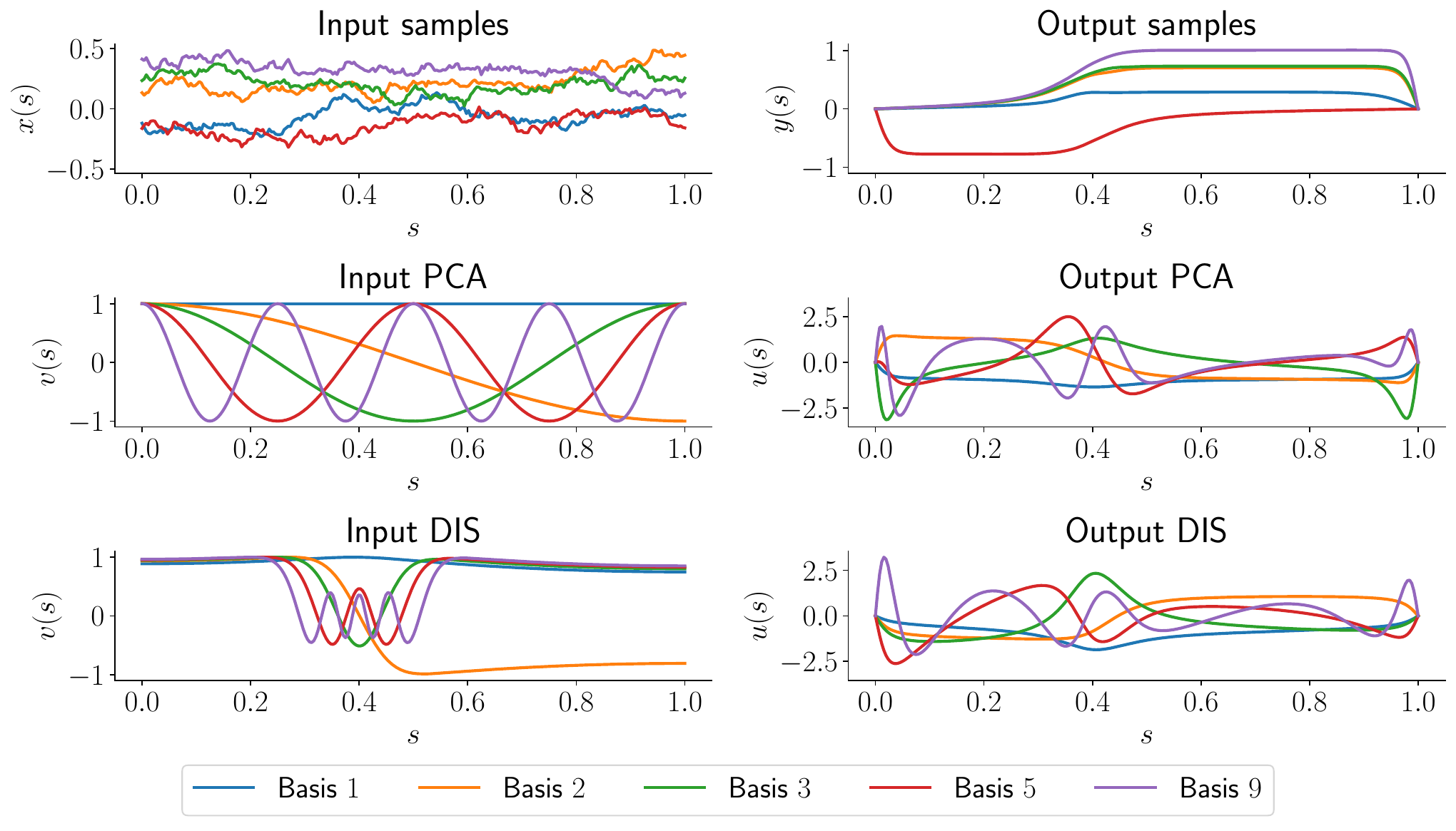}
        \caption{Steady Burgers equation}
    \end{subfigure}
    \caption{Samples of input-output pairs along with the 1st, 2nd, 3rd, 5th, and 9th PCA/DIS basis vectors for the 1D PDE examples. Top: semilinear elliptic PDE. Bottom: steady Burgers equation. Note the localization of the input DIS basis vectors to the region $s \leq 0.5$ in the semilinear elliptic PDE example where the diffusion coefficient is smallest, 
    and to the region near $s = 0.4$ that is the center of the Gaussian profile scaling the source in the Burgers example.}
    \label{fig:samples_and_bases_1d}
\end{figure}

\begin{figure}
    \centering
    \begin{subfigure}{0.99\textwidth}
        \centering
        \includegraphics[width=0.4\textwidth]{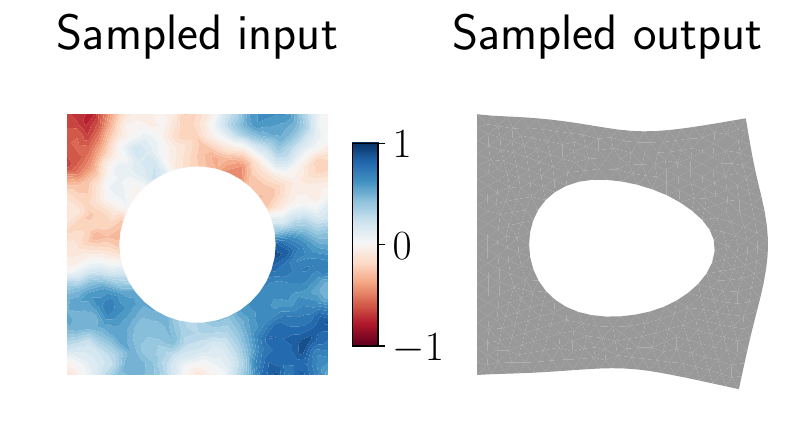}
        \caption{Sample input and output}
    \end{subfigure}

    \vspace{10pt}

    \begin{subfigure}{0.99\textwidth}
        \centering
        \includegraphics[width=0.99\textwidth]{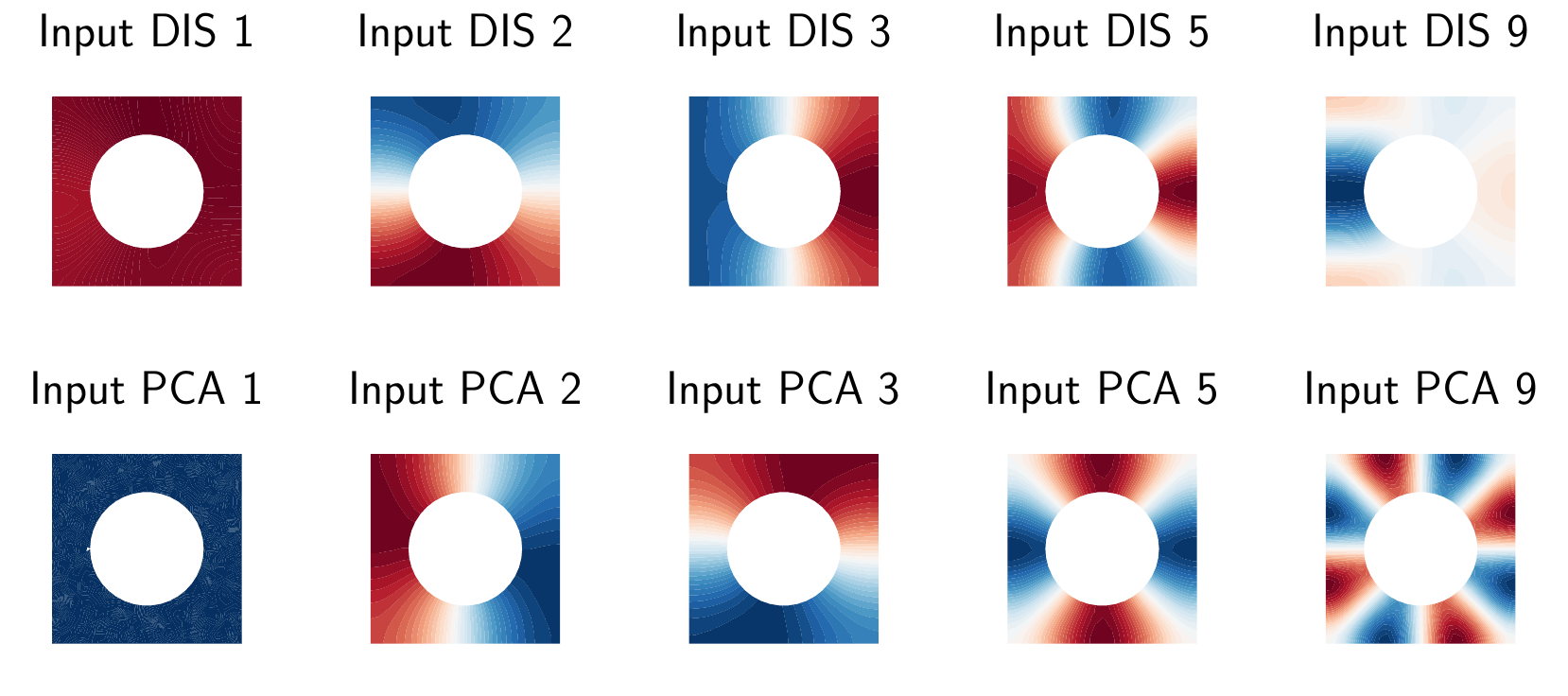}
    \end{subfigure}
    \begin{subfigure}{0.99\textwidth}
        \centering
        \includegraphics[width=0.5\textwidth]{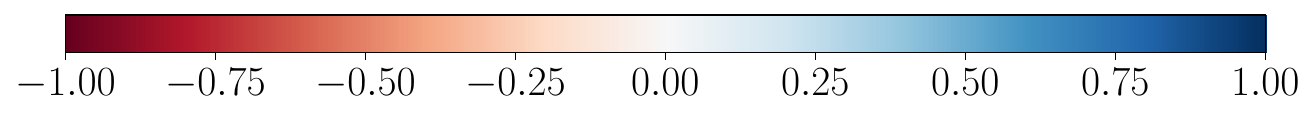}
    \end{subfigure}
    \begin{subfigure}{0.99\textwidth}
        \centering
        \includegraphics[width=0.99\textwidth]{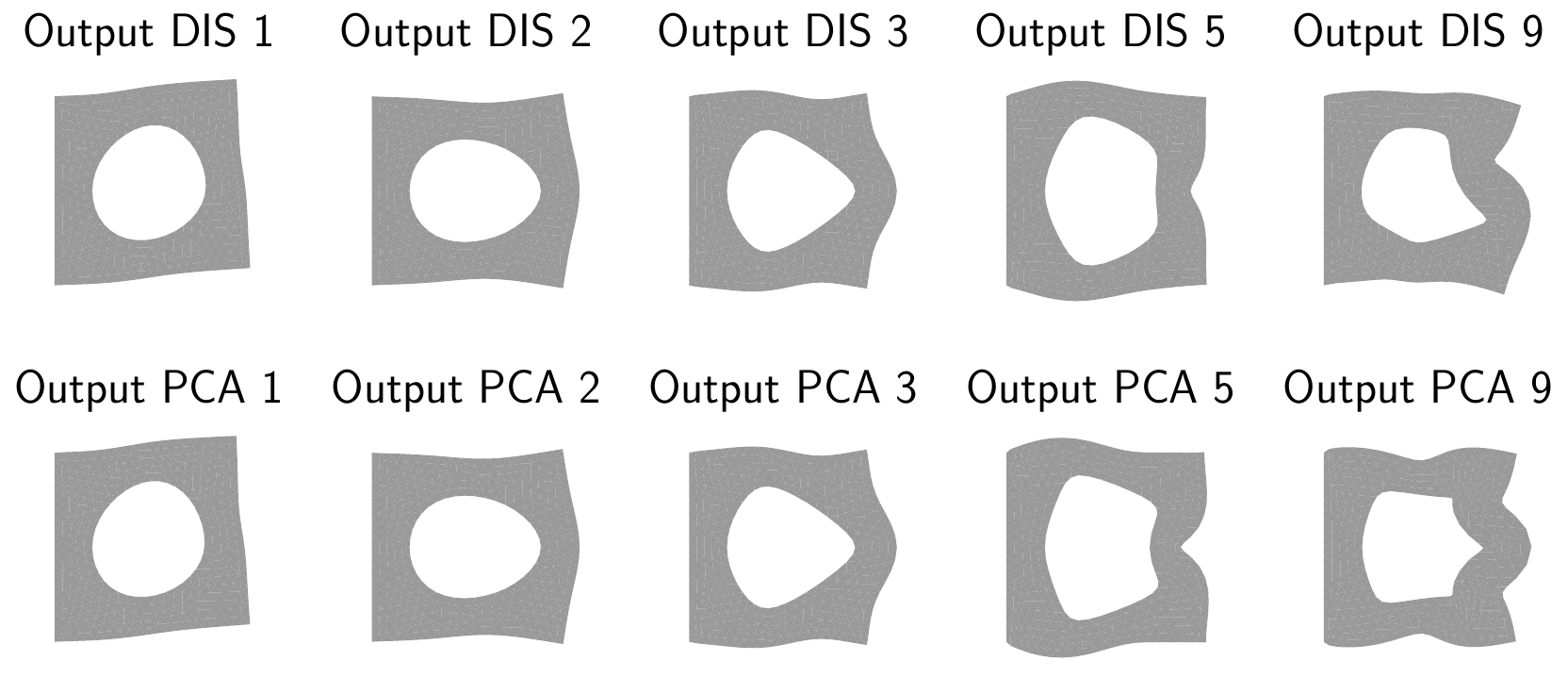}
        \caption{Input and output bases}
    \end{subfigure}
    \caption{Top: a sample of an input-output pair for the linear elasticity problem.  
    Bottom: the 1st, 2nd, 3rd, 5th, and 9th PCA/DIS basis vectors for the input and output of the linear elasticity problem.
    }
    \label{fig:samples_and_bases_2d}
\end{figure}

\begin{figure}
    \begin{subfigure}{0.99\textwidth}
        \centering
        \includegraphics[width=0.99\textwidth]{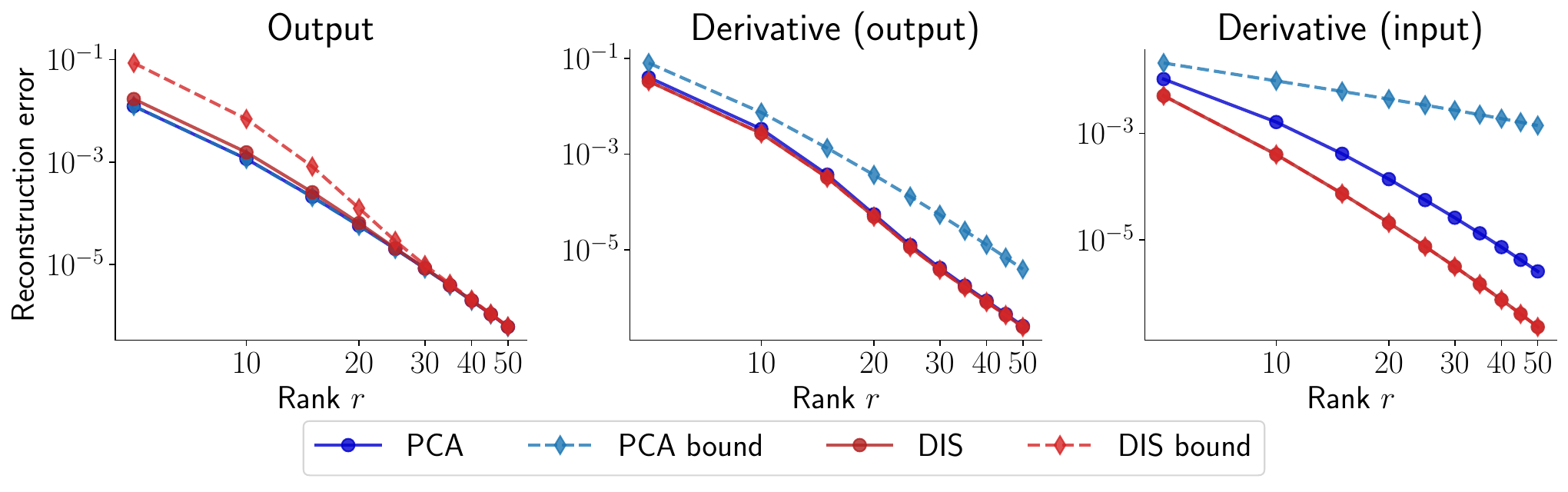}
        \caption{Semilinear elliptic PDE}
    \end{subfigure}
    \begin{subfigure}{0.99\textwidth}
        \centering
        \includegraphics[width=0.99\textwidth]{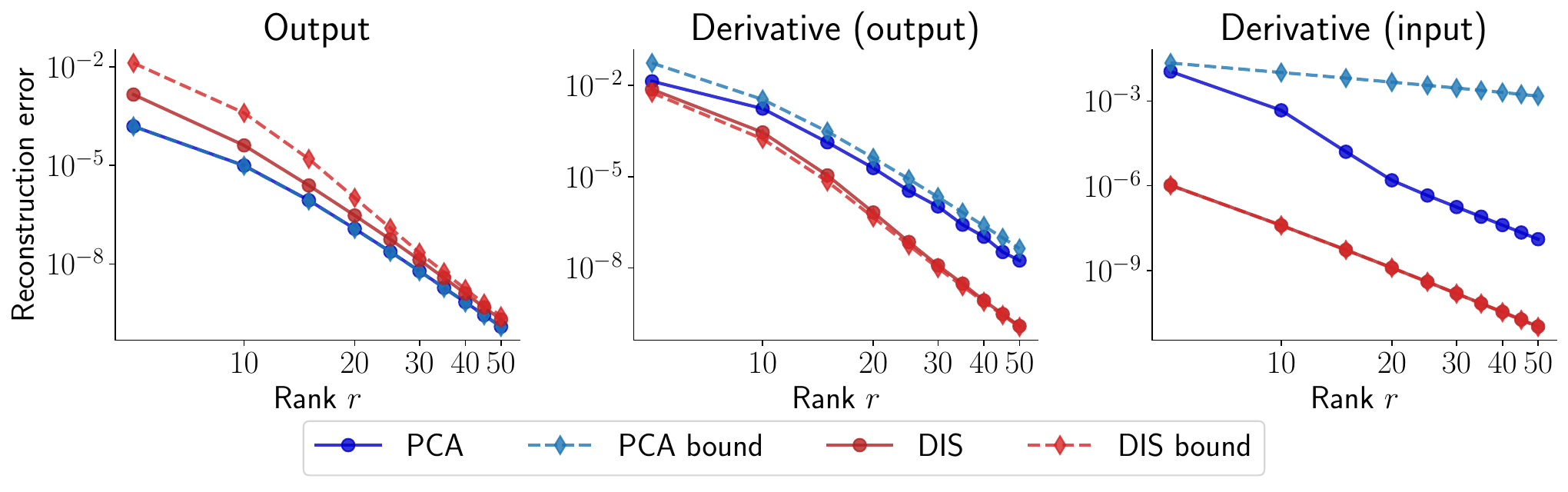}
        \caption{Steady Burgers equation}
    \end{subfigure}
    \begin{subfigure}{0.99\textwidth}
        \centering
        \includegraphics[width=0.99\textwidth]{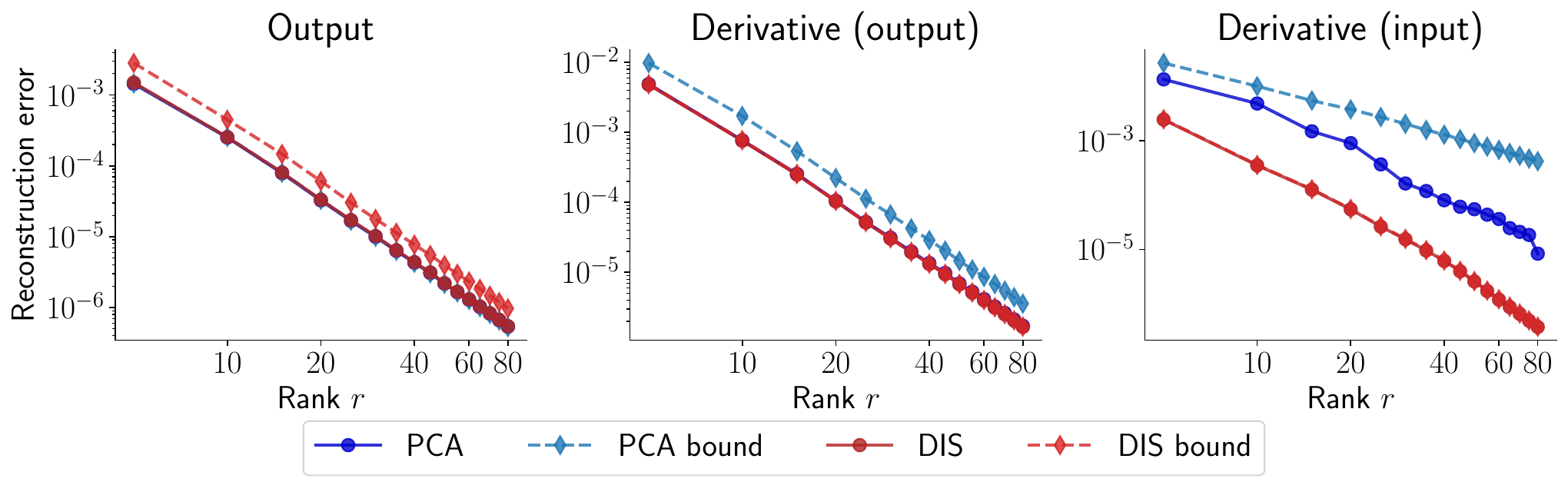}
        \caption{Linear elasticity equation}
    \end{subfigure}
    \caption{Reconstruction errors (normalized by the squared second moments of the test data) in the output and its derivatives for the three PDE problems. Left column: reconstruction error of the output by output PCA/DIS. Center column: reconstruction error of the derivative by the output PCA/DIS. Right column: reconstruction error of the derivative by the input PCA/DIS. Bounds for the reconstruction error in terms of the trailing eigenvalues are shown in the dashed lines. Here, true bounds are shown for the output PCA/DIS in output reconstruction and input/output DIS for derivative reconstruction cases.
    For the reconstruction of the derivatives by input/output PCA basis, we do not explicitly compute the constant, but instead the plot shows the line corresponding to $K_1 \sum_{i=r+1}^{\infty} \lambda_i^{\POD}$ and $K_2 \sum_{i=r+1}^{\infty} \mu_i^{\KLE}$ for some prescribed $K_1, K_2$ to illustrate the decay rate of the reconstruction error bound.
    }
    \label{fig:reconstruction_errors}
\end{figure}

In \Cref{fig:reconstruction_errors}, we present the reconstruction errors for the output (by the output PCA/DIS) and its derivative (by output PCA/DIS and input PCA/DIS) for the three PDE problems considered in this work.
That is, we plot the reconstruction errors, 
\begin{align*}
    \frac{\|(I - \Pry)(\fun - \fbar)\|_{L^2_{\gamma}}^2}{\|\fun\|_{L^2_{\gamma}}^2}, \quad
    \frac{\|(I - \Pry)\derivXC \fun\|_{L^2(\gamma, \HS(\XC,\cY))}^2}{\|\derivXC \fun\|_{L^2(\gamma, \HS(\XC,\cY))}^2}, \quad
    \text{and} \quad
    \frac{\|\derivXC \fun (I - \Qrx)\|_{L^2(\gamma, \HS(\XC,\cY))}^2}{\|\derivXC \fun\|_{L^2(\gamma, \HS(\XC,\cY))}^2}, \quad
\end{align*}
for different ranks $r$,
where each is normalized by the squared second moment of the output or the derivative.
Here, the errors are computed using a test set of $N_\text{test} = 22,\!000$ output and derivative samples.
The corresponding reconstruction error bounds of \Cref{sec:reconstruction_errors} are also plotted for reference.
Specifically, for the output reconstruction by output PCA/DIS, and the derivative reconstruction by input/output DIS, the bounds are simply the trailing eigenvalue sums and are computed and plotted exactly in \Cref{fig:reconstruction_errors}.
On the other hand, for the derivative reconstruction by input/output PCA, 
the bounds are additionally scaled by a constant ($\bE_{\gamma}[\|\deriv \fun\|_{\cL(\cX,\cY)}^2]$ and $K_D$, respectively). 
We do not compute these explicitly, and instead plot the lines $K_1 \sum_{i=r+1}^{\infty} \lambda_i^{\POD}$ and $K_2 \sum_{i=r+1}^{\infty} \mu_i^{\KLE}$ for some prescribed values of $K_1, K_2$ to illustrate the decay rate of the reconstruction error bound.

For the output reconstruction, output PCA and DIS perform comparably, with significant reductions in the reconstruction error as the rank $r$ increases.
The two only differ slightly in the steady Burgers example, where PCA errors noticeably smaller for the lower ranks.
Similarly, the two output bases are also comparable when reconstructing the derivative, 
where the reconstruction errors using the PCA and DIS bases are almost identical for the semilinear elliptic PDE and linear elasticity examples.
However, for the steady Burgers example, the DIS reconstruction errors are consistently smaller than the PCA, though both errors do decay rapidly.
On the other hand, when reconstructing the derivative using the input PCA/DIS, 
the two approaches perform quite differently, where the DIS bases produce significantly lower errors than the PCA bases.
As previously mentioned, this is to be expected, since the input PCA basis does not incorporate any information of the underlying map and hence cannot be expected to accurately capture the derivative information of the operator.

We also make a few remarks on the reconstruction error bounds. First, we note that for the cases of output reconstruction by the output PCA and derivative reconstruction by the input/output DIS, 
the bounds are, in fact, equalities, and match the observed errors.
In the case of the output reconstruction by the output DIS, 
the bounds, though slightly conservative, do provide a good approximation of the observed errors, especially for the higher ranks.
For the derivative reconstruction by the output PCA, the bound required the additional assumption of \Cref{assumption:derivative_inverse_inequality}, postulating the existence of a constant that relates the map's derivative $\derivXC \fun$ to its output $\fun$.
While this is not proven for the examples, the reconstruction errors do appear to decay at a rate consistent with a bound of the form $K_D \sum_{i=r+1}^{\infty}\lambda_i^{\POD}$,
suggesting that the assumption may be reasonable.
Finally, for the derivative reconstruction bounds of the input PCA, we observe that the $K \sum_{i=r+1}^{\infty} \mu_i^{\KLE}$ bounds are overly conservative for the 1D example problems. The observed errors instead decay at a faster rate that tracks with the DIS rate,
suggesting that the PCA bases may be serendipitously aligned with the sensitivity of the operator in these cases.
On the other hand, for the 2D linear elasticity problem, 
the observed errors follow more closely with the $K \sum_{i=r+1}^{\infty} \mu_i^{\KLE}$ bound,
which is a slower decay than the DIS.

\subsubsection{Excess risks in reconstruction errors}
Next, we consider the excess risks associated with the sampling errors in the empirical PCA/DIS.
That is, we compute the empirical PCA and DIS using smaller sample sizes ranging from $N = 80$ to $N = 640$,
and then compare the excess in reconstruction errors when compared to the proxy of the exact PCA/DIS, which we compute with $N = 20,\!000$.
That is, for each empirical output PCA $\Uryhat^{\POD}$, we evaluate the excess risk in the output reconstruction error,
\[
    \frac{\|(I - \Pryhat^{\POD})(\fun - \fbar)\|_{L^2_{\gamma}}^2 - \|(I - \Pry^{\POD})(\fun - \fbar)\|_{L^2_{\gamma}}^2}{\|\fun\|_{L^2_{\gamma}}^2}, \quad
\]
while for the empirical output/input DIS, $\Uryhat^{\JTJ}$ and $\Vrxhat^{\JTJ}$, we evaluate the excess risk in the derivative reconstruction error,
\[
    \frac{
    \| (I - \Pryhat^{\JJT}) \derivXC \fun \|_{L^2(\gamma, \HS(\XC,\cY))}^2 - \| (I - \Pry^{\JJT}) \derivXC \fun \|_{L^2(\gamma, \HS(\XC,\cY))}^2
    }{\|\derivXC \fun\|_{L^2(\gamma, \HS(\XC,\cY))}^2}
\]
and 
\[
    \frac{\|\derivXC \fun (I - \Qrxhat^{\JTJ})\|_{L^2(\gamma, \HS(\XC,\cY))}^2 - \|\derivXC \fun (I - \Qrx^{\JTJ})\|_{L^2(\gamma, \HS(\XC,\cY))}^2}{\|\derivXC \fun\|_{L^2(\gamma, \HS(\XC,\cY))}^2},
\]
all normalized by the squared second moment of the output or the derivative, respectively.
Here, the errors are again estimated using the test set of $N_\text{test} = 22,\!000$ function and derivative samples.
This is repeated over 10 independent runs to obtain results on the average excess risk.
We present these in \Cref{fig:semilinear_adr_excess_risk}, \Cref{fig:steady_burgers_excess_risk}, and \Cref{fig:linear_elasticity_excess_risk} for the semilinear elliptic, steady Burgers, and linear elasticity problems, respectively.

The excess risk corresponding to the semilinear elliptic and linear elasticity examples both appear to follow the local rate of $\cO(N^{-1})$ for both PCA and DIS and for both ranks considered.
Moreover, the excess risks decrease as the ranks increase,
implying that the inverse dependency on the spectral gap in the bounds \Cref{prop:pod_sample_error_expectation}
and \Cref{prop:dis_sampling_error_mean} is not strongly felt in these examples.
With the rapid decay of eigenvalues, spectral gaps for the larger ranks are orders of magnitude smaller than that of the smaller ranks and should otherwise have a great impact on the excess risk bound.
This suggests that the particular form of the excess risk bound is pessimistic in these scenarios, and that more care is needed in deriving sharper estimates, such as those in \cite{BlanchardBousquetZwald06,ReissWahl20}.

The results for the steady Burgers example is less conclusive.
The $\cO(N^{-1})$ rate is observed in some cases, while the error appears to follow more closely the rate of $\cO(N^{-1/2})$ for the output DIS with $r = 10$ and the input DIS for both $r = 10$ and $r = 30$.
However, we also note that the resulting trends are likely polluted with sampling errors in computing both the PCA/DIS used as the reference, as well as in the estimation of the excess risks themselves by finite test data and finitely many runs.
In particular, the excess risks for the input DIS are on the order of $10^{-10}$ using just $N = 80$, and hence small errors in the computations can lead to significant variations in the estimated excess risks. 
For this reason, one should be careful in drawing conclusions from the steady Burgers example.

\begin{figure}[t!]
    \begin{subfigure}{0.99\textwidth}
        \centering
        \includegraphics[width=0.99\textwidth]{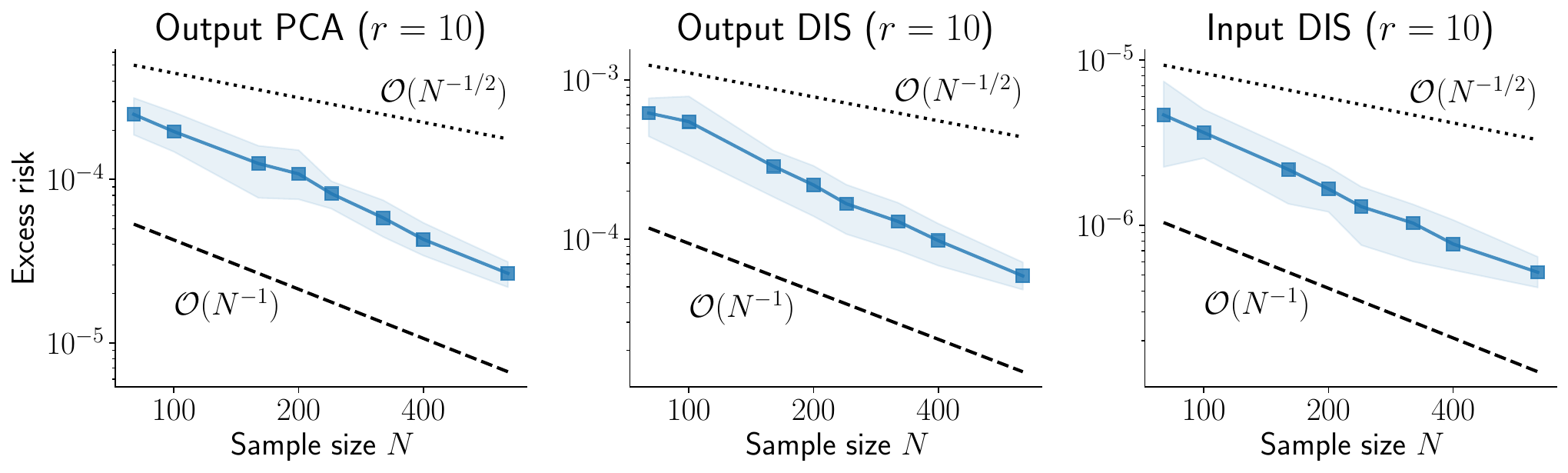}
    \end{subfigure}
    \begin{subfigure}{0.99\textwidth}
        \centering
        \includegraphics[width=0.99\textwidth]{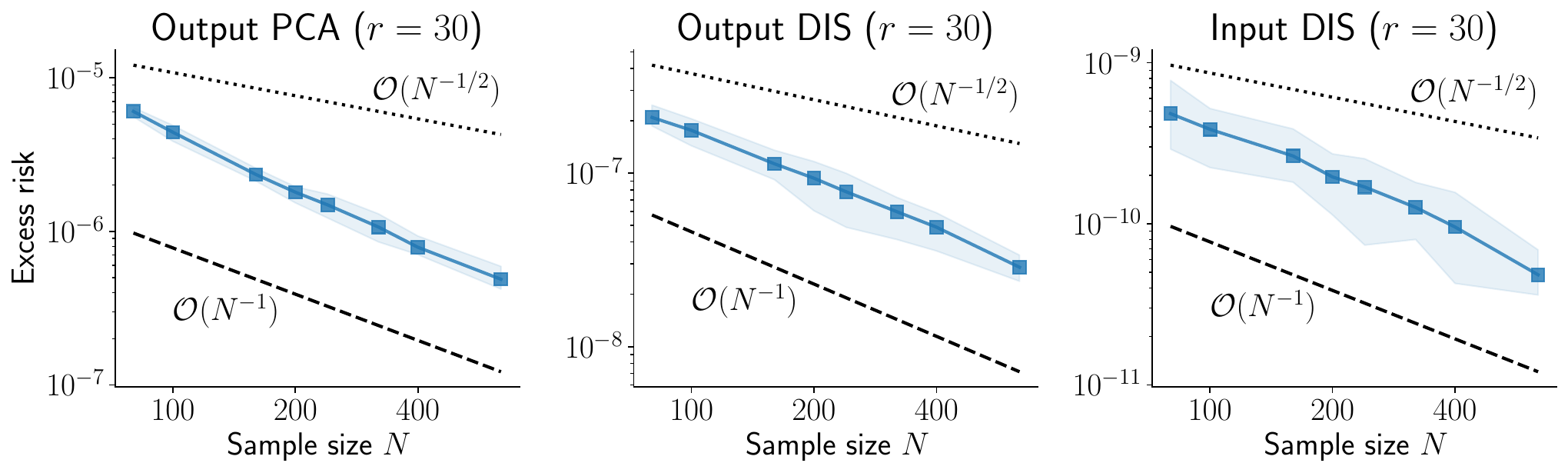}
    \end{subfigure}
    \caption{Excess risks in the output and derivative reconstruction errors for the semilinear elliptic PDE problem for ranks $r= 10$ and $r = 30$. 
    Values are normalized by the squared second moments of the test data.
    For reference, the global $N^{-1/2}$ rate is shown in the dotted line while the local $N^{-1}$ rate is shown in the dashed line. The filled regions correspond to the 10\%--90\% quantile ranges of the 10 independent runs.}
    \label{fig:semilinear_adr_excess_risk}
\end{figure}

\begin{figure}[t!]
    \begin{subfigure}{0.99\textwidth}
        \centering
        \includegraphics[width=0.99\textwidth]{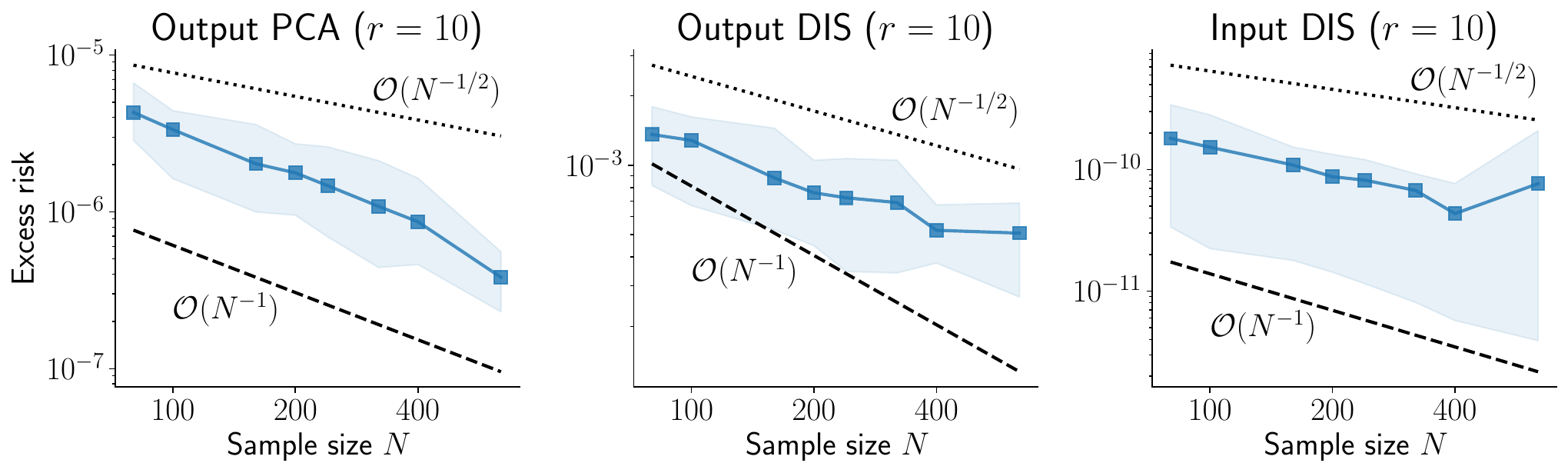}
    \end{subfigure}
    \begin{subfigure}{0.99\textwidth}
        \centering
        \includegraphics[width=0.99\textwidth]{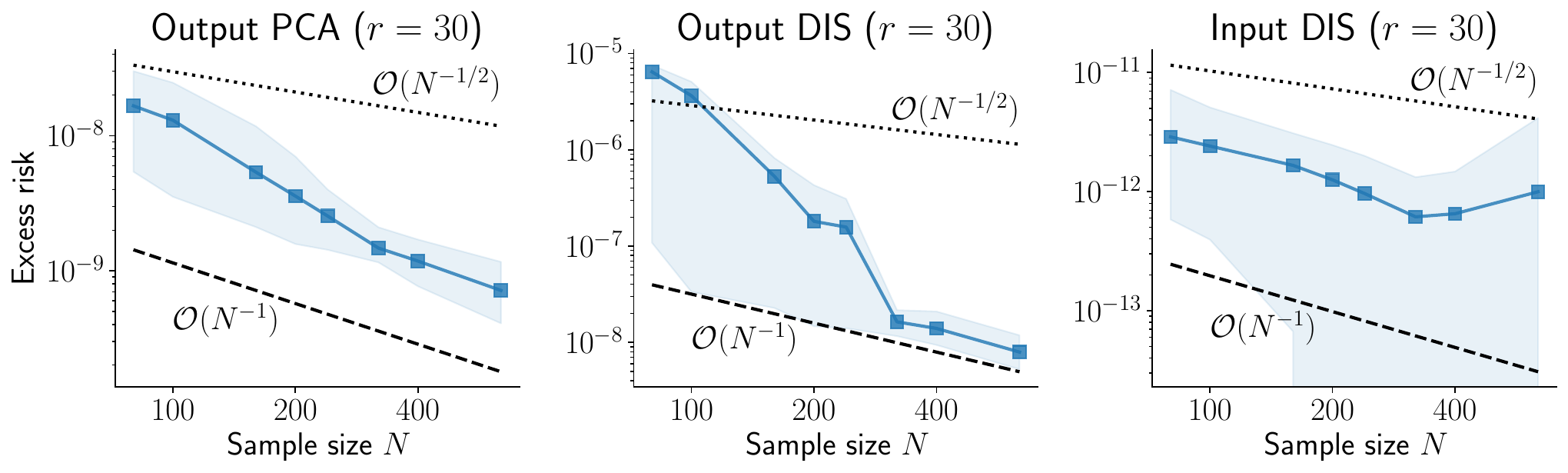}
    \end{subfigure}
    \caption{Excess risks in the output and derivative reconstruction errors for the steady Burgers problem for ranks $r= 10$ and $r = 30$. 
    Values are normalized by the squared second moments of the test data.
    For reference, the global $N^{-1/2}$ rate is shown in the dotted line while the local $N^{-1}$ rate is shown in the dashed line.
    The filled regions correspond to the 10\%--90\% quantile ranges of the 10 independent runs.
    }
    \label{fig:steady_burgers_excess_risk}
\end{figure}

\begin{figure}[t!]
    \begin{subfigure}{0.99\textwidth}
        \centering
        \includegraphics[width=0.99\textwidth]{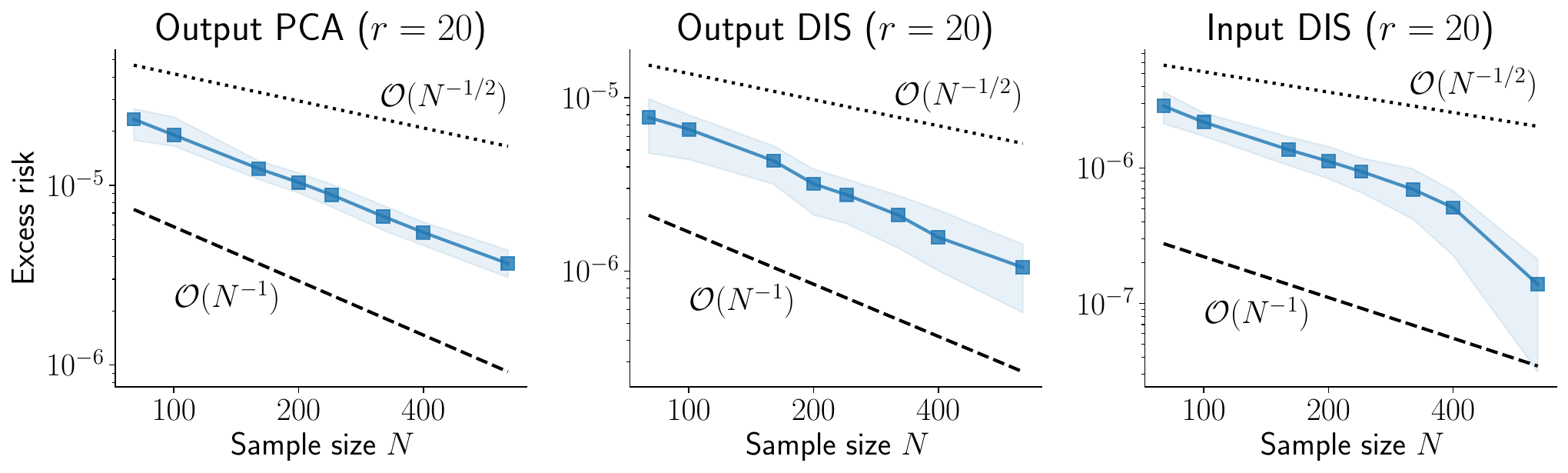}
    \end{subfigure}
    \begin{subfigure}{0.99\textwidth}
        \centering
        \includegraphics[width=0.99\textwidth]{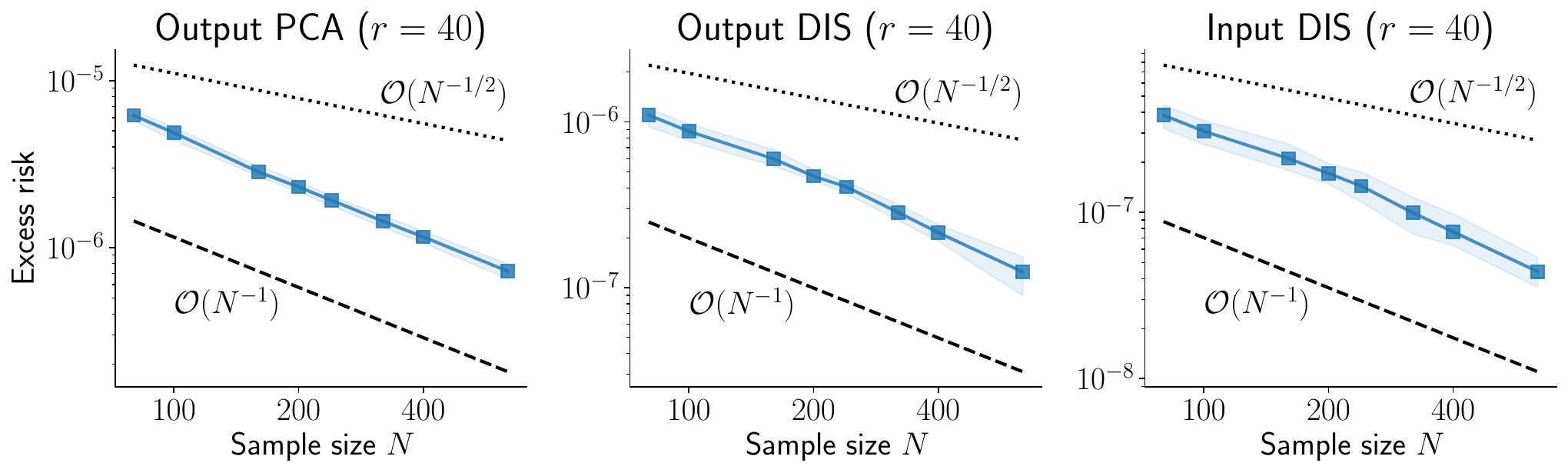}
    \end{subfigure}
    \caption{Excess risks in the output and derivative reconstruction errors for the linear elasticity problem for ranks $r= 20$ and $r = 40$. 
    Values are normalized by the squared second moments of the test data.
    For reference, the global $N^{-1/2}$ rate is shown in the dotted line while the local $N^{-1}$ rate is shown in the dashed line.
    The filled regions correspond to the 10\%--90\% quantile ranges of the 10 independent runs.
    }
    \label{fig:linear_elasticity_excess_risk}
\end{figure}

\subsubsection{Neural network generalization errors}
We now present experiments in which we attempt to learn the operators and their derivatives using the RBNO architecture.
For these experiments, we train RBNOs for the four possible pairs of input/output bases
across a range of different ranks.
For each rank, we consider training sample sizes of $N = 640$ up to $N = 16,\!000$.
We then compute the errors in predicting the output and its derivative, i.e., 
\[
    \frac{\|\fun - \ftilde_{\theta}\|_{L^2_{\gamma}}^2}{\|\fun\|_{L^2_{\gamma}}^2} \quad \text{and} \quad
    \frac{\|\derivXC \fun - \derivXC \ftilde_{\theta}\|_{L^2(\gamma, \HS(\XC,\cY))}^2}{\|\derivXC \fun\|_{L^2(\gamma, \HS(\XC,\cY))}^2}.
\]
The end-to-end neural operator generalization errors are estimated on a test dataset consisting of $N_\text{test} = 4,\!000$ function and derivative samples.
These error results are presented in two different ways.
For fixed ranks, we plot the generalization errors as a function of the number of training samples, while for fixed sample sizes, we plot the generalization errors as a function of the rank.
The results are shown in \Cref{fig:nn_errors_semilinear_adr}, \Cref{fig:nn_errors_steady_burgers}, and \Cref{fig:nn_errors_linear_elasticity} for the semilinear elliptic, steady Burgers, and linear elasticity problems, respectively.

We begin by making a few general observations. 
First, for any fixed rank $r$, both the output and derivative generalization errors 
decrease as the training sample size $N$ increases,
often continuing until a plateau is reached (i.e~a floor value in the error where further increasing the sample size does not lead to a significant decrease in the error).
This likely corresponds to the limit of either the ridge function approximation (i.e., truncated inputs/outputs) or the approximation power of the neural network itself,
as the plateau tends to occur later for larger ranks and the corresponding floor values in error are lower.
We also note that, often, output error reaches the plateau long before the derivative error, and the derivative errors continue to improve with increasing sample sizes.
This is especially noticeable in the steady Burgers example, 
where fast decrease in the derivative errors with respect to the training sample size begins \textit{after} the output errors begin to plateau.
The observed behavior may suggest that the neural networks tend to first learn the output values before they learn the derivatives.

We also recognize that, unlike the reconstruction errors, 
the generalization errors of the neural operators given a fixed training sample size
do not decrease monotonically with respect to rank, as is typical of bias-variance trade-offs in neural network training.
For smaller ranks, the generalization errors can be made small with fewer training samples, but ultimately reaching the limit of the truncated representation.
On the other hand, the ridge functions can be made more expressive with a larger rank,
but require more training samples to learn an accurate neural network representation.
Thus, for a fixed number of training samples, 
there is an optimal rank that achieves the smallest generalization error for the output and derivative predictions.
As seen in Figures \ref{fig:nn_errors_semilinear_adr} to \ref{fig:nn_errors_linear_elasticity}, the optimal ranks are larger when the training sample sizes are larger.
Moreover, the optimal ranks for approximating the output tends to be larger than the optimal ranks for approximating the derivatives, which is consistent with the observation that the derivatives are more difficult to learn. 

Finally, we discuss the effect of the choice of dimension reduction.
In particular, the choice of input dimension reduction, which determines the optimal ridge function error, has the greatest impact on the generalization errors.
In most cases, the input DIS significantly outperforms the input PCA in terms of the errors for both the output and the derivatives.
This suggests that despite the fact that the DIS basis does not explicitly incorporate the second-derivative information required to bound the derivative ridge function errors (as in \Cref{prop:ridge_derivative_bound_dis}), it remains an effective choice of dimension reduction for learning the derivatives in practice.
The differences in errors between the input DIS and PCA do, however, tend to decrease as the rank increases. For large ranks, the generalization errors are likely dominated by the neural network's ability to learn the underlying mapping rather than the dimension reduction.

On the other hand, the choice of output dimension reduction only modestly affects the generalization errors, since it is only involved in the reconstruction of the output/derivatives.
Across the examples considered, this difference is only significant for the steady Burgers example at low ranks ($r = 10,15$) and for sufficiently large training sets, 
where the training errors can be made sufficiently small such that the difference between the reconstruction errors is apparent.
In such cases, the two strategies perform as expected---the PCA leads to smaller errors in the output while the DIS leads to smaller derivative errors.

\begin{figure}
        \begin{subfigure}{0.95\textwidth}
        \includegraphics[width=0.99\textwidth]{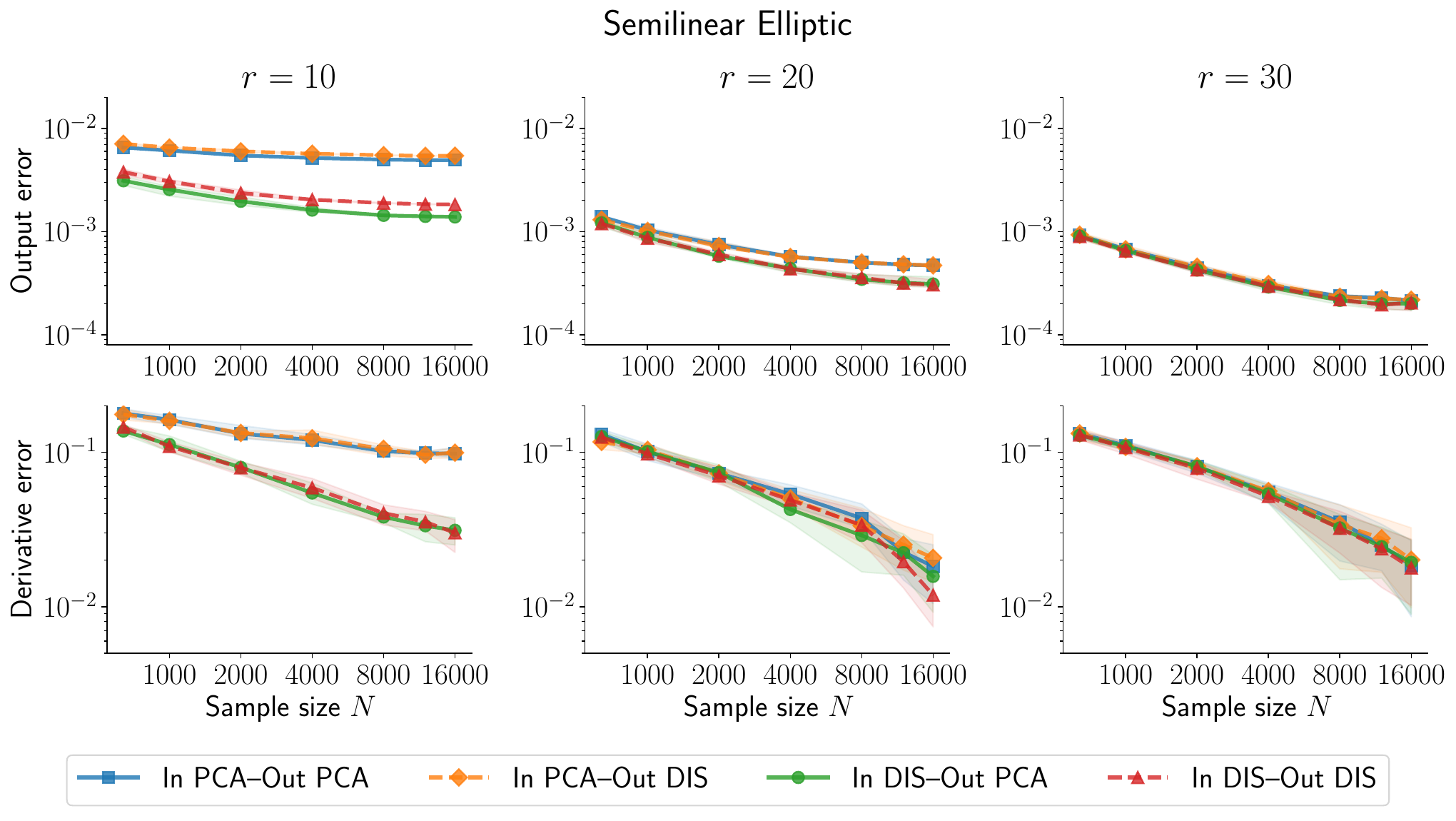}
        \caption{Generalization errors for the semilinear elliptic PDE as a function of the number of training samples for fixed ranks $r =$ $10$, $20$, and $30$.}
        \end{subfigure}
        \begin{subfigure}{0.95\textwidth}
        \includegraphics[width=0.99\textwidth]{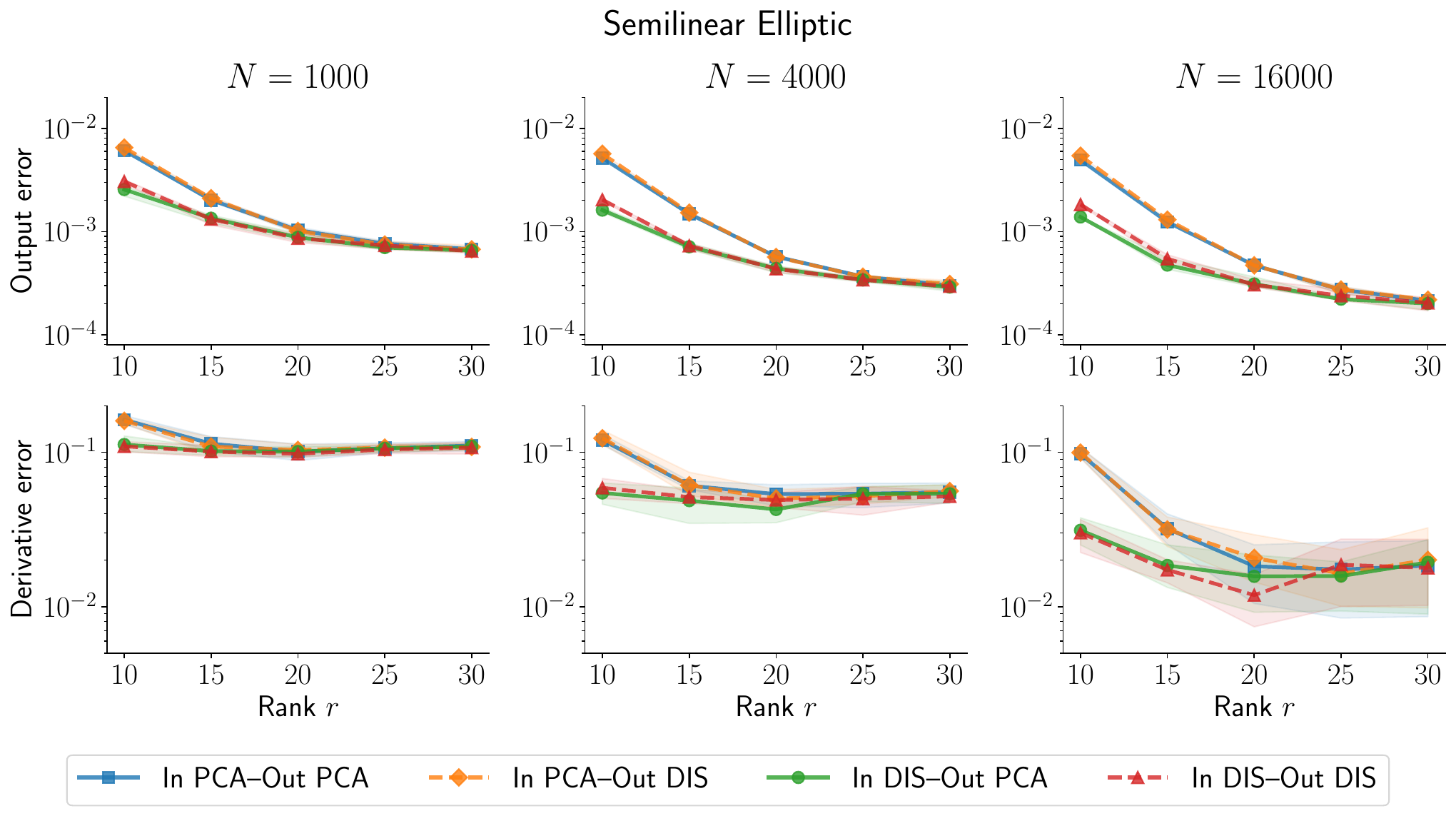}
        \caption{Generalization error for the semilinear elliptic PDE as a function of the rank for a fixed number of training samples, $N =1,\!000$, $4,\!000$, and $16,\!000$.}
        \end{subfigure}
        \caption{Generalization errors for the output ($\|\cdot\|_{L^2_{\gamma}}^2$) and derivatives ($| \cdot |_{H^1_{\gamma}}^2$) of the semilinear elliptic PDE normalized by the squared second moments of the test data. 
        Results are presented as both a function of training sample size (top) and as a function of rank (bottom). Solid lines correspond to averages over 10 independent runs, while the filled regions correspond to the 10\%--90\% quantile ranges across these runs.}
        \label{fig:nn_errors_semilinear_adr}
\end{figure}

\begin{figure}
    \begin{subfigure}{0.95\textwidth}
    \includegraphics[width=0.99\textwidth]{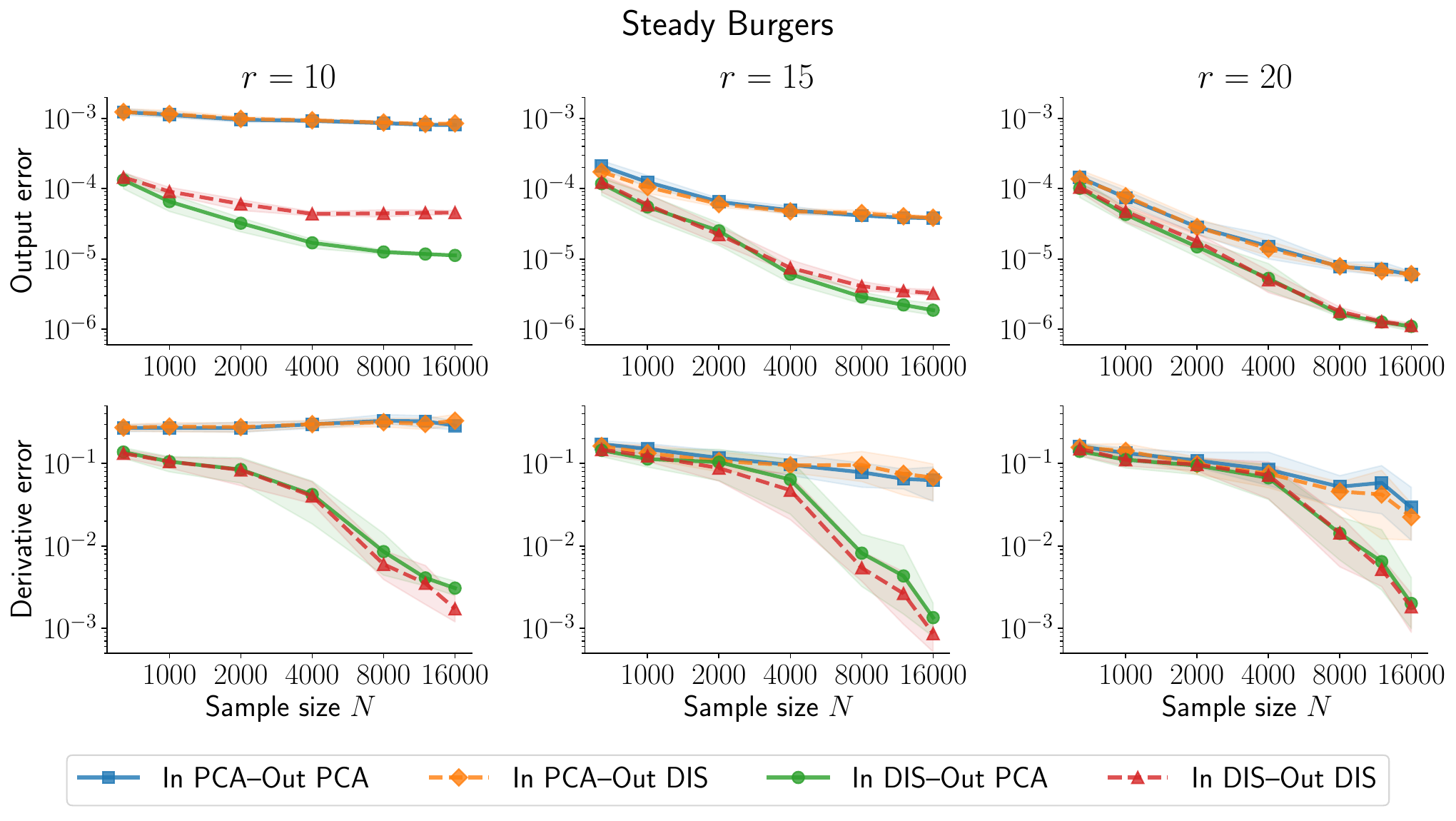}
    \caption{Generalization errors for the steady Burgers PDE as a function of the number of training samples for fixed ranks $r =$ $10$, $20$, and $30$.}
    \end{subfigure}
    \begin{subfigure}{0.95\textwidth}
    \includegraphics[width=0.99\textwidth]{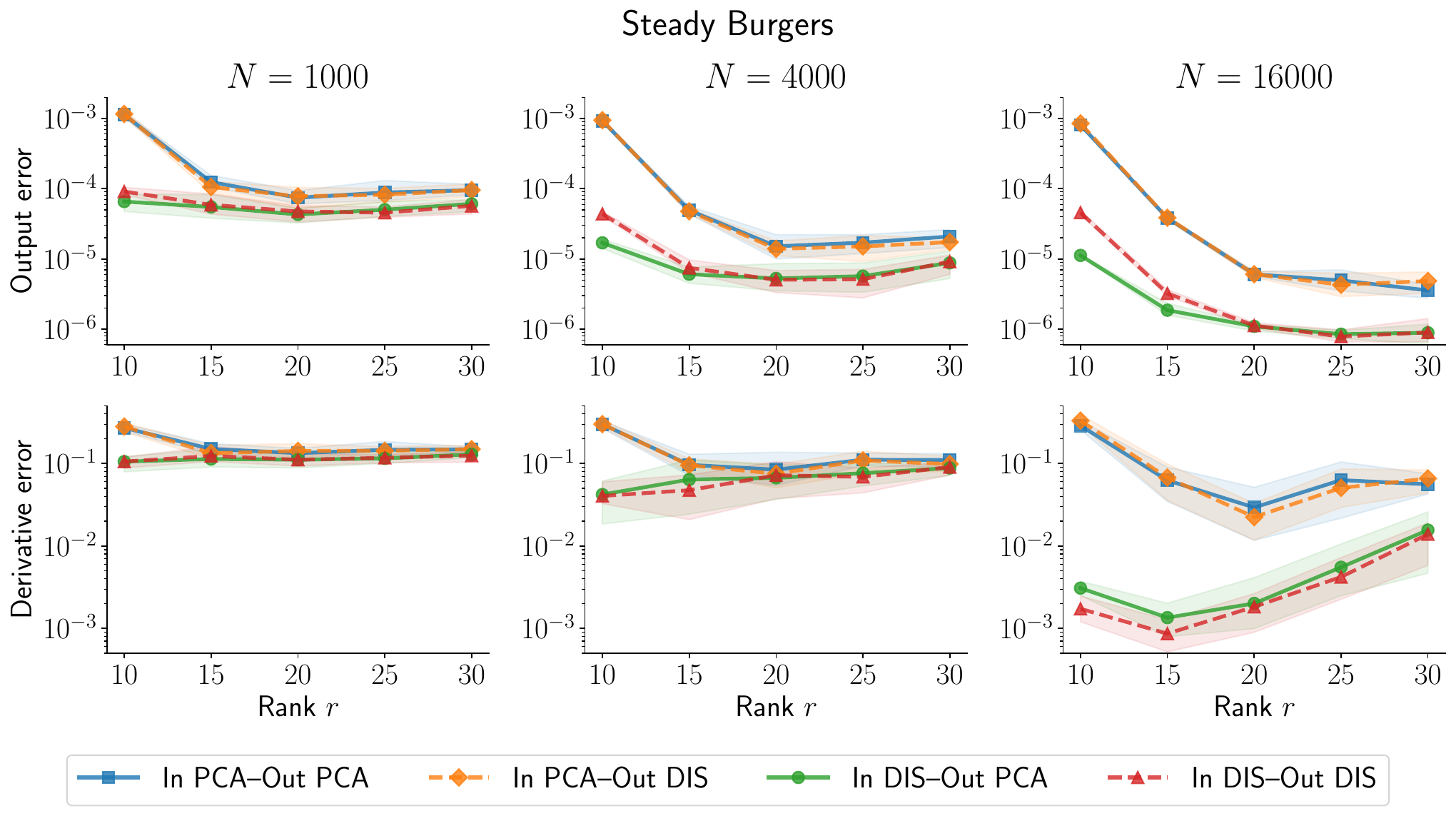}
    \caption{Generalization error for the steady Burgers PDE as a function of the rank for a fixed number of training samples, $N =$ $1,\!000$, $4,\!000$, and $16,\!000$.}
    \end{subfigure}
    \caption{Neural network generalization errors for the output and derivatives of the steady Burgers PDE, presented as both a function of training sample size (top) and as a function of rank (bottom). Solid lines correspond to averages over 10 independent runs, while the filled regions correspond to the 10\%--90\% quantile ranges across these runs.}
    \label{fig:nn_errors_steady_burgers}
\end{figure}

\begin{figure}
    \begin{subfigure}{0.95\textwidth}
    \includegraphics[width=0.99\textwidth]{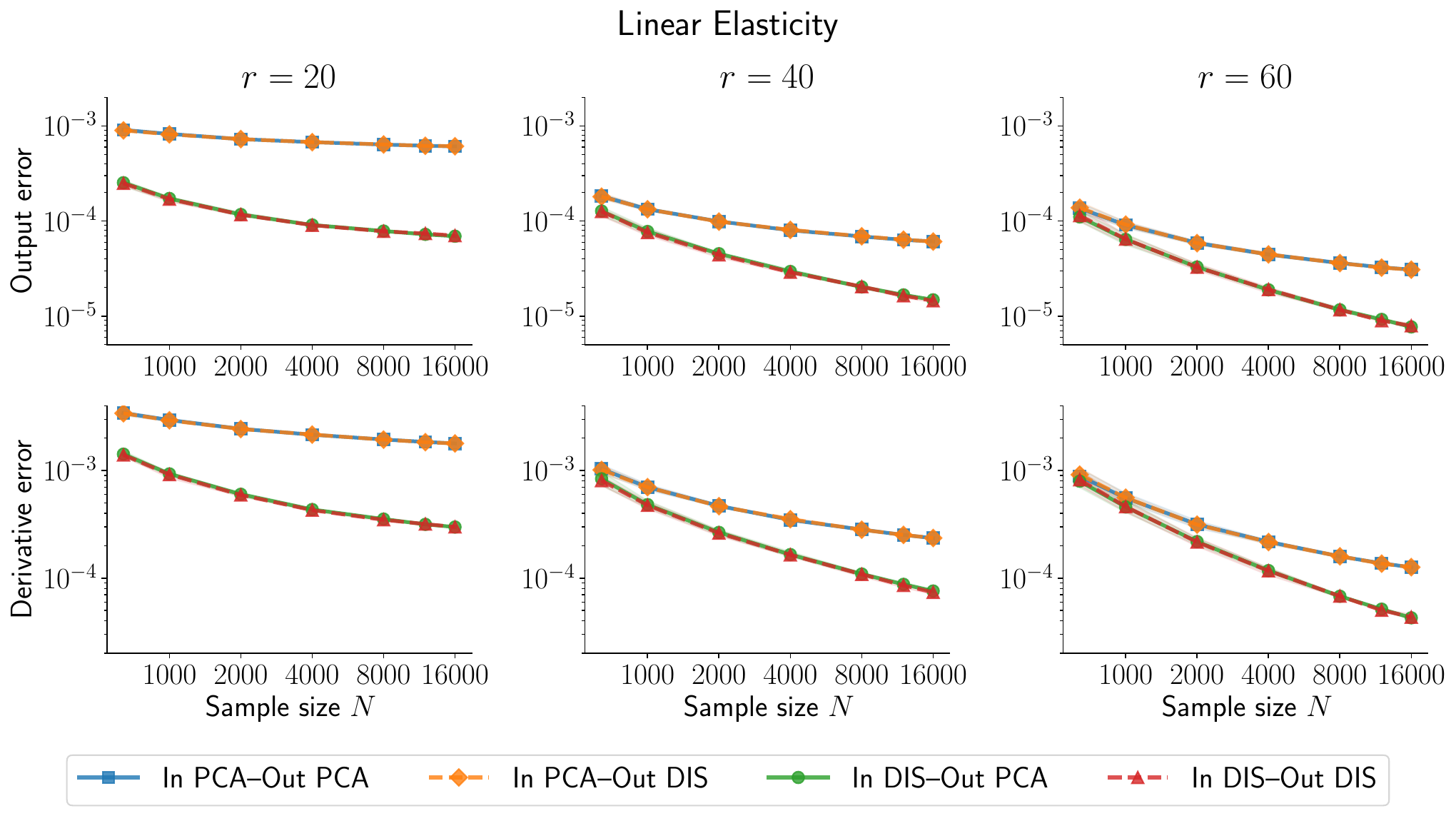}
    \caption{Generalization errors for the linear elasticity PDE as a function of the number of training samples for fixed ranks $r =$ $20$, $40$, and $60$.}
    \end{subfigure}
    \begin{subfigure}{0.95\textwidth}
    \includegraphics[width=0.99\textwidth]{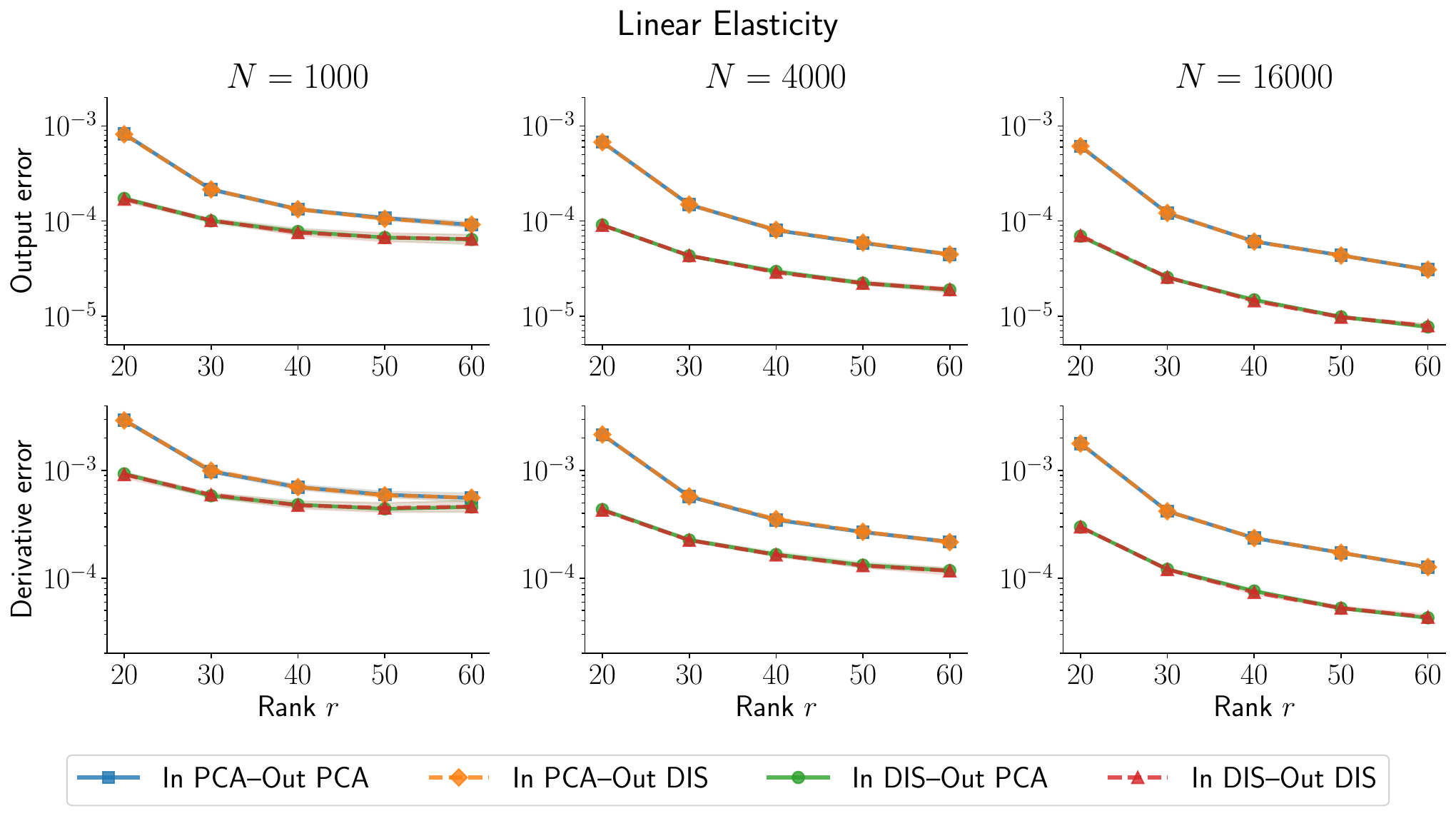}
    \caption{Generalization error for the linear elasticity PDE as a function of the rank for a fixed number of training samples, $N =$ $1,\!000$, $4,\!000$, and $16,\!000$.}
    \end{subfigure}
    \caption{Neural network generalization errors for the output and derivatives of the linear elasticity PDE, presented as both a function of training sample size (top) and as a function of rank (bottom). Solid lines correspond to averages over 10 independent runs, while the filled regions correspond to the 10\%--90\% quantile ranges across these runs.}
    \label{fig:nn_errors_linear_elasticity}
\end{figure}





\section{Conclusions}\label{sec:dibasis_conclusions}
In this work, we considered the analysis of derivative-informed operator surrogates 
that use reduced bases representations of the input and output spaces and a neural network that is constructed to approximate the finite-dimensional mapping between the truncated inputs and outputs. 
The main contributions of the present work are $H^1_{\gamma}$ approximation error bounds developed for different choices of dimension reduction---namely, PCA and DIS. 

To this end, we first presented a universal approximation theorem for generic orthonormal bases, which shows that approximation of operators in $H^m_{\gamma}$ for any $m \geq 0$ to arbitrary accuracy using the reduced basis architecture is possible.
To provide more insight on how the dimension reduction strategy affects the approximation errors, we proceeded with a basis-specific analysis, decomposing the error into dimension reduction error and neural network approximation error.
The dimension reduction errors were considered in detail for the PCA and DIS bases, including the statistical errors associated with the sample-based estimators of the PCA/DIS operators. 
The neural network approximation error was analyzed using the universal approximation results for Sobolev norms from \cite{Hornik91}, which we extended to accommodate deep architectures and non-saturating activation functions. 
The different results were combined to provide an end-to-end error bound. We note that, of the various combinations, 
the mildest assumptions were required to bound the $H^1_{\gamma}$ approximation error when using PCA for input dimension reduction and DIS for output dimension reduction.
In developing these results, we also filled gaps in existing results for the derivative informed subspaces in the $L^2_{\gamma}$ case, especially in the separable Hilbert space setting. 

We also presented numerical results to validate our theoretical predictions, 
comparing the reconstruction errors, excess risks, as well as the end-to-end errors of the resulting neural operators in both $L^2_{\gamma}$ and $H^1_{\gamma}$.
For the reconstruction errors, we observed that the PCA/DIS eigenvalue decay does indeed translate to appropriate bounds in the reconstruction errors.
In the case of empirical PCA/DIS, we observed that the excess risk of the reconstructions typically decrease with sample size $N$ at a rate consistent with the local regime of $\cO(N^{-1})$ as opposed to the global regime of $\cO(N^{-1/2})$.
We did not, however, observe the effect of the inverse dependence on the spectral gap ($\lambda_{r} - \lambda_{r+1}$) in our numerical examples. 
This suggests that our approach for bounding the excess risk is pessimistic,
and that more refined approaches such as those taken in \cite{BlanchardBousquetZwald06,ReissWahl20}
may be necessary to provide sharper and more practical bounds.

When used to construct operator surrogates, we observed that the choice of dimension reduction (PCA vs DIS) is especially important for the performance of the neural operator in both output and derivative learning.
This has important practical implications---an effective choice of dimension reduction allows for smaller rank and hence smaller networks, which in turn requires fewer training samples and lower training times.
This is especially true for the derivative error, which often required more training samples to learn accurately compared to the output itself.
In this regard, the DIS significantly outperformed the PCA for input dimension reduction, since it explicitly accounts for the sensitivity of the underlying map,
while the PCA and DIS provided similar performance when used for output dimension reduction for the examples considered.

We note that several of our results for dimension reduction required additional assumptions. For example, \Cref{assumption:derivative_inverse_inequality} and \Cref{assumption:hessian_inverse_inequality} were needed in order to bound higher-order derivatives by lower-order ones.
Although the error bounds resulting from these assumptions 
are realized in the numerical examples considered,
it would be of interest to identify more fundamental properties of operators that can be used to ascertain the reconstruction/ridge function bounds for the derivatives. 
One possible example is holomorphy. 
It is known that holomorphic operators are amenable to linear dimension reduction
\cite{LanthalerMishraKarniadakis22, Lanthaler23, SchwabZech23,AdcockDexterMoraga24, HerrmannSchwabZech24,AdcockDexterMoraga25}, possibly allowing one to derive favorable approximation rates in the infinite-dimensional Sobolev norms in terms of both the latent dimension $r$ and the size of the neural network.

We also note that neither the statistical learning errors nor the training (optimization) errors of the neural network are treated theoretically in this work.
However, both are important for the understanding of derivative-informed operator learning, since many performance benefits arise due to the inclusion of derivative information in the \textit{Sobolev training} process itself. 
Thus, we believe operator learning with derivatives remains open to further theoretical and numerical explorations, with many interesting outstanding research questions.

\section{Acknowledgements}
This research is supported by the U.S. Department of Energy under the award DE-SC0023171
and by the U.S. National Science Foundation under awards DMS-2324643
and OAC-2313033.
This research made use of computing resources of the National Energy Research Scientific Computing Center (NERSC), a Department of Energy Office of Science User Facility, under award ALCC-ERCAP0030671.
The authors would also like to acknowledge helpful discussions with Lianghao Cao and Jakob Zech.

\appendix
\crefalias{section}{appendix}

\newpage
\section*{Appendices}

\section{Useful supporting results}

\subsection{Cameron--Martin space}
We note some useful facts about the Cameron--Martin space. 
\begin{lemma} Suppose $v_i \in \cX$ are eigenvectors of $\Cx$ such that $\Cx v_i = \lambda_i v_i$, $i \in \bN$. 
  Then $\tilde{v}_i = \Cx^{\half} v_i$ is an orthonormal basis for the Cameron--Martin space $\XC$. 
\end{lemma}
\begin{proof}
    Consider any $h \in \XC$. Then $h = \Cx^{\half}w$ for some $w \in \cX$. Moreover, $(v_i)_{i=1}^{\infty}$ is an orthonormal basis for $\cX$ such that $\varepsilon_n := w - \sum_{i=1}^{n} v_i \linner v_i, w \rinner_{\cX} \rightarrow 0$. 
    Now note that 
    \begin{align*}
        \varepsilon_n &= \Cx^{-\half} h - \sum_{i=1}^{n} \Cx^{-\half} \Cx^{\half} v_i \linner \Cx^{-\half} \Cx^{\half} v_i, \Cx^{-\half} h \rinner_{\cX} \\
        &= \Cx^{-\half} \left( h - \sum_{i=1}^{n} \tilde{v}_i \linner \tilde{v}_i, h \rinner_{\XC} \right) \\
        &=: \Cx^{-\half} \tilde{\varepsilon}_n,
    \end{align*}
    where $\tilde{\varepsilon}_n = h - \sum_{i=1}^{n} \tilde{v}_i \linner \tilde{v}_i, h \rinner_{\XC}$ is complementary component of the projection onto the span of $(\tilde{v}_i)_{i=1}^{n}$ in $\XC$.
    It remains to see that $\|\tilde{\varepsilon}_n\|_{\XC}^2 = \linner \tilde{\varepsilon}_n, \tilde{\varepsilon}_n \rinner_{\XC} = \linner \varepsilon_n, \varepsilon_n \rinner_{\cX} \rightarrow 0$. 
\end{proof}

\begin{lemma}
    The inclusion operator $\iota : \cE \ni w \mapsto w \in \cX$ is Hilbert--Schmidt.
\end{lemma}
\begin{proof}
    We use the basis generated by the covariance $(\Cx^{\half} v_i)_{i=1}^{\infty} = (\lambda_i^{\half} v_i)_{i=1}^{\infty}$ to evaluate the Hilbert--Schmidt norm,
    \[\|\iota\|_{\HS(\XC, \cX)}^2 = \sum_{i=1}^{\infty} \linner \lambda_i^{\half} v_i, \lambda_i^{\half} v_i \rinner_{\cX} = \sum_{i=1}^{\infty} \lambda_i = \tr(\Cx) < \infty. \]
\end{proof}

\subsection{Convergence in Sobolev norms}
Here, we provide a supporting result for the convergence of functions in the Sobolev norm $H^1_{\gamma} = H^1(\gamma, \cY)$.
\begin{lemma}\label{lemma:convergence_of_sobolev}
    Let $\fun \in H^1_{\gamma}$ 
    and let $(\fun_n)_{i=1}^{\infty}$ be a sequence of 
    $H^1_{\gamma}$ functions such that 
    $\fun_n \rightarrow \fun$ in 
    $L^2(\gamma, \cY)$ 
    and $\derivXC \fun_n \rightarrow \Psi$ for some 
    $\Psi \in L^2(\gamma, \HS(\XC, \cY))$.
    Then, $\fun_n \rightarrow \fun$ in 
    $H^1_{\gamma}$
    and $\derivXC \fun = \Psi$. 
\end{lemma}
\begin{proof}
    Recall that for any sequence $(\varphi_n)_{n=1}^{\infty}$ in $\cF \cC^{\infty}$ that converges to $\fun$ in $L^2_{\gamma} := L^2(\gamma, \cY)$ 
    and is Cauchy in $H^1_{\gamma}$, we have $\varphi_n \rightarrow \fun$ in $H^1_{\gamma}$ (see Section 5.2 of \cite{Bogachev98}). 
    In particular, this means $\lim_{n \rightarrow \infty} \derivXC \varphi_n = \derivXC \fun$. 
    It remains to construct the sequence $\varphi_n \in \cF \cC^{\infty}$ using $\fun_n \in H^1_{\gamma}$. 
    For $k \in \bN$, we will use a subsequence $\fun_{n_k}$ such that $\| \fun_{n_k} - \fun \|_{L^2_{\gamma}} \leq 1/k$
    and $\| \derivXC \fun_{n_k} - \Psi \|_{L^2_{\gamma}} \leq 1/k$. 
    Moreover, by the density of $\cF \cC^{\infty}$ in $H^1_{\gamma}$, we can find $\varphi_{n_k}$ such that 
    $\|\varphi_{n_k} - \fun_{n_k} \|_{H^1_{\gamma}} \leq 1/k$. 
    Thus we have 
    \[ \|\fun - \varphi_{n_k}\|_{L^2_{\gamma}} \leq \|\fun - \fun_{n_k}\|_{L^2_{\gamma}} + \|\fun_{n_k} - \varphi_{n_k}\|_{L^2_{\gamma}}
    \leq 2/k. \]
    Moreover, 
    \[ \|\derivXC \varphi_{n_k} - \Psi \|_{L^2(\gamma, \HS(\XC,\cY))} \leq \|\derivXC \varphi_{n_k} - \derivXC \fun_{n_k} \|_{L^2(\gamma, \HS(\XC,\cY))} + \| \derivXC \fun_{n_k} - \Psi \|_{L^2(\gamma, \HS(\XC,\cY))} \leq 2/k. \]
    That is, $\varphi_{n_k} \rightarrow \fun$ in $L^2_{\gamma}$ and $\derivXC \varphi_{n_k}$ converges to $\Psi$ in ${L^2(\gamma, \HS(\XC,\cY))}$. 
    By the uniqueness of the limit, this means $\varphi_{n_k} \rightarrow \fun$ in $H^1_{\gamma}$ and $\derivXC \fun = \Psi$.
\end{proof}

\subsection{Trace minimization and eigenvalue problems}
In this section, we state Fan's theorem, which is a result on the maximization of quadratic forms over orthonormal sets. 
This is related to trace minimization and is used extensively for the reconstruction error results. 
\begin{lemma}[Fan's theorem, as stated in \cite{BhattacharyaHosseiniKovachkiEtAl21}]\label{lemma:fans_theorem}
  Let $(\cH, \linner \cdot, \cdot \rinner_{\cH})$ be a separable Hilbert space and $A : \cH \rightarrow \cH$ be a non-negative, self-adjoint, compact operator.
  Denote by $\lambda_1 \geq \lambda_2 \geq \dots$ the eigenvalues of $A$, and for any $r \in \bN$ (where $\bN = \{1, 2, \dots\}$), let $S_r$ the set of collections of $r$ orthonormal elements of $\cH$. Then, 
  \begin{equation}
      \max_{(u_i)_{i=1}^{r} \in S_r} \sum_{i=1}^{r} \linner A u_i, u_i \rinner_{\cH} = \sum_{i=1}^{r} \lambda_i .
  \end{equation}
\end{lemma}

Fan's theorem asserts that of all orthogonal projections, 
projections onto dominant eigenvectors of a non-negative, self-adjoint, compact operator $A$ 
minimizes the trace of $(I - \Pry) A (I - \Pry)$, which shows up in the analysis of the reconstruction errors.
This is summarized in the following lemma.
\begin{lemma}[Trace minimization]\label{lemma:trace_minimization}
  Let $(\cH, \linner \cdot, \cdot \rinner_{\cH})$ be a separable Hilbert space and $A : \cH \rightarrow \cH$ be a non-negative, self-adjoint, trace-class operator with eigenpairs $(\lambda_i, u_i)_{i=1}^{\infty}$ given in order of decreasing eigenvalues.
  Let also $\cP_r(\cH)$ denote the class of rank $r$ orthogonal projections in $\cH$. 
  Then,
  \begin{equation}
    \min_{P_r \in \cP_r(\cH)} \tr( (I - P_r) A(I - P_r)) = \sum_{i=r + 1}^{\infty} \lambda_i,
  \end{equation}
  and the minimum is attained by $P_r = \sum_{i=1}^{r} u_i \otimes u_i$.
\end{lemma}
\begin{proof}
  Note that for trace-class operator $A$ and $I - \widetilde{P}_r \in \cL(\cH, \cH)$, their composition is also trace-class. 
  By the cyclic property of the trace, we then have 
  \[ \tr((I - \widetilde{P}_r) A (I - \widetilde{P}_r)) = \tr(A (I - \widetilde{P}_r)) 
    = \tr(A) - \tr(\widetilde{P}_r A \widetilde{P}_r). \]
  For any $\widetilde{P}_r \in \cP_r(\cH)$, the projection can be written as $\widetilde{P}_r = \sum_{i=1}^{r} \tilde{u}_i \otimes \tilde{u}_i$ for some orthonormal basis of $\cH$, $(\tilde{u}_i)_{i=1}^{\infty}$, 
  where $(\tilde{u}_i)_{i=1}^{r}$ spans the range of $\widetilde{P}_r$.
  Then, using the same basis to compute the second trace, we have 
  \[ \tr(\widetilde{P}_r A \widetilde{P}_r) = \sum_{i=1}^{r} \linner \tilde{u}_i, A \tilde{u}_i \rinner_{\cH}, \]
  which by \Cref{lemma:fans_theorem}, has a maximal value of $\sum_{i=1}^{r} \lambda_i$. 
  Moreover, this is attained by $(u_i)_{i=1}^{r}$ coming from the eigenvectors of $A$ with the largest eigenvalues.
  Thus, the minimum of $\tr((I - \widetilde{P}_r) A (I - \widetilde{P}_r))$
  is given by
  \[ \tr(A) - \sum_{i=1}^{r}(u_i, A u_i) = \sum_{i=1}^{\infty} \lambda_i - \sum_{i=1}^{r} \lambda_i = \sum_{i=r + 1}^{\infty} \lambda_i, \]
  which is attained by the projection $P_r = \sum_{i=1}^{r} u_i \otimes u_i$.
\end{proof}

When the eigenvalue problem is solved for an approximation $\widehat{A}$ (e.g. an empirical estimate of $A$), 
the corresponding projection is sub-optimal. 
The excess error can be bounded by the difference between $A$ and $\widehat{A}$ as measured in the $\HS(\cH, \cH)$ norm. 
\begin{lemma}\label{lemma:projection_errors}
  Let $(\cH, \linner \cdot, \cdot \rinner_{\cH})$ be a separable Hilbert space.
  Suppose $A \in \cL(\cH, \cH)$ and 
  $\widehat{A} \in \cL(\cH, \cH)$ 
  are non-negative, self-adjoint, and trace-class,
  with eigenpairs 
  $(\lambda_i, u_i)_{i=1}^{\infty}$ 
  and
  $(\lambdahat_i, \uhat_i)_{i=1}^{\infty}$, respectively,
  both given in order of decreasing eigenvalues.
  Then, given $r \in \bN$, the projection $\widehat{P}_r = \sum_{i=1}^{r} \uhat_i \otimes \uhat_i$
  satisfies
  \begin{equation}\label{eq:projection_error_global}
    \tr(A (I - \widehat{P}_r)) \leq \sum_{i=r+1}^{\infty} \lambda_i 
    + \sqrt{2r} \|A - \widehat{A} \|_{\HS(\cH, \cH)}
  \end{equation}
  and 
  \begin{equation}\label{eq:projection_error_local}
    \tr(A (I - \widehat{P}_r)) \leq \sum_{i=r+1}^{\infty} \lambda_i 
    + \frac{ 2 \|A - \widehat{A} \|_{\HS(\cH,\cH)}^2}{\lambda_{r} - \lambda_{r+1}} .
  \end{equation}
\end{lemma}
\begin{proof}
  The proof follows that presented in \cite{ReissWahl20} in the context of covariance estimation, noting that it applies generally for self-adjoint trace-class operators.
  We reproduce the proof here for completeness. 
  We start by writing
  \[ \tr(A (I - \widehat{P}_r)) = \tr(A (I - P_r)) + \tr(A (P_r - \widehat{P}_r)) = \sum_{i=r+1}^{\infty} \lambda_i + \tr(A (P_r - \widehat{P}_r)),\]
  where $P_r = \sum_{i=1}^{r} u_i \otimes u_i$ comes from the exact optimal projection.
  We now seek to bound the excess error 
  \begin{equation}
    \mathscr{E}_r = \tr(A(P_r - \widehat{P}_r)),
  \end{equation}
  often referred to as the excess risk in covariance estimation.
  Since $\widehat{P}_r$ maximizes $\tr(\widehat{A} \widehat{P}_r)$ amongst rank $r$ orthogonal projections,
  we know $\tr(\widehat{A} \widehat{P}_r) \leq \tr(\widehat{A} P_r)$
  and so $ \tr(\widehat{A}(P_r - \widehat{P}_r)) \leq 0.$
  Subtracting this from the excess error yields
  \[ 
  \mathscr{E}_r = \tr(A (P_r - \widehat{P}_r)) \leq 
    \tr((A - \widehat{A}) (P_r - \widehat{P}_r)).\]
  Using the Cauchy--Schwarz inequality, we have 
  \begin{equation}\label{eq:excess_risk_cauchy_schwarz}
  \mathscr{E}_r \leq \tr((A - \widehat{A}) (P_r - \widehat{P}_r)) = \linner A - \widehat{A}, P_r - \widehat{P}_r \rinner_{\HS(\cH, \cH)} \leq \|A - \widehat{A}\|_{\HS(\cH, \cH)} \|P_r - \widehat{P}_r \|_{\HS(\cH, \cH)}.
  \end{equation}
  Expanding the projection error gives
  \[
    \|P_r - \widehat{P}_r \|_{\HS(\cH, \cH)}^2 = \|P_r\|_{\HS(\cH, \cH)}^2 +  \|\widehat{P}_r\|_{\HS(\cH, \cH)}^2
    - 2 \sum_{i=1}^{\infty} \linner P_r e_i, \widehat{P}_r e_i \rinner_{\cH}.
  \]
  where $\|P_r\|_{\HS(\cH, \cH)}^2 = r$ and $\|\widehat{P}_r\|_{\HS(\cH, \cH)}^2 = r$. 
  Moreover, using $e_i = u_i$, 
  \[ \linner P_r u_i, \widehat{P}_r u_i \rinner_{\cH} = \linner u_i, \widehat{P}_r u_i \rinner_{\cH} \geq 0. \]
  Thus, $\|P_r - \widehat{P}_r \|_{\HS(\cH, \cH)} \leq \sqrt{2r}$, giving \eqref{eq:projection_error_global}.

  For \eqref{eq:projection_error_local}, we first note that the bound is trivially infinite if $\lambda_{r+1} = \lambda_r$,
  so we proceed assuming $\lambda_{r} > \lambda_{r+1}$.
  We claim that the excess error upper bounds the projection error as
  \begin{equation}\label{eq:projection_error_claim}
    \|P_r - \widehat{P}_r \|_{\HS(\cH, \cH)}^2 \leq \frac{2 \mathscr{E}_r}{\lambda_{r} - \lambda_{r+1}}.
  \end{equation}
  Combining this claim with \eqref{eq:excess_risk_cauchy_schwarz}, we have 
  \[
    \mathscr{E}_{r}^2 \leq \|A - \widehat{A} \|_{\HS(\cH,\cH)}^2 \| P_r - \widehat{P}_r\|_{\HS(\cH, \cH)}^2
    \leq \frac{2 \mathscr{E}_r}{\lambda_{r} - \lambda_{r+1}} \|A - \widehat{A} \|_{\HS(\cH,\cH)}^2.
  \]
  Dividing through by $\mathscr{E}_r$ gives \eqref{eq:projection_error_local}.

  It remains then to prove the claim \eqref{eq:projection_error_claim}.
  To this end, we begin by introducing the notation $A_i = u_i \otimes u_i$, $P_r^{\perp} = I - P_r$, and $\widehat{P}_r^{\perp} = I - \widehat{P}_r$.
  We also have the identities 
  $P_r - \widehat{P}_r = \widehat{P}_r^{\perp} - P_r^{\perp}$ 
  and
  $
  \tr(P_r \widehat{P}_r^{\perp})
    = \tr( \widehat{P}_r  {P}_r^{\perp}).
  $
  This latter identity follows from the fact 
  \[
  \tr(P_r \widehat{P}_r^{\perp})
  = \tr( P_r (\widehat{P}_r^{\perp} - P_r^{\perp}) ) 
  = \tr( -P_r (\widehat{P}_r- P_r) ) 
  = \tr( (I-P_r) (\widehat{P}_r- P_r) ) 
  = \tr( P_r^{\perp} \widehat{P}_r) ,
  \]
  since $\tr( \widehat{P}_r- P_r ) = 0$.
  Now, expanding the norm $\| P_r - \widehat{P}_r \|_{\HS(\cH,\cH)}^2$, we have
  \begin{align*}
    \|P_r - \widehat{P}_r \|_{\HS(\cH, \cH)}^2 &=  \tr((P_r - \widehat{P}_r)(P_r - \widehat{P}_r))
    = \tr((P_r - \widehat{P}_r)(\widehat{P}_r^{\perp} - P_r^{\perp}))
    = 2 \tr(P_r \widehat{P}_r^{\perp})
  \end{align*}
  where we have used the fact that $\tr(P_r P_r^{\perp}) = \tr(\widehat{P}_r \widehat{P}_r^{\perp}) = 0$ and 
  $\tr(P_r \widehat{P}_r^{\perp})
  = \tr( \widehat{P}_r  {P}_r^{\perp}).$
  Moreover, since 
  $\tr(P_r \widehat{P}_r^{\perp}) 
  = \tr(P_r \widehat{P}_r^{\perp} \widehat{P}_r^{\perp} P_r)
  = \|P_r \widehat{P}_r^{\perp}\|_{\HS(\cH,\cH)}^2$
  , we can further expand this as
  \begin{align} \label{eq:projection_error_expansion}
    \|P_r - \widehat{P}_r \|_{\HS(\cH, \cH)}^2 
    = 2 \| P_r \widehat{P}_r^{\perp} \|_{\HS(\cH, \cH)}^2 
    = 2 \sum_{i=1}^{r} \| A_i \widehat{P}_r^{\perp} \|_{\HS(\cH, \cH)}^2.
  \end{align}

  On the other hand, the excess error is given by 
  \begin{align*}
    \mathscr{E}_r 
    &= \sum_{i=1}^{\infty} \lambda_i \linner A_i, P_r - \widehat{P}_r \rinner_{\HS(\cH, \cH)} \\
    &= \sum_{i=1}^{r} \lambda_i \linner A_i, \widehat{P}_r^{\perp} - {P}_r^{\perp} \rinner_{\HS(\cH, \cH)}
     - \sum_{j=r+1}^{\infty} \lambda_j \linner A_j, \widehat{P}_r - P_r \rinner_{\HS(\cH, \cH)} \\
    &= \sum_{i=1}^{r} \lambda_i \tr(A_i \widehat{P}_r^{\perp})
     - \sum_{j=r+1}^{\infty} \lambda_j \tr( A_j \widehat{P}_r ).
  \end{align*}
  Moreover, we have 
  $
  \beta \tr(P_r \widehat{P}_r^{\perp})
    - \beta \tr( \widehat{P}_r  {P}_r^{\perp}) = 0
  $ 
  for any $\beta \in \bR$. Thus, subtracting this from the sum yields 
  \begin{align} \nonumber
    \mathscr{E}_r
    &= \sum_{i=1}^{r} (\lambda_i - \beta) \tr(A_i \widehat{P}_r^{\perp})
     + \sum_{j=r+1}^{\infty} (\beta - \lambda_j) \tr( A_j \widehat{P}_r ) \\
    &= \sum_{i=1}^{r} (\lambda_i - \beta) \|A_i \widehat{P}_r^{\perp} \|_{\HS(\cH, \cH)}^2
     + \sum_{j=r+1}^{\infty} (\beta - \lambda_j) \|A_j \widehat{P}_r \|_{\HS(\cH, \cH)}^2
     \nonumber
  \end{align}
  Taking $\beta  = \lambda_{r+1}$
  and comparing the resulting expression with \eqref{eq:projection_error_expansion}, we see that 
  \begin{align*}
    \frac{2 \mathscr{E}_r}{\lambda_{r} - \lambda_{r+1}}  
    \geq 2\sum_{i=1}^{r} \frac{\lambda_i - \lambda_r}{\lambda_{r} - \lambda_{r+1}}\|A_i \widehat{P}_r^{\perp}\|_{\HS(\cH, \cH)}^2
    \geq 2 \sum_{i=1}^{r} \|A_i \widehat{P}_r^{\perp}\|_{\HS(\cH, \cH)}^2 = \|P_r - \widehat{P}_r\|_{\HS(\cH, \cH)}^2,
  \end{align*}
  since $\lambda_i - \lambda_{r+1} \geq \lambda_{r} - \lambda_{r+1} \geq 0$ for $i \leq r$ and $\lambda_{r+1} - \lambda_{j} \geq 0$ for $j \geq r+1$.
  This proves the claim.


\end{proof}

\subsection{Interchanging differentiation and integration} \label{sec:interchange_differentiation_integration}
For the analysis of the $H^1_{\gamma}$ approximation properties of conditional expectations, 
we need to consider the interchange of (Sobolev) differentiation and (conditional) expectation. 
We first look at the sufficient conditions for interchanging Fr\'echet differentiation and integration.
A general result is difficult to find, so we state and prove a version that is sufficient for our purposes.

\begin{theorem}[Interchanging Fr\'echet differentiation and integration] \label{theorem:leibniz}
  Let $\cX_1$, $\cX_2$ and $\cY$ be separable Banach spaces, $\fun : \cX_1 \times \cX_2 \rightarrow \cY$ 
  be a continuously (Fr\'echet) differentiable function, and $\nu$ be a probability measure on $\cX_2$. 
  Suppose that $\fun$ additionally satisfies the following assumptions
  \begin{enumerate}
    \item The function $\fun(x_1, \cdot) : \cX_2 \rightarrow \cY$ 
      is integrable with respect to $\nu$ for every $x_1 \in \cX_1$ 
    \item The partial derivative $\partial_{x_1} \fun(x_1, \cdot) : \cX_2 \rightarrow \cL(\cX_1, \cY)$ 
      is integrable with respect to $\nu$ for every $x_1 \in \cX_1$
    \item There exists an open set $\cU \subset \cX_1$ and a function $g \in L^1(\cX_2, \nu; \bR)$
    such that 
    \begin{equation}\label{eq:interchange_integrability_condition} 
      \| \partial_{x_1} \fun(x_1, x_2) \|_{\cL(\cX_1, \cY)} \leq g(x_2)
    \end{equation} 
    for every $x_1 \in \cU$ and $x_2 \in \cX_2$.
  \end{enumerate}
  Then, for all $x_1 \in \cU$, 
  \begin{equation}
    \deriv_{x_1} \bE_{x_2 \sim \nu}[\fun(x_1, x_2)]  
    = \bE_{x_2 \sim \nu}[\partial_{x_1} \fun(x_1, x_2)].
  \end{equation}
\end{theorem}
\begin{proof}
  Let $\fun_1(x_1) := \bE_{x_2 \sim \nu}[\fun(x_1, x_2)]$ and $\Psi (x_1) := \bE_{x_2 \sim \nu}[\partial_{x_1} \fun(x_1, x_2)]$,
  where both are well-defined by the integrability assumptions.
  In particular, at any $x_1 \in \cX_1$, $\Psi(x_1)$ is a linear operator on $\cX_1$, and for any $h \in \cX_1$, 
  \[ \| \Psi(x_1) h \| = \| \bE_{x_2 \sim \nu}[\partial_{x_1} \fun(x_1, x_2) h]  \|
  \leq \bE_{x_2 \sim \nu}[ \|\partial_{x_1} \fun(x_1, x_2)\|_{\cL(\cX_1, \cY)} ]\|h\|. \]
  Here, we are using the shorthand $\| \cdot \|$ to denote the implied Banach space norm.
  Thus, integrability implies 
  $ \bE_{x_2 \sim \nu}[ \|\partial_{x_1} \fun(x_1, x_2)\|_{\cL(\cX_1, \cY)} ] < \infty,$ 
  such that $\Psi(x_1) \in \cL(\cX_1, \cY)$ for any $x_1 \in \cX_1$.

  Consider now the residual for $h \in \cX_1$,
  \[
    \frac{\|R(x_1, h)\|}{\|h\|} := \frac{\| \fun_1(x_1 + h) - \fun_1(x_1) - \Psi(x_1) h \|}{\|h\|},
  \]
  as $h \rightarrow 0$. 
  Combining the terms inside the expectation, we have
  \[
    \frac{\|R(x_1, h)\|}{\|h\|} =
    \frac{\| \bE_{x_2 \sim \nu}[ \fun(x_1 + h, x_2) - \fun(x_1, x_2) - \partial_{x_1} \fun(x_1, x_2) h ] \|}{\|h\|}.
  \]
  Again, by Bochner integrability, we can bring the norm inside the expectation
  \[
    \frac{\|R(x_1, h)\|}{\|h\|} 
    \leq  \bE_{x_2 \sim \nu} \left[ \frac{\| \fun(x_1 + h, x_2) - \fun(x_1, x_2) - \partial_{x_1} \fun(x_1, x_2) h \|}{\|h\|} \right].
  \]
  Since $\fun$ is differentiable, the integrand
  \[
    \frac{\|R_1(x_1, x_2, h)\|}{\|h\|} :=  \frac{\| \fun(x_1 + h, x_2) - \fun(x_1, x_2) - \partial_{x_1} \fun(x_1, x_2) h \|}{\|h\|} \rightarrow 0
  \]
  as $h \rightarrow 0$ for every $x_1 \in \cX_1$ and $x_2 \in \cX_2$. 
  It remains to bound this by an $L^1$ function to apply the dominated convergence theorem. 
  Applying the triangle inequality and the mean value theorem gives 
  \[
    \frac{\|R_1(x_1, x_2, h)\|}{\|h\|} \leq 
    \frac{\sup_{t \in [0,1]}\|\partial_{x_1}(\fun(x_1+th, x_2)) \| \|h\| +  \|\partial_{x_1} \fun(x_1, x_2)\| \|h\|}{\|h\|}. 
  \]
  For $x_1 \in \cU$ and $h$ sufficiently small, we have $x_1 + th \in \cU$ for all $t \in [0,1]$. 
  This implies
  $\sup_{t \in [0,1]}\|\partial_{x_1}(\fun(x_1+th, x_2)) \| \leq g(x_2)$. 
  Thus, 
  \[
    \frac{\|R_1(x_1, x_2, h)\|}{\|h\|} \leq 2 g(x_2).
  \]
  Applying the dominant convergence theorem thus gives $\|R(x_1,h)\|/\|h\| \rightarrow 0$, 
  and so $\deriv \fun_1(x_1) = \Psi (x) = \bE_{x_2 \sim \nu}[\deriv_{x_1} \fun(x_1, x_2)]$.
\end{proof}
Using \Cref{theorem:leibniz}, we can proceed to consider conditional expectations for functions 
in $\cF \cC^{\infty}$.

\begin{lemma}[Interchanging Fr\'echet differentiation and conditional expectation]\label{lemma:interchange_differentiation_expectation_smooth}
  Suppose $\fun \in \cF \cC^{\infty}$
  and let $(v_i)_{i=1}^{\infty}$ be an orthonormal basis of the Cameron--Martin space $\XC$ for which $(\Cx^{-1} v_i)_{i=1}^{\infty}$ are well-defined in $\cX$. 
  Then, for the reduced basis $\Vrx = (v_i)_{i=1}^{r}$ and its corresponding projection $\Qrx = \Vrx \pinv{\Vrx}$, the conditional expectation $\bE_{\gamma}[\fun | \sigma(\pinv{\Vrx})] \in C^{\infty}_{b}(\cX, \cY)$, 
  and its derivatives are given by 
  \[ 
    \deriv^k \bE_{\gamma}[\fun | \sigma(\pinv{\Vrx})] = \bE_{\gamma}[\deriv^{k} \fun(\Qrx \cdot, \dots, \Qrx \cdot) | \sigma(\pinv{\Vrx})],
  \]
  where $\deriv^{k} \fun(x)(\Qrx \cdot, \dots, \Qrx \cdot) : (h_1, \dots, h_k) \mapsto \deriv^{k} \fun(x)(\Qrx h_1, \dots, \Qrx h_k)$. 
\end{lemma}
\begin{proof}
  We start with the first derivative. The conditional expectation is given by 
  \[ \bE_{\gamma}[\fun | \sigma(\pinv{\Vrx})](x) = \int \fun(\Qrx x + (I - \Qrx)y) \gamma(dy). \]
  First, since the derivatives of $\fun \in \cF \cC^{\infty}$ are continuous and of finite rank, the mappings $x \mapsto \deriv^k \fun(x)$ take values in a separable subspace of $\cL_k(\cX, \cY)$ and hence are Bochner measurable.
  By boundedness of $\fun$ and its derivatives, the function $(x, y) \mapsto \fun(\Qrx x + (I - \Qrx)y)$ satisfies the hypotheses of \Cref{theorem:leibniz}. 
  We can therefore interchange differentiation and integration to obtain
  \[ \deriv \bE_{\gamma}[\fun | \sigma(\pinv{\Vrx})](x) = \int \deriv \fun(\Qrx x + (I - \Qrx)y) \Qrx \gamma(dy) = \bE_{\gamma}[\deriv \fun \Qrx| \sigma(\pinv{\Vrx})]. \]
  Let $L_k$ denote the bound on $\|\deriv^k \fun (x)\|_{\cL_k(\cX, \cY)}$. To see boundedness of the derivative, we have 
  \[ 
    \|\deriv \bE_{\gamma}[\fun | \sigma(\pinv{\Vrx})](x)\|_{\cL(\cX, \cY)}
    \leq \int \|\deriv \fun(\Qrx x + (I - \Qrx)y) \Qrx \|_{\cL(\cX,\cY)} \gamma(dy) 
    \leq L_1.
  \]
  On the other hand, to see continuity, we have 
  \begin{align*}
    & \|\deriv \bE_{\gamma}[\fun | \sigma(\pinv{\Vrx})](x_2) -  \deriv \bE_{\gamma}[\fun | \sigma(\pinv{\Vrx})(x_1)]\|_{\cL(\cX,\cY)}  \\
    & \qquad \leq \int \|\deriv \fun(\Qrx x_2 + (I - \Qrx)y) \Qrx - \deriv \fun(\Qrx x_1 + (I - \Qrx)y)\Qrx \|_{\cL(\cX, \cY)} \gamma(dy) \\
    & \qquad \leq \int \|\deriv \fun(\Qrx x_2 + (I - \Qrx)y) - \deriv \fun(\Qrx x_1 + (I - \Qrx)y)\|_{\cL(\cX, \cY)} \gamma(dy) \\
    & \qquad \leq \int L_2 \|x_2 - x_1\|_{\cX} \; \gamma(dy) \\
    & \qquad = L_2 \|x_2 - x_1\|_{\cX},
  \end{align*}
  where the last inequality follows from the mean value theorem and boundedness of the second derivative. 
  Higher-order derivatives follow by induction, 
  starting from the hypothesis that 
  \[D^k \bE_{\gamma}[\fun | \sigma(\pinv{\Vrx})] =
  \bE_{\gamma}[\deriv^{k} \fun(\Qrx \cdot, \dots, \Qrx \cdot) | \sigma(\pinv{\Vrx})]\]
  for some $k \geq 1$. 
  Then, we can view $x \mapsto\deriv^{k} \fun(x)(\Qrx \cdot, \dots, \Qrx \cdot)$ as another separably-valued map with bounded derivatives, 
  for which the same arguments can be used to show
  \[D \bE_{\gamma}[\deriv^{k} \fun(\Qrx \cdot, \dots, \Qrx \cdot) | \sigma(\pinv{\Vrx})] = \bE_{\gamma}[\deriv^{k+1} \fun(\Qrx \cdot, \dots, \Qrx \cdot) | \sigma(\pinv{\Vrx})].\]
  Thus, $\bE_{\gamma}[\fun | \sigma(\pinv{\Vrx})]$ has continuous and bounded $k$th order derivatives for all $k \in \bN$.

\end{proof}

Finally, the result can be extended to Sobolev differentiation of functions in $H^1_{\gamma}$ through a density argument, giving us \Cref{lemma:interchange_differentiation_conditional_expectation}.
\begin{proof}(of \Cref{lemma:interchange_differentiation_conditional_expectation})
    By the density argument, we can find a sequence of $\cF \cC^{\infty}$ functions, $(\fun_n)_{i=1}^{\infty}$ convergent to $\fun$ in $H^1_{\gamma}$. 
    This also means $\bE_{\gamma}[\fun_n | \sigma(\pinv{\Vrx})]$ converges to $\bE_{\gamma}[\fun | \sigma(\pinv{\Vrx})]$ in $L^2_{\gamma}$,
    since 
    \begin{align*} 
    & \| \bE_{\gamma}[\fun_n | \sigma(\pinv{\Vrx})] - \bE_{\gamma}[\fun | \sigma(\pinv{\Vrx})] \|_{L^2_{\gamma}}^2 \\
    &\qquad =  \int \left\| \int \fun_n(\Qrx x + (I - \Qrx)y) - \fun(\Qrx x + (I - \Qrx)y ) \gamma(dy) \right\|^2 \gamma(dx) \\
    &\qquad \leq  \iint \| \fun_n(\Qrx x + (I - \Qrx)y) - \fun(\Qrx x + (I - \Qrx)y ) \|^2 \gamma(dy) \gamma(dx) \\
    &\qquad = \| \fun_n - \fun \|_{L^2_{\gamma}}^2. 
    \end{align*}

    Since $\fun_n \in \cF \cC^{\infty}$ we can interchange the integration and Fr\'echet differentiation as in \Cref{lemma:interchange_differentiation_expectation_smooth}, such that
    $\derivXC \bE_{\gamma}[\fun_n | \sigma(\pinv{\Vrx})] = \bE_{\gamma}[\derivXC \fun_n \Qrx | \sigma(\pinv{\Vrx})]$.
    Moreover, we can verify $\bE_{\gamma}[\derivXC \fun_n \Qrx | \sigma(\pinv{\Vrx})]$ converges to $\bE_{\gamma}[\derivXC \fun \Qrx | \sigma(\pinv{\Vrx})]$, since 
    \begin{align*} 
        \| \bE_{\gamma}[& \derivXC \fun_n \Qrx | \sigma(\pinv{\Vrx})] - \bE_{\gamma}[\derivXC \fun \Qrx | \sigma(\pinv{\Vrx})] \|_{L^2_{\gamma}}^2 \\
        &=  \int \left\| \int \derivXC \fun_n(\Qrx x + (I - \Qrx)y) - \derivXC \fun(\Qrx x + (I - \Qrx)y)\Qrx  \gamma(dy) \right\|_{\HS(\XC, \cY)}^2 \gamma(dx) \\
        &\leq  \iint \| \derivXC \fun_n(\Qrx x + (I - \Qrx)y) - \derivXC \fun(\Qrx x + (I - \Qrx)y ) \|_{\HS(\XC, \cY)}^2 
            \|\Qrx \|_{\cL(\XC, \XC)}^2 \gamma(dy) \gamma(dx) \\
        &= \| \derivXC \fun_n - \derivXC \fun \|_{L^2_{\gamma}}^2. 
    \end{align*}
    Applying \Cref{lemma:convergence_of_sobolev}, we obtain $\derivXC \bE_{\gamma}[\fun | \sigma(\pinv{\Vrx})] = \bE_{\gamma}[\derivXC \fun \Qrx | \sigma(\pinv{\Vrx})]$.
\end{proof}

A corollary of this result is that the conditional expectation preserves the 
$H^m_{\gamma}$ norm.
\begin{corollary}\label{cor:conditional_expectation_preserves_norm}
  Suppose $\fun \in H^m_{\gamma}$
  and let $(v_i)_{i=1}^{\infty}$ be an orthonormal basis of the Cameron--Martin space $\XC$ for which $(\Cx^{-1} v_i)_{i=1}^{\infty}$ are well-defined in $\cX$. 
  Then, for the reduced basis $\Vrx = (v_i)_{i=1}^{r}$, we have $\| \bE_{\gamma}[ \fun | \sigma(\pinv{\Vrx})]\|_{H^m_\gamma} \leq  \| \fun \|_{H^m_\gamma}$.
\end{corollary}
\begin{proof}
  As previously shown, the conditional expectation preserves the $L^2_{\gamma}$ norm. 
  Moreover, repeatedly applying the previous lemma yields 
  \[ 
    \derivXC^k \bE_{\gamma}[\fun | \sigma(\pinv{\Vrx})] = 
    \bE_{\gamma}[\derivXC^k \fun(\Qrx(\cdot), \dots, \Qrx(\cdot)) | \sigma(\pinv{\Vrx})].
  \]
  Thus, applying the result for $L^2(\gamma, \HS_k(\XC, \cY))$ yields
  \begin{align*}
    \|\derivXC^k \bE_{\gamma}[\fun | \sigma(\pinv{\Vrx})]\|_{L^2(\gamma,\HS_k(\XC, \cY))}^2 
      & = \|\bE_{\gamma}[\derivXC^k \fun(\Qrx(\cdot), \dots, \Qrx(\cdot)) | \sigma(\pinv{\Vrx})] \|_{L^2(\gamma,\HS_k(\XC, \cY))}^2  \\
    & \leq 
      \|\derivXC^k \fun(\Qrx(\cdot), \dots, \Qrx(\cdot)) \|_{L^2(\gamma,\HS_k(\XC, \cY))}^2  \\
    & \leq 
      \|\derivXC^k \fun\|_{L^2(\gamma,,\HS_k(\XC, \cY))}^2,
  \end{align*}
  which holds for any $k \leq m$. 
\end{proof}

\subsection{Example: Polynomial forms}\label{sec:example_polynomial_forms}
In this section, we show that $n$th order polynomial forms on $\cX$ satisfy the derivative inverse inequality (\Cref{assumption:derivative_inverse_inequality}) and the Hessian inverse inequality (\Cref{assumption:hessian_inverse_inequality}) with constants $K_D = n$ and $K_H = n - 1$.
Let $H_n : \bR \rightarrow \bR$ denote the $n$th Hermite polynomial, defined for $n = 0, 1, \dots$ by
\begin{equation}
  H_n(x) = \frac{(-1)^{n}}{\sqrt{n!}} \exp \left( \frac{x^2}{2}\right) \frac{d^n}{dx^n} \exp \left( - \frac{x^2}{2} \right).
\end{equation}
With this choice of normalization, the Hermite polynomials form an orthonormal basis of $L^2(\gamma_1)$ where $\gamma_1$ is the standard Gaussian measure on $\bR$.
Moreover, one has $H_n'(x) = \sqrt{n} H_{n-1}(x)$, which implies that the derivatives of Hermite polynomials are also orthogonal under the $L^2(\gamma_1)$ inner product.

We can also define multivariate Hermite polynomials as tensor product of univariate Hermite polynomials. 
To this end, for a multi-index $\alpha = (\alpha_1, \alpha_2, \dots)$, $\alpha_i \in \bN \cup \{0 \}$, 
let us define $|\alpha| = \sum_{i=1}^{\infty} \alpha_i$ and $\supp \alpha = \{ i \in \bN: \alpha_i \neq 0 \}$.
Then, for $\alpha \in \Lambda := \{ \alpha : |\alpha| < \infty\}$, we define
\begin{equation}
  H_{\alpha}(x) = \prod_{i \in \supp \alpha} H_{\alpha_i}(x_i).
\end{equation}

In the case of separable infinite-dimensional Hilbert spaces $\cX$ and $\cY$, 
the family of Hermite polynomials $H_{\alpha}$, $\alpha \in \Lambda$, along with an orthonormal basis of the Cameron--Martin space $\XC$, can be used to construct an orthonormal basis of $L^2_{\gamma} = L^2(\gamma, \cY)$.
For example, given the PCA basis $(v_i^{\KLE})_{i=1}^{\infty}$ (normalized in $\XC$),
one defines the basis functions $H_{\gamma, \alpha} : \cX \rightarrow \bR$, 
\begin{equation}
  H_{\gamma, \alpha}(x) := \prod_{i \in \supp \alpha} H_{\alpha} \left( 
    \linner x, \Cx^{-1} v_i^{\KLE} \rinner_{\cX}\right).
\end{equation}
Any function $\fun \in L^2_{\gamma}$ can then be represented as
\begin{equation}
  \fun(x) = \sum_{\alpha \in \Lambda} c_{\alpha} H_{\gamma, \alpha}(x)
\end{equation}
with coefficients $c_{\alpha} \in \cY$, 
which is known as the Wiener Chaos expansion.
Moreover, 
\begin{equation}
  \| \fun \|_{L^2_{\gamma}}^2 = \sum_{\alpha \in \Lambda} \| c_{\alpha} \|_{\cY}^2.
\end{equation}
Note that the derivatives of $H_{\gamma, \alpha}$ along $v_k^{\KLE}$, which are orthonormal in $\XC$, are given by
\begin{equation}
\derivXC H_{\gamma, \alpha}(x) v_k^{\KLE} 
  = \sqrt{\alpha_k} H_{\gamma, \alpha - e_k}(x),
\end{equation}
where the multi-index $\alpha - e_k$ is defined as
\begin{equation}
  (\alpha - e_k)_j = \begin{cases}
    \max(\alpha_j - 1,0) & j = k, \\
    \alpha_j & j \neq k.
  \end{cases}
\end{equation}
This means that 
\begin{equation}\label{eq:derivative_inverse_hermite}
  \| \derivXC H_{\gamma, \alpha}(x) \|_{L^2(\gamma, \HS(\XC,\cY))}^2 
  = \sum_{k=1}^{\infty}
  \alpha_k \int H_{\gamma, \alpha - e_k}(x)^2 \gamma(dx)
  = |\alpha| \|H_{\gamma, \alpha}\|_{L^2_{\gamma}}^2
  = |\alpha|
\end{equation}
We refer to \cite[Chapter 1]{Nualart06} for a more complete exposition.
\begin{proposition}\label{prop:hermite_derivative_and_hessian}
  Let $\fun \in L^2_{\gamma}$ be represented by Hermite polynomials up to degree $n$, 
  \[
    \fun(x) = \sum_{|\alpha| \leq n} c_{\alpha} H_{\gamma, \alpha}(x).
  \]
  Then, $\fun$ satisfies the derivative inverse inequality \eqref{assumption:derivative_inverse_inequality} with constant $K_D = n$ and the Hessian inverse inequality \eqref{eq:hessian_inverse_inequality} with constant $K_H = n - 1$.
\end{proposition}
\begin{proof}
  We first note that since the Wiener Chaos expansion for $\fun$ is of finite polynomial order, $\fun \in H^k_{\gamma}$ for any $k \geq 0$.
  We start by considering the derivative inverse inequality. 
  For $\fun = \sum_{|\alpha| \leq n} c_{\alpha} H_{\gamma, \alpha}$,
  its mean shifted form is given by
  $\fun - \fbar = \sum_{1 \leq |\alpha| \leq n} c_{\alpha} H_{\gamma, \alpha}$
  Then, given any $u \in \cY$, 
  we consider the function $u^* \circ (\fun - \fbar)$, where $u^* \in \cY'$ is defined as $u^*(y) = \linner y, u \rinner_{\cY}$.
  This has the representation 
  $u^* \circ (\fun - \fbar) = \sum_{1 \leq |\alpha| \leq n} \linner u, c_{\alpha}\rinner_{\cY} H_{\gamma, \alpha}$.
  Now,
  \[
    \linner u, H_{\cY} u \rinner_{\cH}
    = \bE_{\gamma}[\| \derivXC (u^* \circ \fun) \|_{\HS(\XC, \cY)}^2]
    = \bE_{\gamma}[\| \derivXC (u^* \circ (\fun - \fbar)) \|_{\HS(\XC, \cY)}^2].
  \]
  Analogous to \eqref{eq:derivative_inverse_hermite},
  \cite[Proposition 1.2.2]{Nualart06} gives
  \[
    \bE_{\gamma}[\| \derivXC (u^* \circ (\fun - \fbar)) \|_{\HS(\XC, \cY)}^2]
    \leq n \| u^* \circ (\fun - \fbar) \|_{L^2(\gamma)}^2 
    = n \linner u, \Cy u \rinner_{\cY},
  \]
  which shows that $K_D = n$.

  The Hessian inverse inequality is shown similarly. 
  We first note that the $\derivXC \fun$ has a Wiener chaos expansion of degree $n-1$,
  \[
    \derivXC \fun = \sum_{|\alpha| \leq n -1} A_{\alpha} H_{\gamma, \alpha},
  \]
  where $A_{\alpha} \in \HS(\XC, \cY)$.
  Moreover, let $\gfun := \derivXC \fun v = \sum_{|\alpha| \leq n - 1} A_{\alpha}v H_{\gamma, \alpha} $ 
  for which 
  \cite[Proposition 1.2.2]{Nualart06} again gives
  \[
    \bE_{\gamma}[ \|\derivXC (\derivXC \fun v)\|_{\HS(\XC, \cY)}^2]
    = \bE_{\gamma}[ \|\derivXC \gfun\|_{\HS(\XC, \cY)}^2]
    \leq 
    (n-1) \bE_{\gamma}[ \|\gfun\|_{\cY}^2]
    = (n-1)\bE_{\gamma}[\| \derivXC \fun v\|_{\cY}^2].
  \]
  Thus, $K_H = n-1$.
\end{proof}

\section{Universal approximation theorem in Sobolev norms}\label{sec:nn_approximation}
\subsection{Original theorem of Hornik} 
\label{sec:dibasis_hornik}
Hornik \cite{Hornik91} proved a universal approximation of functions using a single-layered, arbitrary-width neural network in any weighted Sobolev norm on finite-dimensional input and output spaces with a finite measure $\nu$. 
For $f \in C^{m}(\bR^{d}, \bR)$, this Sobolev norm is
defined in \cite{Hornik91} as
\begin{equation}
  \| f \|_{W^{m,p}(\nu, \bR)} = \left( \int_{\bR^{d}} |f (x)|^p d\nu(x) + 
  \sum_{1 \leq |\alpha|_1\leq m} \int_{\bR^{d}} |\partial^{\alpha} f(x)|^p d\nu(x) \right)^{\frac{1}{p}},
\end{equation}
where we are using the multiindex notation for derivatives, with $\alpha = (\alpha_1, \dots, \alpha_d)$ and
\[\partial^{\alpha} f(x) = \frac{\partial^{|\alpha|} f}{\partial x_1^{\alpha_1} \dots \partial x_d^{\alpha_d}}(x).\]
Note that for $p = 2$ this corresponds to the use of Frobenius norms on the derivatives $D^k f$, $k \leq m$.

The original result \cite[Theorem 4]{Hornik91}
considers activation functions that are non-constant with bounded value and derivatives up to some order $m \geq 0$,
and states that single-layered neural networks with such activation functions are 
universal approximators of the space 
\[ C^{m,p}_{\nu}(\bR^{\din}) := \{\fun \in C^{m}(\bR^{\din}) : \|g \|_{W^{m,p}(\nu, \bR)} < \infty \}.
\]
We paraphrase this result as follows.
\begin{theorem}[Universal approximation with smooth, bounded activation functions]\label{theorem:universal_approx_hornik}
  Suppose $\nu$ is a finite measure on $\bR^{\din}$ and 
  $ g \in C^{m, p}_{\nu}(\bR^{\din})$ for some $m \geq 0$, $p \geq 1$.
  Let $\psi \in C_b^m(\bR)$ be a non-constant activation function.
  Then, for any $\epsilon > 0$, 
  there exists a single-layered neural network $\varphi_{\psi, \theta}$ with activation function $\psi$ 
  and some width $d_W \in \bN$ such that 
  \begin{equation}
    \|  g - \varphi_{\psi, \theta} \|_{W^{m,p}(\nu, \bR)} \leq \epsilon,
  \end{equation}
  where
  \[\varphi_{\psi, \theta}(x) = \mathbf{W}_2 \psi(\mathbf{W}_1 + \mathbf{b}_1) + \mathbf{b}_2 \]
  with weights and biases 
  $\mathbf{W}_1 \in \bR^{\din \times d_W}$, $\mathbf{W}_2 \in \bR^{d_W \times 1}$, 
  $\mathbf{b}_1 \in \bR^{d_W}$ and $\mathbf{b}_2 \in \bR$.
\end{theorem}

\subsection{Extension to smooth ReLU-like activation functions}
In practice, smooth, non-saturating activation functions, such as GeLU, softplus, and SiLU, are often used.
We can extend the universal approximation theorem of \cite{Hornik91} to include these activation functions by introducing a second layer into the neural network. 

We will extend the notion of $C^{m,p}_{\nu}$ to include the vector-valued case by defining 
\[
  C^{m,p}_{\nu}(\bR^{\din}, \bR^{\dout}) := \{f \in C^{m}(\bR^{\din}, \bR^{\dout}), 
  \| f \|_{W^{m,p}(\nu, \bR^{\dout})} < \infty \}, 
\]
where the Sobolev norm is analogously defined 
\begin{equation}
  \| f \|_{W^{m,p}(\nu, \bR^{\dout})} = \left( \int_{\bR^{d}} \|f(x)\|_2^p d\nu(x) + 
  \sum_{|\alpha|\leq m} \int_{\bR^{d}} \|\partial^{\alpha} f(x)\|_2^p d\nu(x) \right)^{\frac{1}{p}}.
\end{equation}

We then have the following.
\begin{theorem}[Universal approximation with smooth, ReLU like activation functions]\label{theorem:universal_approx_hornik_extended}
  Suppose $\nu$ is a finite measure on $\bR^{\din}$ and $g \in C^{m,p}_{\nu}(\bR^{\din}, \bR^{\dout})$ for some $m \geq 0$, $p \geq 1$.
  Let $\psi$ be an activation function in the class $\cA^\infty_b$. 
  Then, for any $\epsilon > 0$, 
  there exists a two-layered neural network $\varphi_{\psi, \theta}$ 
  with activation function $\psi$ and some width $d_W \in \bN$ such that 
  \begin{equation}
    \| g - \varphi_{\psi, \theta} \|_{W^{m,p}(\nu, \bR^{\dout})} \leq \epsilon,
  \end{equation}
  where 
  \[\varphi_{\psi, \theta}(x) = \mathbf{W}_3 \psi(\mathbf{W}_2 \psi(\mathbf{W}_1 + \mathbf{b}_1) + \mathbf{b}_2) + \mathbf{b}_3,\]
  with weights
  $\mathbf{W}_1 \in \bR^{\din\times d_W}$, 
  $\mathbf{W}_2 \in \bR^{d_W \times d_W}$, 
  $\mathbf{W}_3 \in \bR^{d_W \times \dout}$, 
  and biases
  $\mathbf{b}_1, \mathbf{b}_2 \in \bR^{d_W}$, 
  and $\mathbf{b}_{3} \in \bR^{\dout}$.
\end{theorem}
\begin{proof}
  This result provides extensions to the previous theorem by considering $\psi \in \cA^{\infty}_b$.
  The extension follows from the fact that we can construct a new activation function 
  $\tilde{\psi}(t) = \psi(-\psi(t)+b)$, using some $b \in \bR$, 
  that is non-constant and smooth with bounded derivatives up to order $k$. 

  We begin by verifying that $\tilde{\psi}$ is non-constant.
  Since $\psi \in \cA^{\infty}_b$, we know $|\psi(0)| \leq a_{-}$, and $\psi$ is increasing and non-saturating for positive inputs. Thus, 
  we choose $b > a_{-}$, such that $-\psi(0) + b > 0$ and $-\psi(t) + b < -\psi(0) + b$ for $t > 0$. 
  With this choice of $b$, 
  $-\psi(t) + b$ takes positive, changing values when $t$ is in some small interval $(0, \delta)$, 
  and so $\psi(-\psi(t) + b)$ is non-constant. 

  To see boundedness, we note that for $t > 0$, $-\psi(t) + b$ strictly decreases, such that $-\psi(t) + b \leq -\psi(0) + b$, where $-\psi(0) + b > 0$. 
  In other words, for $t > 0$, 
  $|\psi(-\psi(t) + b)| \leq \max(|\psi(-\psi(0)+b)|, a_{-})$.
  For $t \leq 0$, we have $|-\psi(t) + b| \leq a_{-} + b$, such that $\psi(-\psi(t) + b)$ is again bounded.

  When $m \geq 1$,
  boundedness of derivatives can be verified using the Fa\`a di Bruno formula (see, for example, \cite{Roman80}), a generalization of the chain rule for arbitrary order of the derivatives. 
  Specifically, the $\ell$th derivatives are polynomial combinations of the derivatives of $\psi$, i.e., 
  products of $\psi^{(k)}(-\psi(t)+b)$ and $\psi^{(k)}(t)$ with derivative orders $k$ up to $\ell$, all of which are bounded.
  For example, the first derivative of $\tilde{\psi}$ is 
  \[ \frac{d}{dt}\tilde{\psi}(t) = -\psi'(-\psi(t) + b)\psi'(t), \]
  which is bounded since $\psi'(t)$ is bounded. 
  Cases for derivatives of order $\ell \leq m$ 
  follow by repeating this process.
  We then remark that products and compositions of continuous functions are continuous, from which we can conclude that $\tilde{\psi} \in C^m_b(\bR)$. 

  We can therefore apply the universal approximation result of \Cref{theorem:universal_approx_hornik} using the activation function $\tilde{\psi}$. 
  That is, given any $\epsilon > 0$, there exists a neural network,
  \[\varphi_{\tilde{\psi}, \theta}(s) = \mathbf{W}_2 \tilde{\psi}(\mathbf{W}_1 s + \mathbf{b}_1) + \mathbf{b}_2\]
  with some hidden layer size $d_1$, which approximates $g \in C^{m}_b(\bR^{\din}, \bR)$ to $\epsilon$ error. 
  Since $\tilde{\psi}(t) = \psi(-\psi(t) + b)$, a componentwise application of this gives
  \[
    \varphi_{\tilde{\psi}, \theta}(s) = \mathbf{W}_2 \psi(-\mathbf{I} \psi(\mathbf{W}_1 s + \mathbf{b}_1) + \mathbf{b}) + \mathbf{b}_2 ,
  \]
  where $\mathbf{I} \in \bR^{d_1 \times d_1}$ is the identity matrix and $\mathbf{b} = (b, \dots, b) \in \bR^{d_1}$. 
  Thus, $\varphi_{\tilde{\psi}, \theta}$ is a neural network with two hidden layers and activation function $\psi$. 
  Finally, this can be extended to the case when the output is in $\bR^{\dout}$ by stacking individual approximations for each component. 
\end{proof}

\subsection{Extension to deep neural networks} 
\label{sec:proof_nn_deep_extension}
Additionally, the result can be extended to deep neural networks by considering smooth approximations to the identity function $\id(x) = x$.
We will use a form similar to that found in \cite{KovachkiLanthalerMishra21}.
\begin{lemma}\label{lemma:identity_approximation}
  Let $m \geq 0 $, $\psi \in \cA_{b}^{\infty}$, and $t_0 \in \bR$ be such that $\psi'(t_0) \neq 0$. 
  For $h \in (0,1]$, consider the function $\id_h(t) : \bR \rightarrow \bR$, defined as
  \begin{equation}\label{eq:identity_approximation}
    \id_h(t) = \frac{\psi(t_0 + h t) - \psi(t_0 - ht)}{2 h \psi'(t_0)}.
  \end{equation}
  Then, for any $\epsilon, B > 0$, there exists $h \in (0,1]$ such that $\id_h$ approximates the identity mapping $\id(t) = t$ to arbitrary accuracy, i.e., 
  \begin{equation}
    \| \id_h - \id \|_{C^{m}([-B, B])} \leq \epsilon.
  \end{equation}
  Additionally, $\id_h$ is Lipschitz continuous with bounded derivatives up to order $m$. That is 
  \[
    | \id_h(t_1) - \id_h(t_2) | \leq L |t_1 - t_2|, \quad \forall t_1, t_2 \in \bR, 
    \qquad | \id_h^{(k)}(t) | \leq C_k, \quad \forall t \in \bR, k = 1, \dots, m,
  \]
  where the Lipschitz constant $L = L(\psi, t_0)$ and the bound $C_k = C_k(\psi, t_0)$ are independent of $h$.
\end{lemma}
\begin{proof}
  Let $m \geq 0$, $\psi \in \cA_b^{\infty}$, $B > 0$, and $\epsilon > 0$ be given. 
  Additionally, let $a_k \geq 0$ denote the bound on the $k$th derivative of $\psi$, i.e., $|\psi^{(k)}| \leq a_k$. 
  We begin by noting for $t_0$ such that $\psi'(t_0) \neq 0$ and $h \in (0, 1]$, 
  the function $\id_h \in C^{m}(\bR)$ with derivatives
  \begin{equation}\label{eq:identity_approx_derivatives}
      \id_h^{(k)}(t)= h^{k-1} \left( \frac{\psi^{(k)}(t_0 + ht) + (-1)^{k-1} \psi^{(k)}(t_0 - ht)}{2 \psi'(t_0)} \right),
  \end{equation}
  for $1 \leq k \leq m$.

  The approximation is constructed such that 
  $\id_h(0) = 0$ and $\id_h'(0) = 1$, 
  i.e., it is a good approximation of the identity function near the origin. 
  To see this, we apply Taylor's theorem,
  \[
      \id_h(t) = \id_h(0) + \id_h'(0) t + \frac{1}{2}\id_h''(\xi) t^2
      = t +  \frac{1}{2}\id_h''(\xi) t^2
      ,
  \]
  where $\xi \in [0, t]$. Substituting in the expression for the second derivative, we arrive at the following bound for $t \in [-B,B]$,
  \[
      | \id_h(t) - t | \leq h \left| \frac{\psi''(t_0 + h \xi) - \psi''(t_0 - h \xi)}{2 \psi'(t_0)} \right| \frac{t^2}{2} \leq \frac{a_2 B^2}{|\psi'(t_0)|} h
  \] 
  which can be made arbitrarily small by choosing $h$ small.
  For the first derivative, we consider a Taylor expansion on $\id_h'$, such that 
  \[
    \id_h'(t) = \id_h'(0) + \id_h''(\xi) t,
  \]
  where $\xi \in [0, t]$ and $\id_h'(0) = 1$.
  This implies
  \[
    |\id_h'(t) - 1| \leq h \left| \frac{\psi''(t_0 + h \xi) - \psi''(t_0 - h \xi)}{2 \psi'(t_0)} \right| t \leq \frac{a_2 B}{|\psi'(t_0)|} h
  \]
  for $t \in [-B,B]$.
  For higher-order derivatives, we note that $\id^{(k)}(t) = 0$ for any $k \geq 2$.
  Moreover, inspecting \eqref{eq:identity_approx_derivatives}, we see that 
  \[
    |\id_h^{(k)}(t)| \leq h^{k-1} \left| \frac{\psi^{(k)}(t_0 + ht) + (-1)^{k-1} \psi^{(k)}(t_0 - ht)}{2 \psi'(t_0)} \right| \leq \frac{a_k}{|\psi'(t_0)|} h^{k-1}.
  \]
  Choosing $h$ sufficiently small such that all function and derivative errors are less than $\epsilon$ yields the desired error bound over the interval $[-B,B]$.

  To see Lipschitz continuity and boundedness of the derivatives, we again inspect the expression \eqref{eq:identity_approx_derivatives}.
  In particular, for $1 \leq k \leq m$, and $h \in (0,1]$, 
  \[
    |\id_h^{(k)}(t)| \leq \frac{a_k}{|\psi'(t_0)|}, 
  \]
  for all $t \in \bR$. 
  This implies the boundedness of all $k$th order derivatives for $1 \leq k \leq m$, and hence Lipschitz continuity of $\id_h$. 
  Moreover, it is clear that the bounds depend only on $\psi$ and $t_0$, and not on $h$.
\end{proof}

The identity approximation \eqref{eq:identity_approximation} can be thought of as a single neural network layer, and thus can be appended to a neural network $\varphi$ to form a deep neural network $\tilde{\varphi} = \id_h \circ \varphi$ approximating $\varphi$ itself.
\begin{lemma}
  Suppose $\nu$ is a finite measure on $\bR^{\din}$ and $\varphi \in C^{m,p}_{\nu}(\bR^{\din})$ for some $m \geq 1$, $p \geq 1$.
  Additionally, assume that $\varphi$ has bounded derivatives of order $1 \leq k \leq m$.
  Then, for any $\epsilon > 0$, there exists $h \in (0, 1]$ such that 
  \[
    \| \id_h \circ \varphi - \varphi \|_{W^{m,p}(\nu, \bR)} \leq \epsilon.
  \]
  Moreover, $\id_h \circ \varphi \in C^{m,p}_{\nu}(\bR^{\din})$ and also has bounded derivatives of order $1 \leq k \leq m$.
\end{lemma}
\begin{proof}
  Consider $\id_h$ as defined in \eqref{eq:identity_approximation}.
  Moreover, let $M_k$ denote the bound on the $k$th derivative of $\varphi$, i.e., $\|\deriv^k \varphi\|_{F_k} \leq M_k$,
  and let $L_k$ denote the bound on the $k$th derivative of $\id_h$, i.e., $|\id_h^{(k)}(t)| \leq L_k$, 
  both of which are independent of $B$ and $h$. 

  For some $B$, let $h \in (0, 1]$ be chosen such that $\| \id_h - \id \|_{C^{m}([-B, B])} \leq \delta$.
  Consider first the $k = 0$ term of the $W^{m,p}$ norm,
  \[
    \|\id_h \circ \varphi - \varphi \|_{L^p(\nu, \bR)}^p = \int |\id_h \circ \varphi - \varphi|^p d \nu
  \]
 We can decompose this as 
  \[
    \int |\id_h \circ \varphi - \varphi|^p d\nu =
    \int_{|\varphi| \leq B} |\id_h \circ \varphi - \varphi|^p d\nu + 
    \int_{|\varphi| > B} |\id_h \circ \varphi - \varphi|^p d\nu,
  \]
  which is bounded as 
  \begin{equation}\label{eq:identity_W0_bound}
    \int |\id_h \circ \varphi - \varphi|^p d\nu
    \leq \delta^p + \int_{|\varphi| > B} ((1+L_1)|\varphi |)^p d\nu,
  \end{equation}
  where $L_1$ is the bound on the first derivative of $\id_h$ and an upper bound on the Lipschitz constant.
  Note that with the help of the Lebesgue dominated convergence theorem, the second term can be made arbitrarily small by choosing $B$ large enough such that $\nu(|\varphi| > B)$ is small.

  For the first derivative, we have for any $i = 1, \dots \din$,
  \begin{equation}\label{eq:identity_W1_bound}
    \int | \partial_{x_i} (\id \circ \varphi) - \partial_{x_i} \varphi |^p d \nu \leq
    \int_{| \varphi | \leq B} |\id_h' \circ \varphi - 1|^p |\partial_{x_i} \varphi |^p d\nu + 
    \int_{| \varphi | > B} |\id_h' \circ \varphi - 1|^p |\partial_{x_i} \varphi |^p d\nu.
  \end{equation}
  Then, the sum is bounded by
  \[
    \int | \partial_{x_i} (\id \circ \varphi) - \partial_{x_i} \varphi |_2^p d \nu \leq
    \delta^p + \int_{| \varphi | > B} (L_k M_k)^p d\nu.
  \]

  For the higher-order derivatives, we recall the Fa\`a di Bruno formula for the $k$th derivative of a composition of functions \cite{Hardy06}.
  For $f : \bR^{d} \rightarrow \bR$ and $g : \bR \rightarrow \bR$, this takes the form 
  \[
    \frac{\partial^{k} (g \circ f)}{\partial x_{i_1} \dots \partial x_{i_k}} = 
    \sum_{\pi \in \Pi} g^{(|\pi|)}(f) \prod_{J \in \pi} \frac{\partial^{|J|} f}{\prod_{j \in J} \partial x_{i_j}},
  \]
  where $\Pi$ is the set of partitions of $\{1, \dots, k\}$ and $|\pi|, |J|$ refer to the number of elements in $\pi$ and $J$, respectively.
  Note that there is only one partition of size one, i.e., $\pi_k = \{ \{1, \dots, k\}\}$,
  which yields the term 
  $
    g'(f) \frac{\partial^k f}{\partial x_{i_1} \dots \partial x_{i_k}}.
  $
  We can therefore decompose the sum as 
  \[
    \frac{\partial^{k} (g \circ f)}{\partial x_{i_1} \dots \partial x_{i_k}} = 
    g'(f) \frac{\partial^k f}{\partial x_{i_1} \dots \partial x_{i_k}} +
    \sum_{\pi \in \Pi \setminus \pi_k} g^{(|\pi|)}(f) \prod_{J \in \pi} \frac{\partial^{|J|} f}{\prod_{j \in J} \partial x_{i_j}}.
  \]
  Applying this to the composition $\id_h \circ \varphi$, 
  the error in each component of the $k$th derivative is given by
  \[
    \frac{\partial^{k}}{\partial x_{i_1} \dots \partial x_{i_k}} (\id_h \circ \varphi) - 
    \frac{\partial^{k}}{\partial x_{i_1} \dots \partial x_{i_k}} (\varphi) =
    (\id_h' \circ \varphi  - 1) \frac{\partial^k \varphi}{\partial x_{i_1} \dots \partial x_{i_k}} +
    \sum_{\pi \in \Pi \setminus \pi_k} 
    \id_h^{(|\pi|)}(\varphi)
    \prod_{J \in \pi}
    \frac{\partial^{|J|} \varphi}{\prod_{j \in J} \partial x_{i_j}}.
  \]
  Consider the first term, whose $L^p$ norm can be decomposed into 
  \[
    \int_{| \varphi | \leq B} |\id_h' \circ \varphi - 1|^p 
    \left| \frac{\partial^k \varphi}{\partial x_{i_1} \dots \partial x_{i_k}} \right|^p  d\nu
    +      
    \int_{| \varphi | > B} |\id_h' \circ \varphi - 1|^p 
    \left| \frac{\partial^k \varphi}{\partial x_{i_1} \dots \partial x_{i_k}} \right|^p  d\nu.
  \]  
  Again, by the boundedness assumptions, we have 
  \[
    \int_{} |\id_h' \circ \varphi - 1|^p 
    \left| \frac{\partial^k \varphi}{\partial x_{i_1} \dots \partial x_{i_k}} \right|^p  d\nu
    \leq 
    \delta^p M_k  + \int_{\| \varphi \| > B} ((1 + L_1)M_k)^p d\nu.
  \]
  Moreover, for every entry in the summation of the second term, we have 
  \begin{align*}
    \int_{} 
    |\id_h^{(|\pi|)}(\varphi) |^p 
    \left| \prod_{J \in \pi} \frac{\partial^{|J|} \varphi}{\prod_{j \in J} \partial x_{i_j}} 
    \right|^p d\nu 
    & =
    \int_{|\varphi| \leq B} 
    |\id_h^{(|\pi|)}(\varphi) |^p 
    \left| \prod_{J \in \pi} \frac{\partial^{|J|} \varphi}{\prod_{j \in J} \partial x_{i_j}} 
    \right|^p d\nu  \\
    & +
    \int_{|\varphi| > B} 
    |\id_h^{(|\pi|)}(\varphi) |^p 
    \left| \prod_{J \in \pi} \frac{\partial^{|J|} \varphi}{\prod_{j \in B} \partial x_{i_j}} 
    \right|^p d\nu
  \end{align*}
  Note that for $\pi \in \Pi \setminus \{\pi_k \}$, $|\pi| > 1$, and hence $\id_h^{(|\pi|)} \leq \delta$ 
  inside $[-B,B]^d$ and is bounded outside. 
  Thus, we have the following bound,
  \begin{equation}\label{eq:identity_Wk_bound}
    \int_{} 
    |\id_h^{(|\pi|)}(\varphi) |^p 
    \left| \prod_{J \in \pi} \frac{\partial^{|J|} \varphi}{\prod_{j \in J} \partial x_{i_j}} 
    \right|^p d\nu 
    \leq
    \delta^p  \prod_{J \in \pi} M_{|J|}^p 
    + \int_{|\varphi| > B} \left( L_{|\pi|} \prod_{J \in \pi} M_{|J|} \right)^{p} d\nu.
  \end{equation}

  To obtain the overall error bound in the $W^{m,p}(\nu, \bR)$, we combine the bounds in \eqref{eq:identity_W0_bound}, \eqref{eq:identity_W1_bound}, and \eqref{eq:identity_Wk_bound} (summed over all partial derivatives),
  choosing first $B$ such that $\nu(\| \varphi \| > B)$ is small, and then choosing $h$ such that $\delta$ is small.

  We now check that $\id_h \circ \varphi \in C^{m,p}_{\nu}(\bR^{\din})$ and has bounded derivatives.
  In particular, since $\id_h$ is Lipschitz, the fact that $\varphi \in L^p(\nu, \bR)$ implies that $\id_h \circ \varphi \in L^p(\nu, \bR)$.
  Moreover, to see boundedness of the derivatives of $\id_h \circ \varphi$, 
  it suffices to inspect the Fa\`a di Bruno formula and note that both the derivatives of $\id_h$ and $\varphi$ up to order $m$ are bounded.
  Thus, overall, $\id_h \circ \varphi \in C^{m,p}_{\nu}(\bR^{\din})$ and has bounded derivatives of order $1 \leq k \leq m$.

\end{proof}

We can now proceed with the proof of \Cref{theorem:universal_approx_hornik_extended_deep}.
\begin{proof}(of \Cref{theorem:universal_approx_hornik_extended_deep})
  This result follows from the previous theorem by considering the identity approximation $\id_h$ as defined in \Cref{lemma:identity_approximation}.
  Beginning with the two-layer approximation result of \Cref{theorem:universal_approx_hornik_extended},
  we can construct a deep neural network by appending the identity approximation to the output of the two-layer network.
  That is, let $\varphi_1$ be the two-layer approximation to $g$ as defined in \Cref{theorem:universal_approx_hornik_extended},
  and let $\id_h$ be an approximation to the identity function as defined in \Cref{lemma:identity_approximation}.
  Moreover, due to the choice of activation function, $\varphi_1$ has bounded derivatives of order $1 \leq k \leq m$.
  Then, we define $\varphi_2 = \id_h \circ \varphi_1$ as the deep neural network with $d_L = 3$ layers,
  noting that when applied componentwise, $\id_h$ is a single layer neural network with activation function $\psi$ whose width is twice the size of the input dimension.
  The approximation error is given by
  \[
    \| g - \varphi_2 \|_{W^{m,p}(\nu, \bR^{\dout})} \leq
    \| g - \varphi_1 \|_{W^{m,p}(\nu, \bR^{\dout})} +
    \| \varphi_1 - \id_h \circ \varphi_1 \|_{W^{m,p}(\nu, \bR^{\dout})}.
  \]
  By \Cref{theorem:universal_approx_hornik_extended}, 
  we can choose $\varphi_1$ such that $\| g - \varphi_1 \|_{W^{m,p}(\nu, \bR^{\dout})} \leq \epsilon/2$,
  and by \Cref{lemma:identity_approximation}, we can choose $h_1$ such that 
  $\| \id_{h_1} \circ \varphi_1 - \varphi_1 \|_{W^{m,p}(\nu, \bR^{\dout})} \leq \epsilon/2$.
  Since the composition $\id_{h_1} \circ \varphi \in C^{m,p}_{\nu}(\bR^{\din})$ also has bounded derivatives of order $1 \leq k \leq m$, the process can be repeated to construct a deep neural network with $d_L \geq 2$ of the form
  $ \varphi_{\psi,\theta} = \id_{h_{d_L - 2}} \circ \dots \circ \id_{h_1} \circ \varphi_1$ 
  with arbitrary approximation error $\epsilon$.
  The layers of the neural network can be padded by zeros to ensure a constant width $d_W$.
\end{proof}

\subsection{Examples of activation functions in $\cA^{\infty}_b$}
Many of the common, ReLU-like activation functions are in the class $\cA^{\infty}_b$, including the softplus, GeLU, and SiLU functions.

\begin{example}[The softplus activation] 
  The function $\psi(t) = \log(1 + e^t)$ is in $\cA^{\infty}_{b}$.
\end{example}
\begin{proof}
  We first note that $\psi(t)$ is monotonically increasing and bounded below by $0$. We can verify its derivatives are all bounded. 
  Note that the first derivative of the softplus function is the (logistic) sigmoid, 
  $\sigma(t) := 1/(1+e^{-t})$. 
  That is, 
  \[ \psi'(t) = \frac{e^t}{1 + e^t} = \frac{1}{1 + e^{-t}} = \sigma(t). \]
  Evidently, this is continuous and bounded.

  Subsequent derivatives of $\psi$ are derivatives of the sigmoid, which has the well known property 
  \[ \sigma'(t) = \sigma(t)(1 - \sigma(t)). \]
  We can verify by induction that higher derivatives of $\sigma$ are all of the form $p(\sigma) \sigma (1 - \sigma)$, where $p(t)$ is a polynomial function. 
  Clearly this is true for the first derivative. Now, suppose the $k$th derivative is of the form $p_k(\sigma) \sigma (1 - \sigma)$. 
  Then, the $(k+1)$th derivative is given by
  \[ p_k'(\sigma) \sigma^2 (1-\sigma)^2
  + p_k(\sigma) \sigma (1 - \sigma)^2
  - p_k(\sigma) \sigma^2 (1 - \sigma) . \]
  Factorizing out $\sigma (1-\sigma)$ yields an expression of the form $p_{k+1}(\sigma) \sigma (1-\sigma)$, where $p_{k+1}$ is a polynomial with one higher degree.
  Since $\sigma$ is bounded, we can conclude that $\psi^{(k)}(t)$ is continuous and bounded for any $k \in \bN$.
\end{proof}
\begin{example}[The SiLU activation]
  The sigmoid linear unit (SiLU) function, $\psi(t) = t \sigma(t)$, where $\sigma(t) = 1 / (1 + e^{-t})$ is the logistic sigmoid, is in $\cA^{\infty}_b$.
\end{example}
\begin{proof}
  We can verify that $\psi(0) = 0$, and that $\psi(t)$ strictly increases for $t \geq 0$. 
  On the other hand, as $t \rightarrow -\infty$, $\psi(t) \rightarrow 0$, so it must be bounded on the interval $(-\infty, 0]$. 
  For the derivatives, we note that 
  \[ \psi'(t) = \sigma(t) + t \sigma'(t). \]
  Recall that $\sigma'(t) = \sigma(t)(1 - \sigma(t)) = (1+e^{-t})^{-1} \cdot (1+e^{t})^{-1}$, such that $t \sigma'(t)$ tends to zero as $t$ approaches $\pm \infty$. 
  Thus, $\psi'(t)$ is continuous and bounded. 
  Higher derivatives can be verified by induction. 
  The derivatives of $\sigma(t)$ are continuous and bounded, as shown in the softplus example. Furthermore, derivatives of the $t \sigma'(t)$ term can be shown to take the form $(p(\sigma) + t q(\sigma)) \sigma (1 - \sigma)$, where $p$ and $q$ are polynomials. 
  This is true for the zeroth derivative (i.e., $t \sigma'(t) = t \sigma(1-\sigma)$ is of this form). 
  Suppose the $k$th derivative is of the form 
  $(p_k(\sigma) + t q_k(\sigma)) \sigma(1 - \sigma)$. 
  Then, the $(k+1)$th derivative is given by
  \[ (p_k'(\sigma) + t q_k'(\sigma)) \sigma^2 (1-\sigma)^2 + q_k(\sigma) \sigma(1-\sigma) + (p_k(\sigma) + t q_k(\sigma))(\sigma(1-\sigma)^2 - \sigma^2(1-\sigma)).\]
  Factorizing out $\sigma(1-\sigma)$ yields an expression of the form 
  \[ (p_{k+1}(\sigma) + t q_{k+1}(\sigma) ) \sigma(1-\sigma), \]
  where $p_{k+1}$ and $q_{k+1}$ are polynomials with one higher degree.
  Recall that for any polynomial $p$, the function $p(\sigma(t))\sigma(t)(1-\sigma(t))$ is bounded. Moreover, for any polynomial $q$, the function 
  $t q(\sigma(t)) \sigma(t)(1-\sigma(t))$ approaches zero as $t \rightarrow \pm \infty$.
  Thus, higher derivatives of $\psi$ are also continuous and bounded.
\end{proof}

\begin{example}{The GeLU activation} 
  The Gaussian error linear unit (GeLU) function, $\psi(t) = t \Phi(t)$, where $\Phi(t) = \frac{1}{\sqrt{2\pi}} \int_{-\infty}^{t} e^{-x^2/2} dx$ is the cumulative distribution function of the standard Gaussian, is in $\cA^{\infty}_b$.
\end{example}
\begin{proof}
  Similar to the SiLU, we can verify that $\psi(0) = 0$, and that $\psi(t)$ strictly increases for $t \geq 0$. Moreover, as $t \rightarrow -\infty$, $\psi(t) \rightarrow 0$, so it again must be bounded on the interval $(-\infty, 0]$.
  For the first derivative, we have 
  \[ \psi'(t) = \Phi(t) + (2\pi)^{-1/2} t e^{-t^2/2}, \]
  which is bounded since $t e^{-t^2/2}$ approaches zero as $t$ approaches $\pm \infty$.
  Continuity and boundedness of higher-order derivatives can once again be verified by induction, noting that this time, they take the form $p(t) e^{-t^2/2}$, where $p(t)$ is a polynomial.

\end{proof}

\section{Proofs for the error analysis}

\subsection{Proof of the universal approximation of conditional expectations}\label{sec:proof_universal_approx_cond_exp}

\begin{proof}
  Let $\epsilon > 0$ be arbitrary, and consider first some $\eta > 0$. 
  Since $\cF\cC^{\infty}$ is dense in $H^m_{\gamma}$, we can find $\fun_{\eta} \in \cF \cC^{\infty}$ such that $\|\fun - \fun_{\eta} \|_{H^m_{\gamma}} \leq \eta$.
  This also means that $\|\Pry \bE_{\gamma}[\fun | \sigma(\pinv{\Vrx})] - \Pry \bE_{\gamma}[\fun_{\eta} | \sigma(\pinv{\Vrx})]\|_{H^m_{\gamma}} \leq \eta$.

  We then focus on approximating $\Pry \bE_{\gamma}[\fun_{\eta} | \sigma(\pinv{\Vrx})]$.
  By \Cref{lemma:interchange_differentiation_expectation_smooth}, since $\fun_{\eta} \in \cF \cC^{\infty}$, we have $\bE_{\gamma}[\fun_{\eta} | \sigma(\pinv{\Vrx})] \in C^{\infty}_{b}$.
  We therefore look for an approximation of the function 
   $ g(s) := \pinv{\Ury} \bE_{\gamma}[\fun_{\eta} | \sigma(\pinv{\Vrx})](\Vrx s) $
  by a neural network, where $g \in C^{\infty}_{b}(\bR^{r}, \bR^{r})$.
  In particular, the universal approximation result 
  of \Cref{theorem:universal_approx_hornik_extended}
  implies that there exists a neural network $\varphi = \varphi_{\psi, \theta}$ approximating $g$ such that 
  \[\|g - \varphi\|_{H^m(\gamma_r, \bR^r)} \leq \eta \]
  where $\gamma_r \sim \cN(0, I_{r})$ refers to the standard Gaussian on $\bR^{r}$. 
  Moreover, the $L^2_{\gamma}$ error between the projected conditional expectation and the neural operator can be expressed as
  \begin{equation*}
     \int \| \Pry \bE_{\gamma}[\fun_{\eta} | \sigma(\pinv{\Vrx})](x) - \Ury \varphi(\pinv{\Vrx} x)\|_{\cY}^2 \gamma(dx) 
     = \int \| \pinv{\Ury} \bE_{\gamma}[\fun_{\eta} | \sigma(\pinv{\Vrx})](\Vrx \pinv{\Vrx} x) - \varphi(\pinv{\Vrx} x)\|_{2}^2 \gamma(dx),
  \end{equation*}
  since $\|\Pry u\|_{\cY}^2 = \|\pinv{\Ury} u \|_{2}^2$. We can now perform a change of variable to $s = \pinv{\Vrx} x$. Since $\Vrx$ is orthonormal in $\XC$, is distributed according to $\gamma_r$, such that 
  \begin{equation*}
     \int \| \Pry \bE_{\gamma}[\fun_{\eta} | \sigma(\pinv{\Vrx})](x) - \Ury \varphi(\pinv{\Vrx} x)\|_{\cY}^2 \gamma(dx) 
     = \int \| \pinv{\Ury} \bE_{\gamma}[\fun_{\eta} | \sigma(\pinv{\Vrx})](\Vrx s) - \varphi(s)\|_{2}^2 \gamma_r (ds) 
  \end{equation*}
  Moreover, since we have defined $g = \pinv{\Ury} \bE_{\gamma}[\fun_{\eta} | \sigma(\pinv{\Vrx})](\Vrx s)$, we have
  \begin{equation*}
     \int \| \pinv{\Ury} \bE_{\gamma}[\fun_{\eta} | \sigma(\pinv{\Vrx})](\Vrx s) - \varphi(s)\|_{2}^2 \gamma_r (ds) 
     = \int \| g(s) - \varphi(s)\|_{2}^2 \gamma_r (ds).
  \end{equation*}



Similarly, for $k$th derivative, 
\begin{align*}
  & \int \| \deriv^k \Pry (\bE_{\gamma}[\fun_{\eta} | \sigma(\pinv{\Vrx})](x)) - \deriv^k (\Ury \varphi(\pinv{\Vrx} x))\|_{\HS_k(\XC,\cY)}^2 \gamma(dx)  \\
  & \quad = \int \| \Pry \bE_{\gamma}[\deriv^k \fun_{\eta} | \sigma(\pinv{\Vrx})](\Qrx x) (\Qrx \cdot, \dots, \Qrx \cdot) 
    - \Ury \deriv^k \varphi(\pinv{\Vrx} x) (\pinv{\Vrx} \cdot, \dots, \pinv{\Vrx} \cdot) \|_{\HS_k(\XC, \cY)}^2 \gamma(dx) 
\end{align*}
where we have interchanged differentiation and conditional expectation is given by 
\Cref{lemma:interchange_differentiation_conditional_expectation} in \Cref{sec:interchange_differentiation_integration}.
We then follow the same approach, projecting the norm into the reduced subspace and then performing the change of variables $s = \pinv{\Vrx} x$ to obtain a similar bound for the derivatives,
\begin{align*}
  & \int \| \deriv^k \Pry (\bE_{\gamma}[\fun_{\eta} | \sigma(\pinv{\Vrx})](x)) - \deriv^k (\Ury \varphi(\pinv{\Vrx} x))\|_{\HS_k(\XC,\cY)}^2 \gamma(dx)  \\
  & \quad = \int \| \pinv{\Ury} \bE_{\gamma}[\deriv^k \fun_{\eta} | \sigma(\pinv{\Vrx})](\Vrx \pinv{\Vrx} x) (\Vrx \cdot, \dots, \Vrx \cdot) 
    - \deriv^k \varphi(\pinv{\Vrx} x) \|_{F_k}^2 \gamma(dx)  \\
  & \quad = \int \| \pinv{\Ury} \bE_{\gamma}[\deriv^k \fun_{\eta} | \sigma(\pinv{\Vrx})](\Vrx s) (\Vrx \cdot, \dots, \Vrx \cdot) - \deriv^k \varphi(s) \|_{F_k}^2 \gamma_r (ds)  \\
  & \quad = \int \| \deriv^k g(x) - \deriv^k \varphi(s) \|_{F_k}^2 \gamma_r (ds).
\end{align*}
where we have used $\|\cdot\|_{F_k} = \| \cdot \|_{\HS_k(\bR^r, \bR^r)}$ to denote the tensor Frobenius norms on $\bR^r$, 
noting that $\|\cdot\|_{F_1} = \|\cdot\|_{F}$ is the standard matrix Frobenius norm.

  Finally, since the Fr\'echet derivatives coincide with the weak derivatives, we have 
  \[ \| \Pry \bE_{\gamma}[\fun | \sigma(\pinv{\Vrx})] - \Ury \circ \varphi \circ \pinv{\Vrx} \|_{H^m_{\gamma}} 
  = \|g - \varphi \|_{H^m(\gamma_r, \bR^{r})} \leq \eta.
  \]
  Combined with the fact that $\|\Pry \bE_{\gamma}[\fun | \sigma(\pinv{\Vrx})] - \Pry \bE_{\gamma}[\fun_{\eta} | \sigma(\pinv{\Vrx})]\|_{H^m_{\gamma}} \leq \eta$, it suffices to choose $\eta = \epsilon /2$.
\end{proof}

\subsection{Proof of the generic universal approximation} \label{sec:proof_naive_universap_approximation}

The universal approximation of projected conditional expectations in $H^m_{\gamma}$ can be used to obtain the generic universal approximation result of \Cref{theorem:naive_universal_approximation} as follows.
\begin{proof}
  Let $\epsilon$ be arbitrary. Given $\fun \in H^m_{\gamma}$, we have 
  \[ 
    \|\fun\|_{L^2_{\gamma}}^2 
    = \bE_{\gamma} \left[ \sum_{i=1}^{\infty} | \linner u_i, \fun \rinner_{\cY} |^2 \right]
    = \sum_{i=1}^{\infty} \bE_{\gamma} \left[ | \linner u_i, \fun \rinner_{\cY} |^2 \right]  
    < \infty
  \]
  where the interchange of sum and expectation is possible since all terms are non-negative.
  Moreover, for any $r > 0$, we have 
  \[ 
    \|\fun - \Pry \fun \|_{L^2_{\gamma}}^2 
    = \bE_{\gamma} \left[ \sum_{i=r+1}^{\infty} | \linner u_i, \fun \rinner_{\cY} |^2 \right]  
    = \sum_{i=r+1}^{\infty} \bE_{\gamma} \left[ | \linner u_i, \fun \rinner_{\cY} |^2 \right].
  \]
  Thus, $\|\fun - \Pry \fun \|_{L^2_{\gamma}}$ converges monotonically to zero as $r \rightarrow \infty$. 
  Analogously, for any $k \leq m$, 
  \begin{equation} \label{eq:output_projection_error_arb}
    \|\derivXC^k \fun\|_{L^2(\gamma, \HS_k(\XC, \cY))}^2 
    = \sum_{j=1}^{\infty} \sum_{i_1, \dots, i_k}^{\infty} \bE_{\gamma} \left[ \left| \linner u_j, \derivXC^k \fun(w_{i_1}, \dots, w_{i_k}) \rinner_{\cY} \right|^2 \right]
    < \infty,
  \end{equation}
  such that 
  \begin{equation} \label{eq:output_projection_error_arb_derivatives}
    \|\derivXC^k \fun - \Pry \derivXC^k \fun\|_{L^2(\gamma, \HS_k(\XC, \cY))}^2 
    = \sum_{j=r+1}^{\infty} \sum_{i_1, \dots, i_k}^{\infty} \bE_{\gamma} \left[ \left| \linner u_i, \derivXC^k \fun(w_{i_1}, \dots, w_{i_k}) \rinner_{\cY} \right|^2 \right]
    < \infty,
  \end{equation}
  also converges monotonically to zero. 
  This allows us to find some $r_{\cY}$ 
  such that $\|\fun - P_{r_{\cY}} \fun \|_{H^m_{\gamma}} \leq \epsilon/3$.
  We then consider an input rank $r_{\cX} \geq r_{\cY}$, 
  such that 
  the conditional expectation $\bE_{\gamma}[\fun | \sigma(V_{r_{\cX}})]$ also satisfies
  $\| \fun - \bE_{\gamma}[\fun | \sigma(V_{r_{\cX}})] \|_{H^m_{\gamma}} \leq \epsilon/3$. 
  This is possible since the conditional expectations converge to $\fun$ as $r \rightarrow \infty$ by Theorem 5.4.5 of \cite{Bogachev98}.
  We can now increase the output rank to $r_{\cX}$ in order to make the input and output ranks equal, i.e., $r = r_{\cX} \geq r_{\cY}$. 
  Note that 
  \[ 
  \|\fun - \Pry \fun \|_{H^m_{\gamma}} \leq  \| \fun - P_{r_{\cY}} \fun \|_{H^m_{\gamma}} \leq \epsilon/3
  \] 
  due to \eqref{eq:output_projection_error_arb} and \eqref{eq:output_projection_error_arb_derivatives} and 
  \[ 
  \| \Pry \fun - \bE_{\gamma}[\Pry \fun | \sigma(\pinv{\Vrx})] \|_{H^m_{\gamma}}
  \leq \| \fun - \bE_{\gamma}[\fun | \sigma(\pinv{\Vrx})] \|_{H^m_{\gamma}} \leq \epsilon/3,
  \] 
  since $\|\Pry\|_{\cL(\cY, \cY)} = 1$. 

  \Cref{theorem:universal_approx_cond_exp} then implies the existence of a neural network $\varphi_{\psi, \theta}$ such that the neural operator 
  $ \ftilde_{\theta} = \Ury \circ \varphi_{\psi, \theta} \circ \pinv{\Vrx}$ 
  approximates the projected conditional expectation to error 
  \[ \| \bE_{\gamma}[\Pry \fun | \sigma(\pinv{\Vrx})] - \ftilde_{\theta} \|_{H^m_{\gamma}} \leq \epsilon/3. 
  \]
  Thus, combining the error terms in a triangle inequality yields the desired error bound,
  \begin{align*}
    \| \fun - \ftilde_{\theta} \|_{H^m_{\gamma}} \leq 
    \| \fun - \Pry \fun \|_{H^m_{\gamma}} 
    + \| \Pry \fun - \bE_{\gamma}[\Pry \fun | \sigma(\pinv{\Vrx})] \|_{H^m_{\gamma}}
    + \| \bE_{\gamma}[\Pry \fun | \sigma(\pinv{\Vrx})] - \ftilde_{\theta} \|_{H^m_{\gamma}} 
    \leq \epsilon.
  \end{align*}
\end{proof}

\subsection{Proofs for the error decompositions} \label{sec:proof_overall_error}
We begin by proving \Cref{prop:overall_error}.
\begin{proof}(of \Cref{prop:overall_error})
  The results follow from decomposing each term into its component within the projection subspace and its orthogonal complement. 
  We begin with the $L^2_{\gamma}$ error. For $\ftilde = \Ury \circ \varphi \circ \pinv{\Vrx} + \fhat$, we can write the misfit $\|\fun(x) - \ftilde(x)\|_{\cY}^2$ as
  \[ \| \fun(x) - \ftilde(x) \|_{\cY}^2 = \|\Pry(\fun(x) - \ftilde(x))\|_{\cY}^2 + \|(I - \Pry)(\fun(x) - \fhat)\|_{\cY}^2, \]
  since $(I - \Pry) \Ury = 0$.
  Moreover, using the decomposition
  $\fun - \fhat = \fun - \fbar + \fbar - \fhat$, we have
  \[ \|(I - \Pry)(\fun - \fhat) \|_{\cY}^2 
  = \|(I - \Pry)(\fun - \fbar) \|_{\cY}^2 
  + 2 \linner (I - \Pry)(\fun - \fbar), (\fbar - \fhat) \rinner_{\cY} + \|(I - \Pry)(\fbar - \fhat) \|_{\cY}^2.
  \]
  Taking the expectation over $\gamma$ removes the cross term, leaving the desired result.

  For the derivative error, we have $\deriv \ftilde(x) = \Ury \deriv \varphi (\pinv{\Vrx} x) \pinv{\Vrx}$. 
  Thus, using $(v_i)_{i=1}^{\infty}$ as the basis for evaluating the Hilbert--Schmidt norm, we have
  \begin{align*}
    \|\derivXC \fun(x) - \derivXC \ftilde (x)& \|_{\HS(\XC, \cY)}^2 = \sum_{i=1}^{\infty} 
      \| (\derivXC \fun(x) - \derivXC \ftilde(x)) v_i \|_{\cY}^2   \\
    & = \sum_{i=1}^{\rx} \| (\derivXC \fun(x) - \derivXC \ftilde(x)) v_i \|_{\cY}^2 
      + \sum_{i = \rx + 1}^{\infty} \| (\derivXC \fun(x) - \derivXC \ftilde(x) )v_i \|_{\cY}^2 \\
    & = \| ( \derivXC \fun(x) - \derivXC \ftilde(x) ) \Qrx \|_{\HS(\XC, \cY)}^2 + \| \derivXC \fun(x) (I - \Qrx) \|_{\HS(\XC, \cY)}^2,
  \end{align*}
  since $\derivXC \ftilde(x) (I -\Qrx) = 0$. 
  We further decompose the first term into
  \begin{align*} 
    \| ( \derivXC \fun(x) - \derivXC \ftilde(x) ) \Qrx \|_{\HS(\XC, \cY)}^2 &= 
        \| \Pry(\derivXC \fun(x) - \derivXC \ftilde(x)) \Qrx \|_{\HS(\XC, \cY)}^2  \\
        & + \| (I - \Pry) \derivXC \fun(x) \Qrx \|_{\HS(\XC, \cY)}^2 .
  \end{align*}
  Finally, we note that 
  \[ \|(I - \Pry) \derivXC \fun(x) \Qrx \|_{\HS(\XC, \cY)}^2 \leq \|(I - \Pry) \derivXC \fun(x) \|_{\HS(\XC, \cY)}^2. \]
  Combining the individual components and taking the expectation yields the overall bound.


  To write the terms involving $\ftilde$ in the latent spaces, we make use of Parseval's identity with the basis $(u_i)_{i=1}^{\infty}$.
  For the function value, we have 
  \[ \|\Pry(\fun - \ftilde) \|_{\cY}^2 = \|\Ury \pinv{\Ury}(\fun - \Ury \varphi(\pinv{\Vrx} x ) - \fhat) \|_{\cY}^2  
    = \| \pinv{\Ury}(\fun - \fhat) - \varphi(\pinv{\Vrx}x) \|_{2}^2. \] 
  For the derivative, we additionally use $(v_i)_{i=1}^{\infty}$ as the basis for the $\HS(\XC, \cY)$ norm, giving
  \begin{align*} 
  \| \Pry (\derivXC \fun(x) - \derivXC \ftilde )\Qrx \|_{\HS(\XC, \cY)}^2 
    & = \| \pinv{\Ury}\derivXC \fun(x) \Vrx - \deriv \varphi (\pinv{\Vrx} x)\|_{F}^2.
  \end{align*}
\end{proof}


\Cref{theorem:error_decomposition_with_conditional_expectation} then follows as a result of the error decomposition in \Cref{prop:overall_error}
along with the universal approximation of conditional expectations in \Cref{theorem:universal_approx_cond_exp}.
\begin{proof}(of \Cref{theorem:error_decomposition_with_conditional_expectation})
  By \Cref{prop:overall_error}, we can decompose the approximation error into
  \[\|\fun - \ftilde_{\theta}\|_{L^2_{\gamma}}^2 \leq \|\Pry(\fun - \ftilde_{\theta})\|_{L^2_{\gamma}}^2 + \|(I - \Pry)(\fun - \fbar)\|_{L^2_{\gamma}} + \|\fbar - \fhat\|_{\cY}^2,\]
  and 
  \begin{align*}
    |\fun - \ftilde_{\theta}|_{H^1_{\gamma}}^2 \leq & \|\Pry(\derivXC \fun - \derivXC \ftilde_{\theta})\Qrx \|_{L^2(\gamma, \HS(\XC, \cY))}^2 \\
    & + \|(I - \Pry)\derivXC \fun\|_{L^2(\gamma, \HS(\XC,\cY))}^2
    + \|\derivXC \fun (I - \Qrx)\|_{L^2(\gamma, \HS(\XC,\cY))}^2.
  \end{align*}
  We first make use of the conditional expectation $\ftilde_r = \bE_{\gamma}[\fun | \sigma(\pinv{\Vrx})]$ as an intermediate to the approximation $\fun$. The triangle inequality, along with the inequality $(|a| + |b|)^2 \leq 2|a|^2 + 2|b|^2$ yields
  \begin{align*}
    \| \Pry (\fun - \ftilde_{\theta}) \|_{L^2_{\gamma}}^2
    \leq
      2 \| \Pry (\fun - \ftilde_r) \|_{L^2_{\gamma}}^2 + 
      2 \| \Pry (\ftilde_r - \ftilde_{\theta}) \|_{L^2_{\gamma}}^2
  \end{align*}
  and
  \begin{align*}
    \| \Pry (\derivXC \fun - \derivXC \ftilde_{\theta}) \Qrx \|_{L^2(\gamma, \HS(\XC, \cY))}^2
    & \leq 
    2\| \Pry (\derivXC \fun - \derivXC \ftilde_r) \Qrx \|_{L^2(\gamma, \HS(\XC, \cY))}^2 \\
    & \qquad + 
    2\| \Pry (\derivXC \ftilde_r - \derivXC \ftilde_{\theta}) \Qrx \|_{L^2_{\gamma}(\gamma, \HS(\XC, \cY))}^2.
  \end{align*}
  By \Cref{theorem:universal_approx_cond_exp}, there exists neural network $\varphi_{\psi, \theta}$ 
  such that $\Ury \circ \varphi_{\psi, \theta} \circ \pinv{\Vrx}$ approximates 
  $\Pry (\ftilde_r - \fhat) = \bE_{\gamma}[\Pry(\fun - \fhat) | \sigma(\pinv{\Vrx})]$
  in $\|\cdot\|_{H^1_{\gamma}}^2$ to error $\epsilon_{\theta}/2$.
  This means that $\ftilde_{\theta} = \Ury \circ \varphi_{\psi, \theta} \circ \pinv{\Vrx} + \fhat$ satisfies the desired error bounds.
\end{proof}

\subsection{Proof of the subspace Poincar\'e inequality} \label{sec:proof_subspace_poincare}

We begin by stating the Poincar\'e inequality, which holds for Gaussian measures on separable Hilbert spaces \cite{Bogachev98}.
\begin{theorem}[Poincar\'e inequality for Gaussian measures]\label{theorem:poincare}
  Suppose $\fun \in H^1_{\gamma}$. Then, 
  \begin{equation}\label{eq:poincare}
    \bE_{\gamma}[\|\fun - \fbar \|_{\cY}^2] \leq 
    \bE_{\gamma}[\|\derivXC \fun\|_{\HS(\XC, \cY)}^2],
  \end{equation}
  where $\fbar = \bE_{\gamma}[\fun]$ is the mean of $\fun$.
\end{theorem}

The subspace Poincar\'e inequality then follows from the Poincar\'e inequality itself.
\begin{proof}(Of \Cref{theorem:subspace_poincare})
  The procedure is almost identical to that of \cite{ZahmConstantinePrieurEtAl20},
  with justifications to ensure the operations are valid in the separable Hilbert space setting.
  We first consider the case where $\fun \in \cF \cC^{\infty}$, 
  for which we can define
  $\gfun(x,y) := \fun(\Qrx x + (I - \Qrx) y)$, such that 
  $\bE_{y \sim \gamma}[\gfun(x,y)] = \bE_{\gamma}[\fun | \sigma(\pinv{\Vrx})](x)$.
  Since $\Qrx$ and $I - \Qrx$ are continuous linear operators, we have 
  \[\partial_{y} \gfun(x,y) = \deriv \fun(\Qrx x + (I - \Qrx) y) (I - \Qrx). \]
  Moreover, since $\Qrx x + (I - \Qrx)y \sim \gamma$ whenever $(x, y) \sim \gamma \otimes \gamma$, we know
  \[ \iint \|\gfun(x,y)\|_{\cY}^2 \gamma(dy) \gamma(dx) 
  = \iint \|\fun(\Qrx x + (I - \Qrx) y)\|_{\cY}^2  \gamma(dy) \gamma(dx) 
  = \|\fun\|_{L^2_{\gamma}}^2
  < \infty. \]
  An analogous calculation can be made for $\partial_y \gfun$. 
  Therefore, we have
  $ \int |\gfun(x,y)|^2 \gamma(dy) < \infty $
  and $\int \|\partial_y\gfun(x,y) \|_{\HS(\XC, \cY)}^2 \gamma(dy) < \infty$ 
  for almost every $x$.
  This suggests that $\gfun(x, \cdot) \in H^1_{\gamma}$ for almost every $x$, 
  with a Sobolev derivative equal to $\partial_{y} \gfun(x, \cdot)(I - \Qrx)$. 
  Applying the Poincar\'e inequality on separable Hilbert spaces gives
  \[ \int \| \gfun(x,y) - \bE_{y \sim \gamma}[\gfun(x,y)] \|_{\cY}^2 \gamma(dy) 
    \leq \int \| \partial_y \gfun(x,y) \|_{\HS(\XC, \cY)}^2 \gamma(dy) \]
  for almost every $x$.
  Integrating again over $x \sim \gamma$ gives
  \begin{align*} 
    \iint \| \fun(\Qrx x + (I - \Qrx)y) &- \bE_{\gamma}[\fun | \sigma (\pinv{\Vrx})](x) \|_{\cY}^2 
      \gamma(dy) \gamma(dx) \\
    & \leq \iint \| \deriv \fun(\Qrx x + (I - \Qrx) y) (I - \Qrx) \|_{\HS(\XC, \cY)}^2 \gamma(dy) \gamma(dx).
  \end{align*}
  Since 
  $\bE_{\gamma}[\fun | \sigma(\pinv{\Vrx})](x) = \bE_{\gamma}[\fun | \sigma (\pinv{\Vrx})](\Qrx x + (I - \Qrx)y)$
  and $\Qrx x + (I - \Qrx)y \sim \gamma$, we obtain the inequality \eqref{eq:subspace_poincare} for $\fun \in \cF \cC^{\infty}$. 

  We use a density argument to extend the result to arbitrary $\fun \in H^1_{\gamma}$.
  We first note that \eqref{eq:subspace_poincare} is equivalently written as 
  \[
    \|\fun - \bE_{\gamma}[\fun | \sigma(\pinv{\Vrx})]\|_{L^2_{\gamma}}
    \leq \| \derivXC \fun (I - \Qrx)\|_{L^2(\gamma, \HS(\XC, \cY))}.
  \]
  From here, consider an approximating sequence $\fun_n \in \cF \cC^{\infty}$ converging to $\fun$ in $H^1_{\gamma}$.
  An application of the triangle inequality gives
  \[
    \|\fun - \bE_{\gamma}[\fun | \sigma(\pinv{\Vrx})]\|_{L^2_{\gamma}}
    \leq \|\fun_n - \bE_{\gamma}[\fun_n | \sigma(\pinv{\Vrx})]\|_{L^2_{\gamma}} + \|\fun - \fun_n \|_{L^2_{\gamma}} + \|\bE_{\gamma}[\fun - \fun_n | \sigma(\pinv{\Vrx})] \|_{L^2_{\gamma}}.
  \]
  Applying the subspace Poincar\'e inequality for $\fun_n$ yields 
  \[
    \|\fun - \bE_{\gamma}[\fun | \sigma(\pinv{\Vrx})]\|_{L^2_{\gamma}}
    \leq \|\derivXC \fun_n (I - \Qrx)\|_{L^2(\gamma, \HS(\XC, \cY))} +  \|\fun - \fun_n \|_{L^2_{\gamma}} + \|\bE_{\gamma}[\fun - \fun_n | \sigma(\pinv{\Vrx})] \|_{L^2_{\gamma}}.
  \]
  Another application of the triangle inequality yields 
  \begin{align*}
    \|\fun - \bE_{\gamma}[\fun | \sigma(\pinv{\Vrx})]\|_{L^2_{\gamma}}
    \leq & \|\derivXC \fun (I - \Qrx)\|_{L^2(\gamma, \HS(\XC, \cY))} +
      \|(\derivXC \fun_n - \derivXC \fun) (I - \Qrx)\|_{L^2(\gamma, \HS(\XC, \cY))} \\
    & +\|\fun - \fun_n \|_{L^2_{\gamma}} + \|\bE_{\gamma}[\fun - \fun_n | \sigma(\pinv{\Vrx})] \|_{L^2_{\gamma}}. 
  \end{align*}
  Since $\fun_n \rightarrow \fun$ in $H^1_{\gamma}$, the latter three terms in involving the difference of $\fun_n$ and $\fun$ can be made arbitrarily small (see also \Cref{cor:conditional_expectation_preserves_norm}). This leads to the subspace Poincar\'e inequality for general $\fun \in H^1_{\gamma}$.
\end{proof}

\subsection{Proofs for the sampling error bounds}\label{sec:proof_sampling_error_bounds}  

\subsubsection{Proofs for the Monte Carlo errors} \label{sec:proof_monte_carlo_error_bounds}

Here, we present the proofs in the analyses of the sampling errors in \Cref{sec:sampling_error}.
We begin with the proofs for the Monte Carlo error bounds, \Cref{lemma:monte_carlo_error} and \Cref{lemma:fourth_moment_error}.
The proof of \Cref{lemma:monte_carlo_error} is reproduced below from \cite{BhattacharyaHosseiniKovachkiEtAl21}.

\begin{proof}(of \Cref{lemma:monte_carlo_error})
    Given the assumption that the second moment is bounded, we have 
    \begin{align*}
      \bE_N \left[\|\hat{z} - \bar{z}\|_{\cH}^2 \right]
        &= \bE_N \left[\linner \frac{1}{N}\sum_{k=1}^{N} z_k - \bar{z}, \frac{1}{N}\sum_{k=1}^{N} z_k - \bar{z} \rinner_{\cH} \right] \\
        &= \frac{1}{N} \bE_{z \sim \nu} [\|z\|_{\cH}^2] + \frac{N^2 - N}{N^2} \linner \bar{z}, \bar{z} \rinner_{\cH}
        - 2 \linner \bar{z}, \bar{z} \rinner_{\cH} + \linner \bar{z}, \bar{z} \rinner_{\cH} \\
        &= \frac{1}{N} \left( \bE_{z \sim \nu}\left[\|z\|_{\cH}^2 - \|\bar{z}\|_{\cH}^2 \right] \right) \\ 
        &= \frac{1}{N} \left( \bE_{z \sim \nu}\left[\|z - \bar{z}\|_{\cH}^2 \right] \right),
    \end{align*}
    where we have made use of the independence of $z_i, z_j$ for $i \neq j$. 
  \end{proof}

  We then proceed with the proof for the error in the fourth moments.
  \begin{proof}(of \Cref{lemma:fourth_moment_error})
    Assume first that $z$ is a centered random variable, i.e., $\bar{z} = 0$. 
    Then, 
    \[
      \bE_N [\|\hat{z} - \bar{z}\|_{\cH}^4] 
        = \frac{1}{N^4} \bE_{N} \left[  \linner \sum_{i=1}^{N} z_i , \sum_{j=1}^{N} z_j \rinner_{\cH} \linner \sum_{k=1}^{N} z_k , \sum_{l=1}^{N} z_l \rinner_{\cH} \right]
        = \frac{1}{N^4} \sum_{i,j,k,l=1}^{N} \bE_{N} \left[  \linner z_i , z_j \rinner_{\cH} \linner z_k , z_l \rinner_{\cH} \right].
    \]
    For $(z_i)_{i=1}^{N}$ that are i.i.d., the expectation is zero whenever a single index differs from the remaining three, e.g., $i \neq j, k, l$, since 
    \[
      \bE_{N} \left[  \linner z_i , z_j \rinner_{\cH} \linner z_k , z_l \rinner_{\cH} \right] 
      = \linner \bE_{N}[z_i], \bE_{N}[ z_j \linner z_k, z_l \rinner_{\cH}] \rinner_{\cH} 
      = \linner \bE_{z \sim \nu}[z], \bE_{N}[ z_j \linner z_k, z_l \rinner_{\cH}] \rinner_{\cH} 
      = 0.
    \]
    Thus, the only non-zero terms are those where all indices are equal, i.e., $i = j = k = l$, 
    and those where two indices are equal, e.g., $i = j \neq k = l$.
    The former contributes $N$ terms of the form  $\bE_{z \sim \nu}[\|z\|_{\cH}^4]$.
    The latter case corresponds to $\binom{N}{2} \cdot \binom{4}{2} = 3 N(N-1)$ terms, 
    where for every term,
    the Cauchy--Schwarz inequality implies that 
    \[
      \bE_{N} \left[  \linner z_i , z_j \rinner_{\cH} \linner z_k , z_l \rinner_{\cH} \right] 
      \leq \bE_{N} \left[ 
        \|z_i\|_{\cH}
        \|z_j\|_{\cH}
        \|z_k\|_{\cH}
        \|z_l\|_{\cH}
      \right] 
      =
      \bE_{N} \left[ 
        \|z_i\|_{\cH}^2
      \right]
      \bE_{N} \left[ 
        \|z_j\|_{\cH}^2
      \right]
      = 
      \bE_{z \sim \nu} \left[ 
        \|z\|_{\cH}^2
      \right]^2,
    \]
    where we have again made use of independence.
    Combining the two contributions, we arrive at
    \[
      \bE_{N}\left[\|\hat{z} - \bar{z}\|_{\cH}^4 \right]
      \leq \frac{N \bE_{z \sim \nu}[\|z\|_{\cH}^4] + 3 N (N-1) \bE_{z\sim\nu}[\|z\|_{\cH}^2]^2}{N^4}
      \leq \frac{K_z}{N^2},
    \]
    where it suffices to take  
    \[
      K_z = \bE_{z \sim \nu}[\|z\|_{\cH}^4] + 3 \bE_{z \sim \nu}[\|z\|_{\cH}^2]^2.
    \]
    When $z$ is uncentered, we note that 
    \[
      \hat{z} - \bar{z} = \frac{1}{N}\sum_{k=1}^{N} (z_k - \bar{z}) = \frac{1}{N} \sum_{i=k}^{N} \tilde{z}_k
    \]
    where $\tilde{z}_k = z_k - \bar{z}$ are centered i.i.d. random variables.
    Thus, applying the result for the centered case above to the random variable $\tilde{z}$ yields the desired result.
\end{proof}  

\subsubsection{Proofs for PCA and DIS sampling errors}
\label{sec:proof_pod_dis_sampling_errors}

The proofs for \Cref{prop:pod_sample_error_expectation} and \Cref{prop:dis_sampling_error_mean} then follow from the previous results \Cref{lemma:monte_carlo_error} and \Cref{lemma:fourth_moment_error}.

\begin{proof}(of \Cref{prop:pod_sample_error_expectation})
    Lemma B.2 in \cite{BhattacharyaHosseiniKovachkiEtAl21} considers the uncentered case. 
    To derive the result for the centered case, we decompose the covariance estimator $\Cyhat$ as
    \begin{align*} 
      \Cyhat 
        &= \frac{1}{N}\sum_{k=1}^{N}(\fun(x_k) - \fhat) \otimes (\fun(x_k) - \fhat) \\
        &= \frac{1}{N}\sum_{k=1}^{N}(\fun(x_k) - \fbar + \fbar - \fhat) \otimes (\fun(x_k) - \fbar + \fbar - \fhat) \\ 
        &= \frac{1}{N}\sum_{k=1}^{N}(\fun(x_k) - \fbar)  \otimes (\fun(x_k) - \fbar) + (\fbar - \fhat) \otimes (\fbar - \fhat). 
    \end{align*} 
    Let $\Cybar := \frac{1}{N}\sum_{k=1}^{N}(\fun(x_k) - \fbar)  \otimes (\fun(x_k) - \fbar)$, then 
    \begin{equation}\label{eq:split_cy_error}
    \| \Cy - \Cyhat \|_{\HS(\cY, \cY)}^2 \leq 2 \|\Cy - \Cybar \|_{\HS(\cY, \cY)}^2 + 2 \|(\fhat - \fbar) \otimes (\fhat - \fbar) \|_{\HS(\cY, \cY)}^2. 
    \end{equation}
  
    In particular, the first term reverts to the case of uncentered covariances $\|\Cy - \Cybar \|_{\HS(\cY, \cY)}$, which we present here for completeness.
    Since $\|y \otimes y\|_{\HS(\cY, \cY)} = \|y \|_{\cY}^2$ for $y \in \cY$, we have 
    \[\bE_{\gamma}[\|(\fun - \fbar) \otimes (\fun - \fbar)\|_{\HS(\cY, \cY)}^2] = \bE_{\gamma}[\|\fun - \fbar\|_{\cY}^4] < \infty. \]
    This allows us to apply \Cref{lemma:monte_carlo_error} using the random variable $z = (\fun - \fbar) \otimes (\fun - \fbar)$, 
    taking values in the Hilbert space $\HS(\cY, \cY)$, whose mean is $\Cy = \bE_{\gamma}[(\fun - \fbar) \otimes (\fun - \fbar)]$.
    This gives the bound
    \[ 
      \bE_N[\|\Cy - \Cybar\|_{\HS(\cY, \cY)}^2] \leq \frac{\bE_{\gamma}[\|(\fun - \fbar) \otimes (\fun - \fbar) - \Cy \|_{\HS(\cY, \cY)}^2]}{N}.
    \]
    For the second term in \eqref{eq:split_cy_error}, we have
    \[
      \|(\fhat - \fbar) \otimes (\fhat - \fbar) \|_{\HS(\cY, \cY)} = \|\fhat - \fbar\|_{\cY}^2
    \]
    such that 
    \[
      \bE_N[\|(\fhat - \fbar) \otimes (\fhat - \fbar) \|_{\HS(\cY, \cY)}^2] 
      = \bE_{N}[\|\fhat - \fbar\|_{\cY}^4]  
      \leq \frac{\|\fun - \fbar\|_{L^4_{\gamma}}^4 + 3\|\fun - \fbar\|_{L^2_{\gamma}}^4}{N^2}.
    \]
    Combining the bounds for the two terms from \eqref{eq:split_cy_error} yields 
    \[ \bE_{N}[ \|\Cy - \Cyhat\|_{\HS(\cY, \cY)}^2] \leq
    \frac{2 \bE_{\gamma}[\|(\fun - \fbar) \otimes (\fun - \fbar) - \Cy \|_{\HS(\cY, \cY)}^2]}{N}  + \frac{2 \|\fun - \fbar\|_{L^4_{\gamma}}^4 + 6\|\fun - \fbar\|_{L^2_{\gamma}}^4}{N^2}.
    \]
    Since $N^{-2}$ decays faster than $N^{-1}$, we can simply take the constant as
    \[ 
    M_{\Cy} =  2 \bE_{\gamma}[\|(\fun - \fbar) \otimes (\fun - \fbar) - \Cy \|_{\HS(\cY, \cY)}^2] + 2 \|\fun - \fbar\|_{L^4_{\gamma}}^4 + 6\|\fun - \fbar\|_{L^2_{\gamma}}^4 . \]

\end{proof}

The strategy for the sampling error estimates in the DIS case is similar (and in fact, simpler, since we do not need to account for the centering by the empirical mean).

\begin{proof}(of \Cref{prop:dis_sampling_error_mean})
    Recall that for $\Hx = \bE_{\gamma}[\derivXC \fun^* \derivXC \fun]$ and $\Hy = \bE_{\gamma}[\derivXC \fun \derivXC \fun^*]$, 
    their empirical estimators
    \[ \Hxhat = \frac{1}{N} \sum_{k=1}^{N} \derivXC \fun(x_k)^* \derivXC \fun(x_k), \qquad \Hyhat = \frac{1}{N} \sum_{k=1}^{N} \derivXC \fun(x_k) \derivXC \fun(x_k)^*, \]
    are sample averages of the random variables $\derivXC \fun^* \derivXC \fun$ and $\derivXC \fun \derivXC \fun^*$ that take values in the Hilbert spaces $\HS(\XC, \XC)$ and $\HS(\cY, \cY)$, respectively.
    Note that for Hilbert--Schmidt operators $A \in \HS(\cH_1, \cH_2)$ and $B \in \HS(\cH_2, \cH_3)$, which map between separable Hilbert spaces $\cH_1, \cH_2, \cH_3$, the norm of their composition satisfies $\|AB\|_{\HS(\cH_1, \cH_3)} \leq \|A\|_{\HS(\cH_1, \cH_2)} \|B\|_{\HS(\cH_2, \cH_3)}$. Thus, we have
    \[ 
      \|\derivXC \fun^* \derivXC \fun\|_{\HS(\XC, \XC)} \leq \|\derivXC \fun\|_{\HS(\XC, \cY)}^2, \qquad
      \|\derivXC \fun \derivXC \fun^*\|_{\HS(\cY, \cY)} \leq \|\derivXC \fun\|_{\HS(\XC, \cY)}^2.
    \]
    Squaring and taking expectations gives us
    \[ \bE_{\gamma}[\|\derivXC \fun^* \derivXC \fun\|_{\HS(\XC, \XC)}^2] \leq \bE[ \|\derivXC \fun\|_{\HS(\XC, \cY)}^4] < \infty \] 
    and 
    \[ \bE_{\gamma}[\|\derivXC \fun \derivXC \fun^*\|_{\HS(\cY, \cY)}^2] \leq \bE[ \|\derivXC \fun\|_{\HS(\XC, \cY)}^4] < \infty \]
    by our moment assumption. This allows us to apply \Cref{lemma:monte_carlo_error}, which yields the desired result with constants
    \[ 
      M_{\Hx} = \bE_{\gamma}[\|\derivXC \fun^* \derivXC \fun - \Hx\|_{\HS(\XC, \XC)}^2], \qquad 
      M_{\Hy} = \bE_{\gamma}[\|\derivXC \fun \derivXC \fun^* - \Hy\|_{\HS(\cY, \cY)}^2].
    \]
\end{proof}

\subsection{Proof of the main result}
\label{sec:proof_main_theorem}
We now conclude with the proof of the main theorem, \Cref{theorem:main_theorem}, which combines all the results from the previous sections.

\begin{proof}
  First note that, 
  under the assumption that $\fun \in H^{2}_{\gamma} \cap W^{1,4}_{\gamma}$,
  for almost every realization of the samples 
  $(x_k)_{k=1}^{N}, x_k \sim \gamma$ i.i.d., (i.e., almost surely with respect to $\gamma^{N}$),
  the empirical estimator $\fhat \in \cY$,
  the empirical covariance $\Cyhat$ and sensitivity operator $\Hyhat$ lead to orthonormal bases $(\uhat_i)_{i=1}^{\infty}$ in $\cY$,
  and the empirical sensitivity operator $\Hxhat$ leads to an orthonormal basis $(\vhat_i^{\JTJ})_{i=1}^{\infty}$ in $\XC$.
  Additionally, 
  \Cref{assumption:derivative_second_moments} ensures 
  that $\vhat_i^{\JTJ}$ are elements of $\cX$.
  We can therefore apply \Cref{theorem:error_decomposition_with_conditional_expectation} to obtain the existence of a neural network $\varphi_{\psi, \theta}$ such
  \eqref{eq:error_decomposition_with_conditional_expectation_l2} and \eqref{eq:error_decomposition_with_conditional_expectation_h1} 
  hold for $\Uryhat = \Uryhat^{\POD}, \Uryhat^{\JTJ}$ and $\Vrxhat = \Vrx^{\KLE}, \Vrxhat^{\JTJ}$.

  Taking expectations $\bE_{N}$ for the terms in \eqref{eq:error_decomposition_with_conditional_expectation_l2} and \eqref{eq:error_decomposition_with_conditional_expectation_h1}, 
  we have 
  \[\bE_{N}[\|\fbar - \fhat\|_{\cY}^2] \leq \frac{\|\fun - \fbar\|_{L^2_{\gamma}}^2}{N}
  \]
  due to \Cref{corollary:mean_estimator_error},
  while the terms 
  \[
    2 \bE_{N} [ \|\fun - \fhat_r \|_{L^2_{\gamma}}^2], \quad 
    2 \bE_{N} [ \|\Pryhat (\derivXC \fun - \derivXC \fhat_r)\Qrxhat \|_{\HS(\XC,\cY)}^2], 
    \quad 
    \bE_{N} [\|\derivXC \fun (I - \Qrxhat)\|_{\HS(\XC,\cY)}^2]
  \]
  depend on the input dimension reduction $\Vrxhat$ and
  \[
    \bE_{N}[\|(I - \Pryhat)(\fun - \fbar)\|_{L^2_{\gamma}}^2], \quad
    \bE_{N}[\|(I - \Pryhat)(\derivXC \fun - \derivXC \fbar)\|_{L^2_{\gamma}}^2], 
  \]
  depend on the output dimension reduction $\Uryhat$.
  We can then consider each dimension reduction strategy separately, substituting in the results from the reconstruction, ridge function, and sampling errors developed in \Cref{sec:dibasis_approximation_error} to obtain the final approximation errors.

  \paragraph{Empirical output PCA} $\Uryhat = \Uryhat^{\POD}$:
  For the output reconstruction error, \Cref{prop:pod_reconstruction_bound} gives 
  \[
    \|(I - \Pryhat^{\POD}) (\fun - \fbar) \|_{L^2_{\gamma}}^2 
    \leq \sum_{i = \ry + 1}^{\infty} \lambda_i^{\POD}  + \min \left\{ \sqrt{2 r} \|\Cy - \Cyhat \|_{\HS(\cY, \cY)}, \frac{2 \| \Cy - \Cyhat\|_{\HS(\cY,\cY)}^2}{\lambda_{r}^{\POD} - \lambda_{r+1}^{\POD}} \right\}.
  \]
  To bound the derivative reconstruction error, we additionally require \Cref{assumption:derivative_inverse_inequality}, which combined with \Cref{prop:pod_derivative_reconstruction_multiplicative} gives
  \begin{align*}
    & \|(I - \Pryhat^{\POD}) \derivXC \fun \|_{L^2(\gamma, \HS(\XC,\cY))}^2 \\
    & \qquad \leq K_{D} \left(
    \sum_{i = \ry + 1}^{\infty} \lambda_i^{\POD} + \min \left\{ \sqrt{2 r} \|\Cy - \Cyhat \|_{\HS(\cY, \cY)}, \frac{2 \| \Cy - \Cyhat\|_{\HS(\cY,\cY)}^2}{\lambda_{r}^{\POD} - \lambda_{r+1}^{\POD}} \right\} \right).
  \end{align*}
  Taking expectations with respect to the sampling error with the help of \Cref{prop:pod_sample_error_expectation} gives 
  \[
    \bE_{N}\left( 
    \min \left\{ \sqrt{2 r} \|\Cy - \Cyhat \|_{\HS(\cY, \cY)}, \frac{2 \| \Cy - \Cyhat\|_{\HS(\cY,\cY)}^2}{\lambda_{r}^{\POD} - \lambda_{r+1}^{\POD}} \right\} 
    \right)
    \leq 
    \min \left\{
      \sqrt{\frac{2r M_{\Cy}}{N}}, \frac{2 M_{\Cy}}{(\lambda_{r}^{\POD} - \lambda_{r+1}^{\POD}) N}
    \right\} 
  \]
  with $M_{\Cy} = 2 \bE_{\gamma}[\|(\fun - \fbar) \otimes (\fun - \fbar) - \Cy \|_{\HS(\cY, \cY)}^2] + 2 \|\fun - \fbar\|_{L^4_{\gamma}}^4 + 6\|\fun - \fbar\|_{L^2_{\gamma}}^4$. 
  Thus, the overall constants are 
  $K_{\cY}^{(1)} = 1$, $K_{\cY}^{(2)} = K_D$, and 

  \paragraph{Empirical output DIS} $\Uryhat = \Uryhat^{\JJT}$: 
  From \Cref{prop:dis_derivative_reconstruction} and \Cref{prop:output_dis_function_error}, we have 
  \begin{align*}
    \|(I - \Pryhat^{\JJT})(\fun - \fbar)\|_{L^2_{\gamma}}^2 
    &\leq \sum_{i = \ry + 1}^{\infty} \lambda_i^{\JJT} 
    + \min \left\{ \sqrt{2 r} \|\Hy - \Hyhat \|_{\HS(\cY, \cY)}, \frac{2 \| \Hy - \Hyhat\|_{\HS(\cY,\cY)}^2}{\lambda_{r}^{\JJT} - \lambda_{r+1}^{\JJT}} \right\},
    \\
    \|(I - \Pryhat^{\JJT})\derivXC \fun\|_{L^2(\gamma, \HS(\XC,\cY))}^2
    &\leq \sum_{i = \ry + 1}^{\infty} \lambda_i^{\JJT} 
    + \min \left\{ \sqrt{2 r} \|\Hy - \Hyhat \|_{\HS(\cY, \cY)}, \frac{2 \| \Hy - \Hyhat\|_{\HS(\cY,\cY)}^2}{\lambda_{r}^{\JJT} - \lambda_{r+1}^{\JJT}} \right\}.
  \end{align*}
  Taking expectations with respect to the sampling error using \Cref{prop:dis_sampling_error_mean} gives
  \[
    \bE_{N} \left[ \min \left\{ \sqrt{2 r} \|\Hy - \Hyhat \|_{\HS(\cY, \cY)}, \frac{2 \| \Hy - \Hyhat\|_{\HS(\cY,\cY)}^2}{\lambda_{r+1}^{\JJT} - \lambda_{r}^{\JJT}} \right\} \right] \leq \min \left\{ 
      \sqrt{\frac{2r M_{\Hy}}{N}}, \frac{2 M_{\Hy}}{(\lambda_{r}^{\JJT} - \lambda_{r+1}^{\JJT})N} \right\},
  \]
  where $M_{\Hy} = \bE_{\gamma}[\|\derivXC \fun \derivXC \fun^* - \Hy \|_{\HS(\cY,\cY)}^2]$. Thus, the overall constants are 
  $K_{\cY}^{(1)} = K_{\cY}^{(2)} = 1$ and
  $ M_{\cY} = \bE_{\gamma}[\|\derivXC \fun \derivXC \fun^* - \Hy \|_{\HS(\cY,\cY)}^2].$

  \paragraph{Exact input PCA} $\Vrxhat = \Vrx^{\KLE}$: 
    Under \Cref{assumption:derivative_second_moments}, 
    the results in
    \Cref{prop:kle_derivative_reconstruction},
    \Cref{cor:ridge_function_error},
    and \Cref{prop:ridge_derivative_bound_kle} 
    give
    \begin{align*}
        \|\derivXC \fun(I - \Qrx^{\KLE}) 
        \|_{L^2(\gamma, \HS(\XC,\cY))}^2 
        & \leq 
        \bE_{\gamma}\left[ \|\deriv \fun \|_{\cL(\cX, \cY)}^2 \right]\sum_{i = \rx + 1}^{\infty} \mu_i^{\KLE},
        \\
        \| \fun - \fhat_r \|_{L^2_{\gamma}}^2
        & \leq 
        \bE_{\gamma}\left[ \|\deriv \fun \|_{\cL(\cX, \cY)}^2 \right]\sum_{i = \rx + 1}^{\infty} \mu_i^{\KLE}, 
        \\
        \|\Pryhat (\derivXC \fun - \derivXC \fhat_r)\Qrx^{\KLE} \|_{L^2(\gamma, \HS(\XC,\cY))}^2
        & \leq  
        \bE_{\gamma} \left[\|\deriv^{2}\fun\|_{\cL_2(\cX, \cY)}^2 \right] 
        \bE_{\gamma}[\|x\|_{\cX}^2] 
        \sum_{i = \rx + 1}^{\infty} \mu_i^{\KLE}.
    \end{align*}
    Moreover, since there is no sampling error in the exact PCA, these bounds are deterministic, such that $M_{\cX} = 0$ and 
    \[
      K_{\cX}^{(1)}  = 2 \bE_{\gamma}[\| \deriv \fun \|_{\cL(\cX,\cY)}^2],
      \quad
      K_{\cX}^{(2)}  = 2 \bE_{\gamma}[\|\deriv^2 \fun\|_{\cL_2(\cX,\cY)}^2]
      \bE_{\gamma}[\|x\|_{\cX}^2] + 
      \bE_{\gamma}[\| \deriv \fun \|_{\cL(\cX,\cY)}^2].
    \]

    \paragraph{Empirical input DIS} $\Vrxhat = \Vrxhat^{\JTJ}$: 
    \Cref{prop:dis_derivative_reconstruction} and \Cref{cor:ridge_function_error}
    give 
    \begin{align*}
      \| \fun - \fhat_r \|_{L^2_{\gamma}}^2 
      &\leq 
      \sum_{i = \rx + 1}^{\infty} \mu_i^{\JTJ} 
      + \min \left\{ \sqrt{2 r} \|\Hx - \Hxhat \|_{\HS(\XC, \XC)}, \frac{2 \| \Hx - \Hxhat\|_{\HS(\XC,\XC)}^2}{\mu_{r}^{\JTJ} - \mu_{r+1}^{\JTJ}} \right\},  \\
      \| \derivXC \fun(I - \Qrxhat^{\JTJ})\|_{L^2(\gamma, \HS(\XC,\cY))}^2 
      & \leq
      \sum_{i = \rx + 1}^{\infty} \mu_i^{\JTJ} 
      + \min \left\{ \sqrt{2 r} \|\Hx - \Hxhat \|_{\HS(\XC, \XC)}, \frac{2 \| \Hx - \Hxhat\|_{\HS(\XC,\XC)}^2}{\mu_{r}^{\JTJ} - \mu_{r+1}^{\JTJ}} \right\}, 
    \end{align*}
    To bound the derivative ridge function error, we additionally need \Cref{assumption:hessian_inverse_inequality}, which combined with \Cref{prop:ridge_derivative_bound_dis} gives
    \begin{align*}
      & \|\Pryhat(\derivXC \fun - \derivXC \fhat_r) \Qrxhat^{\JTJ}\|_{L^2_{\gamma}(\HS(\XC, \cY))}^2 \\
      & \qquad \leq 
      K_H \left( \sum_{i = \rx + 1}^{\infty} \mu_i^{\JTJ} +\min \left\{ \sqrt{2 r} \|\Hx - \Hxhat \|_{\HS(\XC, \XC)}, \frac{2 \| \Hx - \Hxhat\|_{\HS(\XC,\XC)}^2}{\mu_{r}^{\JTJ} - \mu_{r+1}^{\JTJ}} \right\}\right).
    \end{align*}
    Taking expectations with respect to the sampling distribution
    and applying \Cref{prop:dis_sampling_error_mean} then yields
    \[
    \bE_{N} \left[ \min \left\{ \sqrt{2 r} \|\Hx - \Hxhat \|_{\HS(\XC, \XC)}, \frac{2 \| \Hx - \Hxhat\|_{\HS(\XC,\XC)}^2}{\mu_{r}^{\JTJ} - \mu_{r+1}^{\JTJ}} \right\} \right] \leq \min \left\{ 
      \sqrt{\frac{2r M_{\Hx}}{N}}, \frac{2 M_{\Hx}}{(\mu_{r}^{\JTJ} - \mu_{r+1}^{\JTJ})N} \right\},
    \]
    where $M_{\Hx} = \bE_{\gamma}[\|\derivXC \fun^* \derivXC \fun - \Hx\|_{\HS(\XC,\XC)}^2]$.
    Thus, the overall constants are $K_{\cX}^{(1)} = 2$, $K_{\cX}^{(2)} = 2 K_H + 1$, and $M_{\cX} = \bE_{\gamma}[\|\derivXC \fun^* \derivXC \fun - \Hx\|_{\HS(\XC,\XC)}^2]$.

    As suggested in \Cref{remark:on_the_assumptions}, we have used \Cref{assumption:derivative_second_moments} to ensure the DIS yields $C^{-1}_{\cX} \vhat_i^{\JTJ}$ that are elements of $\cX$. 
    We do not need the full \Cref{assumption:derivative_second_moments} in the input DIS result. Instead, we can relax the assumption by directly assuming that $\fun$ is such that the eigenvalue problem involving $\Hxhat$ yields $C^{-1}_{\cX} \vhat_i^{\JTJ}$ that are indeed elements of $\cX$.
\end{proof}

\bookmarksetup{startatroot}

\bibliography{ref}



\end{document}